\documentclass[12pt]{amsart}

\usepackage[dvips, dvipdfmx]{graphicx}
\usepackage[abs]{overpic}
\usepackage{amssymb}
\usepackage{amsmath}
\usepackage{amsthm}
\usepackage{mathrsfs}
\usepackage{fancybox}
\usepackage{comment}
\usepackage{layout}
\usepackage{color}

\theoremstyle{plain}
\newtheorem{theorem}{Theorem}[section]
\newtheorem{lemma}[theorem]{Lemma}

\newtheorem{proposition}[theorem]{Proposition}
\newtheorem{conjecture}[theorem]{Conjecture}

\theoremstyle{definition}
\newtheorem{definition}[theorem]{Definition}

\newtheorem{remark}[theorem]{Remark}

\graphicspath{{figures/},{figures(mizusawa)/}}

\title[Clasper presentations of Habegger-Lin's action on string links]{Clasper presentations of Habegger-Lin's action on string links}
\author[Kotorii]{Yuka Kotorii}
\address[Kotorii]{Mathematics Program, Graduate School of Advanced Science and Engineering, Hiroshima University, 1-7-1 Kagamiyama Higashi-hiroshima City, Hiroshima 739-8521 Japan}
\address[Kotorii]{International Institute for Sustainability with Knotted Chiral Meta Matter (SKCM$^2$), Hiroshima University, 1-7-1 Kagamiyama Higashi-hiroshima City, Hiroshima 739-8521 Japan}
\address[Kotorii]{Interdisciplinary Theoretical and Mathematical Sciences Program, RIKEN, 2-1, Hirosawa, Wako, Saitama 351-0198, Japan}
\email{kotorii@hiroshima-u.ac.jp}

\author[Mizusawa]{Atsuhiko Mizusawa}
\address[Mizusawa]{Faculty of Science and Engineering, Waseda University, 3-4-1 Okubo, Shinjuku-ku, Tokyo 169-8555, Japan} 
\email{a\_mizusawa@aoni.waseda.jp}
\date{\today}

\begin{document}

\begin{abstract}
Habegger and Lin gave a classification of the link-homotopy classes of links as the link-homotopy classes of string links modulo the actions of conjugations and partial conjugations for string links. 
%After that Hughes showed that the partial conjugations generate the conjugations. 
In this paper, we calculated the actions of the partial conjugations and the conjugations explicitly for 4- and 5-component string links which gave classifications (presentations) of the link-homotopy classes of 4- and 5-component links. As an application, we can run Habegger and Lin's algorithm which determines whether given two links are link-homotopic or not for 4- and 5-component links.
%Two links are link-homotopic if they are transformed to each other by a sequence of self-crossing changes and ambient isotopies. The link-homotopy classes of 4-component links were classified by Levine with great calculous. We modify the results by using Habiro's clasper theory. The new classificaton gives more symmetrical and schematical points of view to the link-homotopy classes of 4-component links. We also gives some new subsets of the link-homotopy classes of 4-component links which are classified by invariants. 
\end{abstract}


%\address{Mathematics Program, Graduate School of Advanced Science and Engineering, Hiroshima University, 1-7-1 Kagamiyama Higashi-hiroshima City, Hiroshima 739-8521 Japan}
%\address{International Institute for Sustainability with Knotted Chiral Meta Matter (SKCM$^2$), Hiroshima University, 1-7-1 Kagamiyama Higashi-hiroshima City, Hiroshima 739-8521 Japan}
%\address{Interdisciplinary Theoretical and Mathematical Sciences Program, RIKEN, 2-1, Hirosawa, Wako, Saitama 351-0198, Japan}
%\email{kotorii@hiroshima-u.ac.jp}

\maketitle
\section{Introduction} \label{introduction}

%Link-homotopy, milnor invariant 
Milnor \cite{Mil} introduced a concept of \textit{link-homotopy} on links, which is a restricted homotopy (he called it homotopy) and a weaker equivalence relation than isotopy.
The link-homotopy can be defined an equivalence relation generated by ambient isotopies and self-crossing changes.
Here the self-crossing change is a crossing change between the same components.  
Milnor gave a complete classification of link-homotopy classes of links with 3- or fewer components by numerical invariants derived from the reduced link group. The reduced link group is a link-homotopy invariant, which is a certain quotient of the fundamental group of the complement of a link.
Furthermore he \cite{Mil2} made the family of numerical invariants, called $\overline{\mu}$-invariants, which is the extension of the numerical invariants in \cite{Mil}.  

%Levine
On the other hand, Levine \cite{Le} showed that the Milnor invariant is not enough to classify 4-component links completely.
He \cite{Le2} improved indeterminacy of $\overline{\mu}$-invariants and then classified 4-component links.
%H-L
In order to classify more general cases, Habegger and Lin \cite{HL} defined a concept of the string link.
They showed that the set of link-homotopy classes of string links has a group structure, which is the same as the group of link-homotopy classes of pure braids. 
They also represented the group as a subgroup of the automorphism group of the reduced free group. 
Then, they defined Milnor's $\mu$-invariants for string links and classified the link-homotopy classes of string links completely by the $\mu$-invariants.
Furthermore they gave the Markov-type theorem for link-homotopy classes of string links. That is, the closures of two string links are link-homotopic if and only if these two can be related by a finite sequence of string links such that any two adjacent string links are either \textit{conjugation} or \textit{partial conjugation}.
%Hughes showed that the partial conjugation generate the conjugation in \cite{H}.
They gave an algorithm which decides whether the closures of two string links are link-homotopic or not and therefore gave an algorithm which decides whether two links are link-homotopic or not by using the actions of partial conjugations and the conjugations.
On the other hand, the orbit space of the group action on the group of link-homotopy classes of string links generated by the partial conjugations and the conjugations is not represented exactly and we can not run the algorithm actually.
After that, Hughes \cite{H} showed that the partial conjugations generate the conjugations. Now it is enough to consider the partial conjugations only. 
\par
In some cases, the link-homotopy classes are described by using the \textit{claspers} introduced by Habiro \cite{Ha} and Gusarov \cite{G} independently. 
For 2- and 3-component links, the standard forms of link-homotopy classes are given by using Hopf chords and Borromean chords (i.e. $C_1$ and $C_2$-trees) in \cite{TY}. For the link-homotopy classes of string links, Yasuhara \cite{Y} gave a canonical form using claspers and showed the correspondence between the Milnor invariants and claspers (see also \cite{MY2}). We also remark the works \cite{MN, KM} on HL-homotopy for handle-body links by using the clasper theory.
\if0
\par
In the previous work \cite{KM2}, the authors generalize the standard form using the claspers in \cite{TY} to the 4-component case and gave a classification of the link-homotopy classes of 4-component links other than Levine's one \cite{Le2}. However, in the proof of the result, we used Levine's result and it did not give an alternative proof of Levine's results. 
\fi
\par
%前a?論文のe¨?a??
For 4-component link-homotopy classes, Levine \cite{Le2} represented a link as commutator numbers and gave its ambiguity by calculating the automorphism group of the reduced group of the link. He then gave a classification of links. On the other hand, in our previous paper \cite{KM2}, we represented a link as the numbers of claspers in the clasper presentation and gave its ambiguity by calculating transformation between claspers. We then gave a geometric classification of links. 
 However, in the proof of the result, we used Levine's result and it did not give an alternative proof of Levine's result.
 \par
 The classification in the previous work \cite{KM2} was applied to Habegger and Lin's algorithm since the calculations of the numbers of claspers correspond to the actions of the partial conjugations. We succeeded in running Habegger and Lin's algorithm for the 4-component case.

%======
\if0
\par
In the link-homotopy classes, clasper theory is especially useful, since we have a canonical form on string link in \cite{Y}. 
We note that work \cite{MN, KM} on HL-homotopy by using clasper theory.

In some cases, the link-homotopy classes are described by using the claspers defined in \cite{Ha}. For 3-component links, the link-homotopy classes are described by Hopf chords and Borromean chords (i.e. $C_1$ and $C_2$ claspers) in \cite{TY}. For the link-homotopy classes of string links, the correspondence between the Milnor invariants and claspers are shown in \cite{Y}. 
\fi
%======

\par
In the present paper, we apply the clasper theory to 4- and 5-component string links. We calculate the actions of the partial conjugations for string-links and explicitly classify 4- and 5-component links. 
%, see Theorems \ref{Rep4-compLink} and \ref{Rep5-compLink} respectively. 

%The following is the 4-component case. The 5-component case is in section 3.

\begin{theorem}[Theorems \ref{Rep4-compLink} and \ref{Rep5-compLink}] \label{mainthm}
Let $\mathcal{L}_n$ be the set of the link-homotopy classes of $n$-component links. For 4-component links, 
$$\mathcal{L}_4=\mathbb{Z}^{12}/{X_4},$$
where $\mathbb{Z}^{12}$ is the set of ordered tuples of 12 integers
$$(y_{12}, y_{13}, y_{14}, y_{23}, y_{24}, y_{34}, y_{123}, y_{124}, y_{134}, y_{234}, y_{1234}, y_{1324})$$
and $X_4$ is the relations generated by the following relations in Table \ref{ModAct4-compSF},
$$\{\overline{x}'_{12},\overline{x}'_{13}, \overline{x}'_{21}, \overline{x}'_{23}, \overline{x}'_{31}\overline{x}'_{32}, \overline{x}'_{41}, \overline{x}'_{42} \}.$$
%is the set of relations in Table \ref{ModAct4-compSF}.
For 5-component links, 
$$\mathcal{L}_5=\mathbb{Z}^{36}/{X_5}$$
where $\mathbb{Z}^{36}$ is the set of ordered tuples of 36 integers
$$(y_{12}, y_{13}, y_{14}, y_{15}, y_{23}, y_{24}, y_{25}, y_{34}, y_{35}, y_{45},$$
$$ y_{123}, y_{124}, y_{134}, y_{125}, y_{135}, y_{145}, y_{234}, y_{235}, y_{245}, y_{345}, $$
$$y_{1234}, y_{1324}, y_{1235}, y_{1245}, y_{1325}, y_{1345}, y_{1425}, y_{1435}, y_{2345}, y_{2435}, $$
$$y_{12345}, y_{12435}, y_{13245}, y_{13425}, y_{14235}, y_{14325})$$
and $X_5$ is the relations generated by the following relations in Table \ref{ModAct4-compSF},
$$\{\overline{x}'_{12}, \overline{x}'_{13}, \overline{x}'_{14}, 
\overline{x}'_{21}, \overline{x}'_{23}, \overline{x}'_{24}, 
\overline{x}'_{31}, \overline{x}'_{32}, \overline{x}'_{34}, 
\overline{x}'_{41}, \overline{x}'_{42}, \overline{x}'_{43},
\overline{x}'_{51}, \overline{x}'_{52}, \overline{x}'_{53}
\}.$$
%is the set of relations in Table \ref{ModAct5-compSF}.
\end{theorem}

\if0
\begin{table}[htb] 
\begin{center}
   \caption{The number of claspers of 4-component links} \label{originaltable}
  \begin{tabular}{|c|c|c|c|c|c|c|} \hline
      & $f_1$ & $f_2$ & $f_3$ & $f_4$ & $t_1$ & $t_2$ \\ \hline 
    $\psi_{21}$ & 0 & 0 & $c_5$ & $-c_1$ & $f_1$ & $0$ \\ \hline 
    $\psi_{41}$ & 0 & $c_6$ & $-c_5$ & $0$ & $0$ & $-f_1$ \\ \hline 
    $\psi_{12}$ & 0 & 0 & $-c_4$ & $c_2$ & $f_2$ & $0$ \\ \hline 
    $\psi_{32}$ & $c_6$ & 0 & $0$ & $-c_2$ & $0$ & $-f_2$ \\ \hline 
    $\psi_{43}$ & $c_5$ & $-c_4$ & $0$ & $0$ & $f_3$ & $0$ \\ \hline 
    $\psi_{23}$ & $-c_5$ & 0 & $0$ & $c_3$ & $0$ & $-f_3$ \\ \hline 
    $\psi_{34}$ & $-c_1$ & $c_2$ & $0$ & $0$ & $f_4$ & $0$ \\ \hline 
    $\psi_{14}$ & 0 & $-c_2$ & $c_3$ & $0$ & $0$ & $-f_4$ \\ \hline 
  \end{tabular} \\
\end{center}
\end{table}
\fi

%\begin{remark} \label{intro-remark2}

%\end{remark}

\begin{remark} \label{Rem-for-Levine's-result}
The presentation of $\mathcal{L}_4$ in Theorem \ref{mainthm} coincides with the one in \cite{KM2}. Thus the calculations in this paper give an alternative proof of the Levine's result. 
\end{remark}

\begin{remark}
The $\overline{\mu}$-invariants for links are several quotients of the $\mu$-invariants for string links. However, the $\overline{\mu}$-invariants are not enough to classify the link-homotopy classes of links. 
So there is a question what is the appropriate quotient for links. 
In \cite{Y, MY2}, the relations between the Milnor's $\mu$-invariants and the number of claspers of the canonical form are given for string links (see Subsection \ref{canonicalform}). 
Theorem \ref{mainthm} determines the quotient (the indeterminacy) of $\overline{\mu}$-invariants for the 5-component case.
\end{remark}

\begin{remark} \label{intro-remark3}
Graff \cite{Gra} showed the presentation in Theorem \ref{mainthm} for 4-component case independently. (We also announced the result for 4-component links in \cite{KM3}.) He also gave a presentation for algebraically split 5-component links. 
\end{remark}
\par 
The calculations for Theorem \ref{mainthm} give the actions of the partial conjugations for the 4- and 5-component cases and they enable us to run Habegger and Lin's algorithm which determines whether given two link are link-homotopic or not. This recovers the algorithm in \cite{KM2} for the 4-component case. 
\par
The organization of this paper is as follows. In Section 2, we review the clasper theory and define a canonical form of string links. In Section 3, we explicitly represent the orbit space of the group action on the group of link-homotopy classes of string links generated by the partial conjugations by using clasper presentations. 
%[[Moreover we made an invariants.]] 
In Section 4, we give an algorithm which decides whether given two 4- or 5-component links are link-homotopic or not in terms of clasper calculations. 
In Section 5, we test the results of the calculations in Section 3 by using the actions of the conjugations for the canonical forms. 

%%%%%%%%%%%%%%%%%%%%%%%%%%%%%%
%イントロ(小鳥a±?
%string linkの定義(小鳥a±?
%Hagebber-LinのMarkov type theorem(小鳥a±?
%clasper(小鳥a±?
%string link /lhの標準形
%4,5成a?のpartial conjugation の計cR?水澤)
%4,5成a?でのアルゴリズa??(水澤)
%conjugation がpartial conjugation で生a?されることの図を使った別証a??水澤)
%Algorithm をstring link のMilnor invariant を使った形に書き換える (水澤)
%%%%%%%%%%%%%%%%%%%%%%%%%%%%%%

%%%%%%%%%%%%%%%%
\section*{Acknowledgement}
The authors thank Professor Akira Yasuhara for helpful comments.
Y.K. is supported by JSPS KAKENHI, Grant-in-Aid for Early-Career Scientists Grant Number 20K14322, and by RIKEN iTHEMS Program and World Premier International Research Center Initiative (WPI) program, International Institute for Sustainability with Knotted Chiral Meta Matter (SKCM$^2$).

%%%%%%%%%%%%%%%%
\section{Preparations} \label{preparations}
%%%%%%%%%%%%%%%%
\subsection{Markov-type theorem for link-homotopy classes of string links}

In this section, we introduce Habegger and Lin's results in \cite{HL} for link-homotopy classes of string links.  
The string links were introduced to study link-homotopy classification of links.
They gave a Markov-type theorem for the link-homotopy classes of links by using string links.  

Let $p_1, \cdots , p_n$ be points lying in order on the $x$-axis on the interior of the unit disk $D^2$. 
An {\it $n$-string link} $\sigma=\sigma_1 \cup \cdots \cup \sigma_n$ is a proper embedding of $n$ disjoint unit intervals $I_1, \cdots , I_n$ into $D^2 \times [0,1]$ such that for each $i=1, \cdots, n$, $\sigma_i(0)=(p_i, 0)$ and $\sigma_i(1)=(p_i, 1)$, where $\sigma_i$ is called the $i$-th string of the string link $\sigma$. 
Each string of a string link inherits an orientation from the usual orientation of the interval.
Figure \ref{stringlink} left shows a planar projection of a 4-string link. 
Composition of  $n$-string links is defined as follows.
Let $\sigma=\sigma_1 \cup \cdots \cup \sigma_n$ and $\sigma'=\sigma'_1 \cup  \cdots \cup \sigma'_n$ be two string links. 
Then the {\it composition}  $\sigma\sigma'= (\sigma\sigma')_1 \cup  \cdots \cup (\sigma\sigma')_n$ of  $\sigma$ and  $\sigma'$ is the string link defined by $(\sigma\sigma')_i =h_1(\sigma_i) \cup h_2(\sigma'_i)$ for each $i=1, \cdots, n$, where $h_1,h_2: D^2 \times [0,1] \rightarrow D^2 \times [0,1]$ are embeddings defined by 
\[ h_1(p,t)=(p,\frac{1}{2}t) \text{ and } h_2(p,t)=(p,\frac{1}{2}+\frac{1}{2}t)  \]
for any $p \in D^2$ and $t \in [0,1]$ (see Figure \ref{stringlink} right). 
The {\it trivial} $n$-string link $1_n$ consists of $\sigma (I_i)={p_i} \times [0,1]$ for each $i=1, \cdots, n$.  
It is known that for fixed $n$, the set of link-homotopy classes of $n$-string links forms a group with multiplication induced by the composition of string links and the link-homotopy class $[1_n]$ of $1_n$ as the identity.  
We denote it by $\mathscr{H}(n)$.
Here the link-homotopy on string links is generated by ambient isotopies relative to endpoints of each strings and self-crossing changes. 

\begin{figure}[h]
$$
\raisebox{-22 pt}{\begin{overpic}[width=170pt]{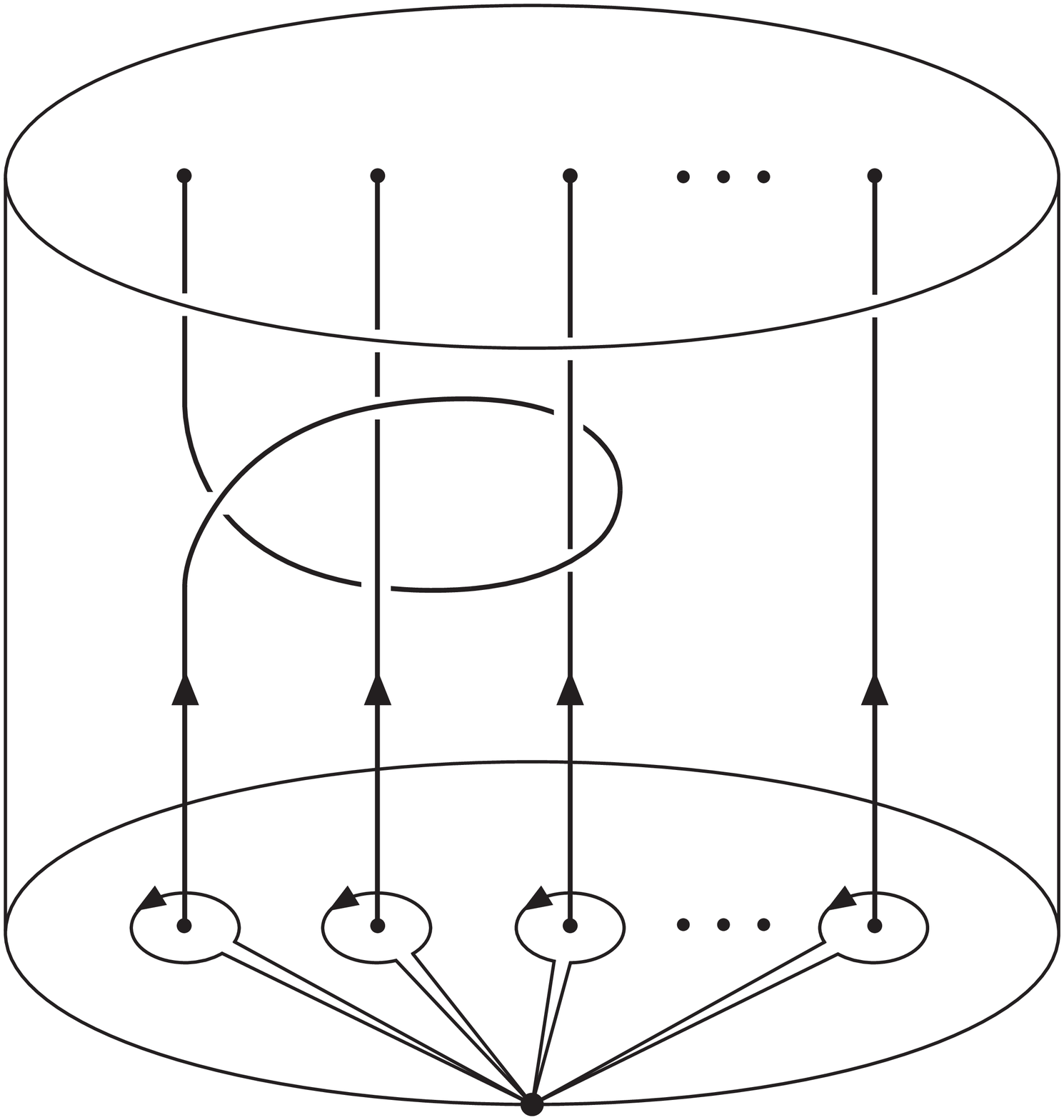}
\put(57,-11){base point} 
\put(11,38){$x_1$} \put(42,39){$x_2$} 
\put(74,39){$x_3$} \put(123,39){$x_n$}
\put(24,157){$p_1$} \put(55,157){$p_2$} 
\put(87,157){$p_3$} \put(136,157){$p_n$} 
\put(-43,173){$D^2\times [0,1]$} 
\end{overpic}} 
\hspace{2.0cm}
\raisebox{-0 pt}{\begin{overpic}[width=130pt]{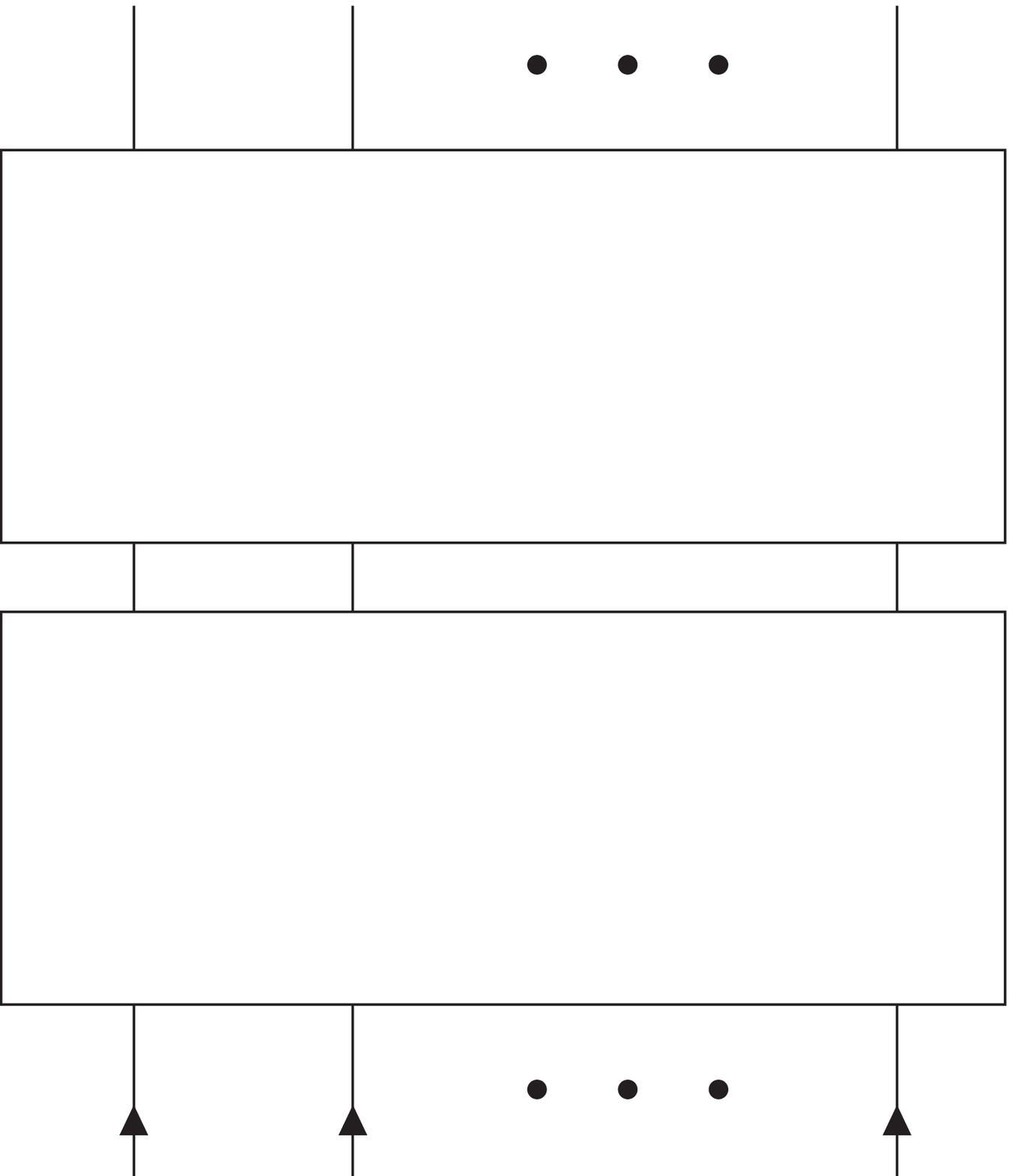}
\put(62,45){\Large $\sigma$} \put(62,103){\Large $\sigma'$}
\put(14,-13){$1$} \put(43,-13){$2$} 
\put(113,-13){$n$} 
\end{overpic}} 
$$ 
\if0
$$
\raisebox{-20 pt}{\begin{overpic}[width=150pt]{string-link-example02-2-2.eps}
\put(-1,25){$x_1$} \put(38,25){$x_2$} 
\put(78,25){$x_3$} \put(116,25){$x_4$}
\put(14,-13){$1$} \put(52.5,-13){$2$} 
\put(92,-13){$3$} \put(131,-13){$4$} 
\end{overpic}} 
\hspace{1.8cm}
\raisebox{-38 pt}{\begin{overpic}[width=130pt]{string-link-prod.eps}
\put(62,45){\Large $\sigma$} \put(62,103){\Large $\sigma'$}
\put(14,-13){$1$} \put(43,-13){$2$} 
\put(113,-13){$n$} 
\end{overpic}} 
$$ 
\fi
\vspace{0.0cm}
\caption{An $n$-string link and a composition.}\label{stringlink}
\end{figure}

\if0
\begin{figure}[h]
$$
\raisebox{-20 pt}{\begin{overpic}[width=150pt]{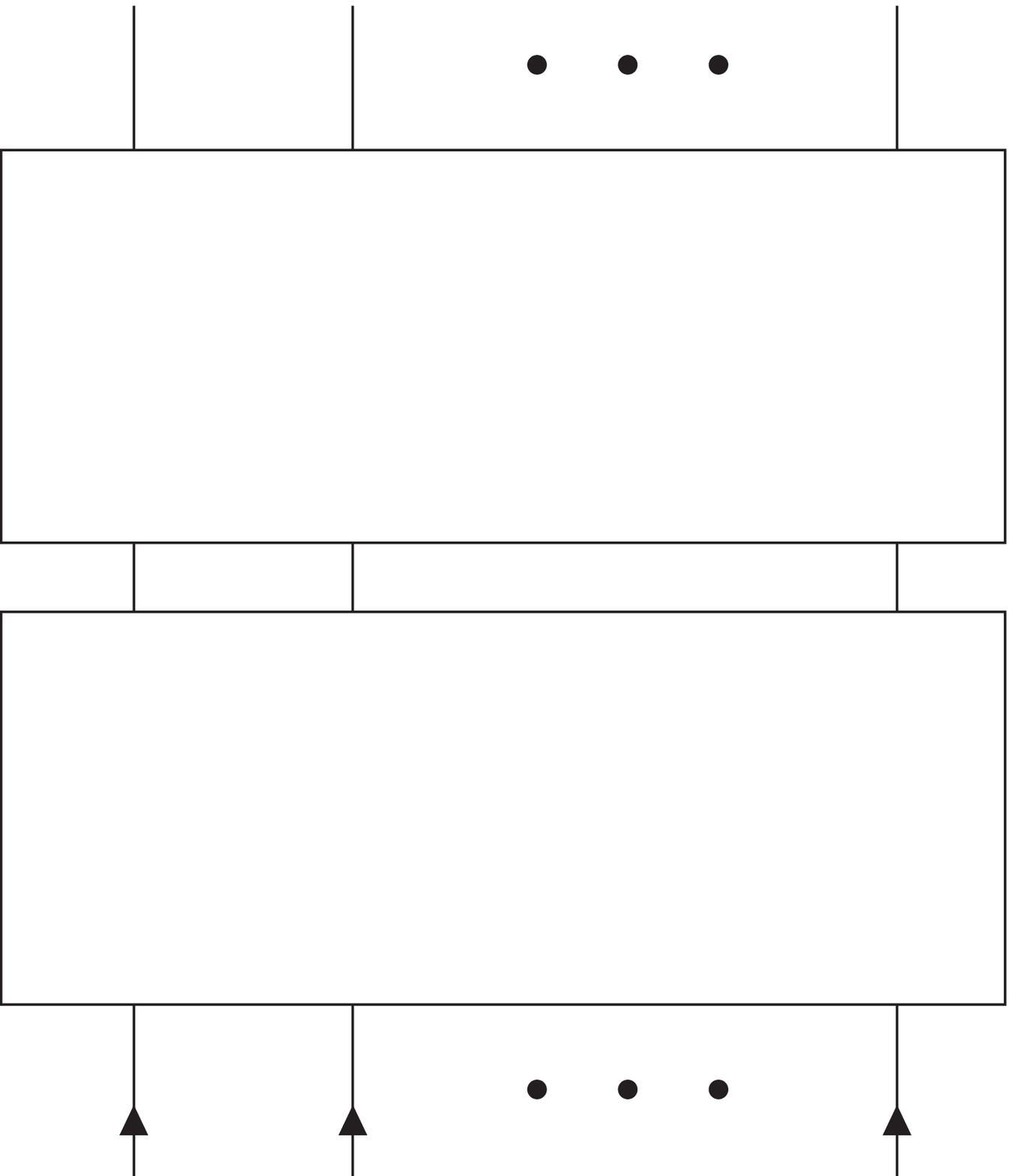}
\end{overpic}} 
$$ 
\caption{A composition.}\label{composition}
\end{figure}
\fi

Let $F(n)$ be a free group generated by $x_1, \cdots , x_n$ corresponding to the loops in Figure \ref{stringlink} left and $RF(n)$ the quotient group of $F(n)$ obtained by adding relations that each $x_i$ commutes with all of its conjugations.
It is known that for each $i$, any $\sigma \in \mathscr{H}(n)$ can be decomposed as $\theta_i g_i$ corresponding to the decomposition $\mathscr{H}(n) = \mathscr{H}(n-1) \ltimes RF_i(n-1)$ obtained from the map $\mathscr{H}(n) \rightarrow \mathscr{H}(n-1)$ given by omission of the $i$-th string, where $RF_i(n-1)$ is generated by the generators except for $x_i$. Figure \ref{decomposition} is an example. 
%gの決め方は??
\begin{figure}[h]
$$
\raisebox{-20 pt}{\begin{overpic}[width=100pt]{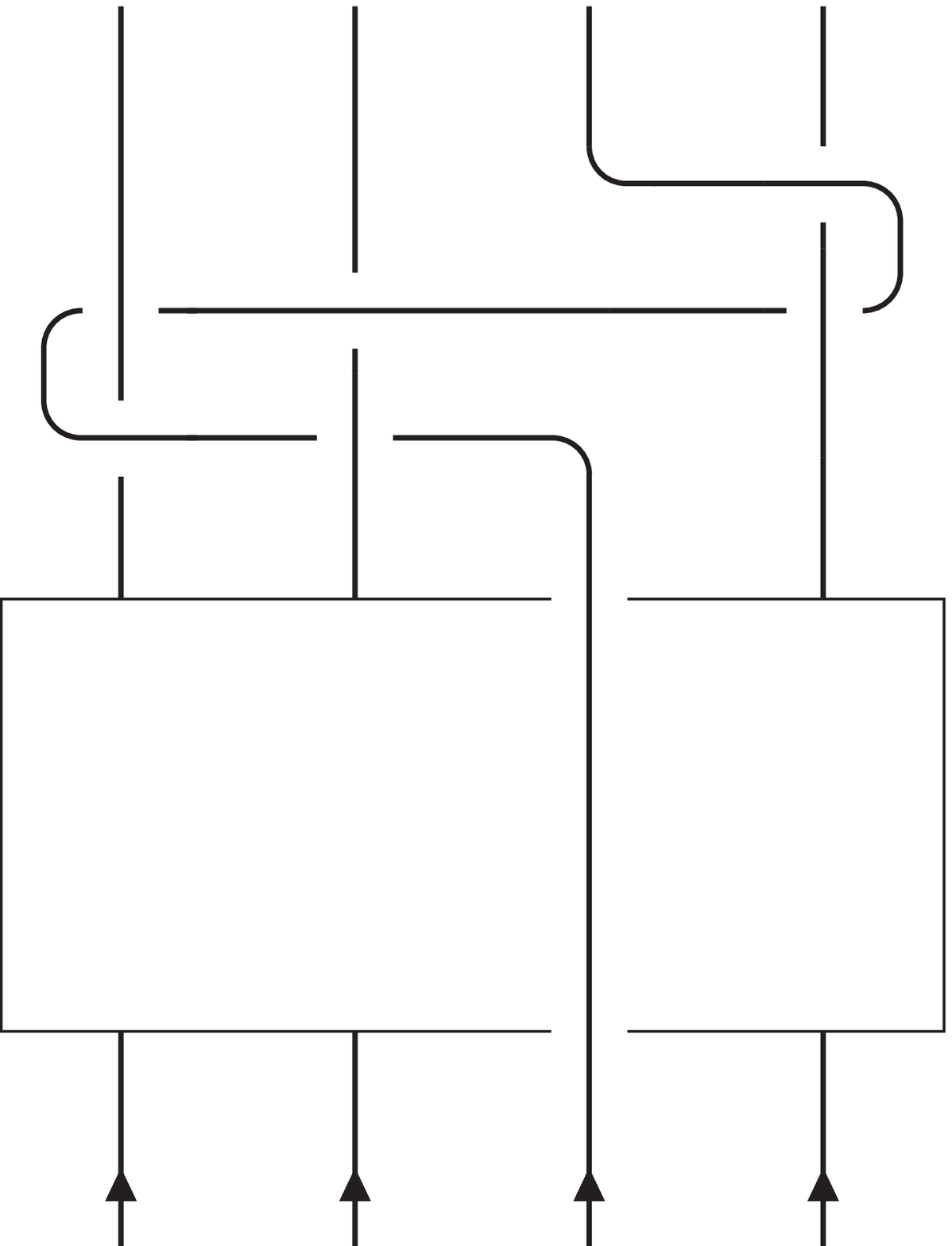}
\put(103,98){\large $g_3=x_2x_1^{-1}x_4^{-1}$} \put(40,41){\large $\theta_3$}
\put(10,-13){$1$} \put(35,-13){$2$} 
\put(60,-13){$3$} \put(84,-13){$4$} 
\end{overpic}} 
$$ 
\if0
$$
\raisebox{-20 pt}{\begin{overpic}[width=120pt]{string-link-decomposition03.eps}
\put(122,98){\Large $g_i$} \put(40,41){\Large $\theta_i$}
\put(10,-13){$1$} \put(29,-13){$i\!-\!1$} 
\put(58,-13){$i$} \put(71,-13){$i\!+\!1$} 
\put(103,-13){$n$} 
\end{overpic}} 
$$ 
\fi
\vspace{0.0cm}
\caption{A decomposition.}\label{decomposition}
\end{figure}

%\begin{theorem}
%string link a??!?定理
%\end{theorem}

\begin{definition}
Let $\sigma \in \mathscr{H}(n)$ be decomposed as $\theta_i g_i$.
A {\it partial conjugation} of $\sigma$ is an element of $\mathscr{H}(n)$ of the form $\theta_i g_i^h$ with $h \in RF_i(k-1)$, where $g_i^h=hg_ih^{-1}$. 
\end{definition}

Habegger and Lin gave the following Markov-type theorem for $\mathscr{H}(n)$, which is an improved version by Hughes's claim \cite{H} that the partial conjugations generate the conjugations.

\begin{theorem}[\cite{HL}, \cite{H}]\label{Markov-type}
Let $\sigma$ and $\sigma'$ be in $\mathscr{H}(n)$. Then the closures of $\sigma$ and $\sigma'$ are link-homotopic if and only if there is a finite sequence $\sigma=\sigma_0, \sigma_1, \cdots, \sigma_m=\sigma'$ 
of elements of $\mathscr{H}(n)$ such that $\sigma_{k}$ is a partial conjugation of $\sigma_{k-1}$ for each $k=1, \cdots, m$. 
\end{theorem}

The partial conjugation is induced by the following group action on $\mathscr{H}(n)$.
The action $\Sigma \cdot \sigma$ of $\Sigma \in \mathscr{H}(2n)$ on $\sigma \in \mathscr{H}(n)$ is defined as illustrated in Figure \ref{action}.
Here we forget the orientation of $\Sigma$ once and give an orientation compatible with that of $\sigma$. 

\begin{figure}[h]
$$
\raisebox{10 pt}{\begin{overpic}[width=180pt]{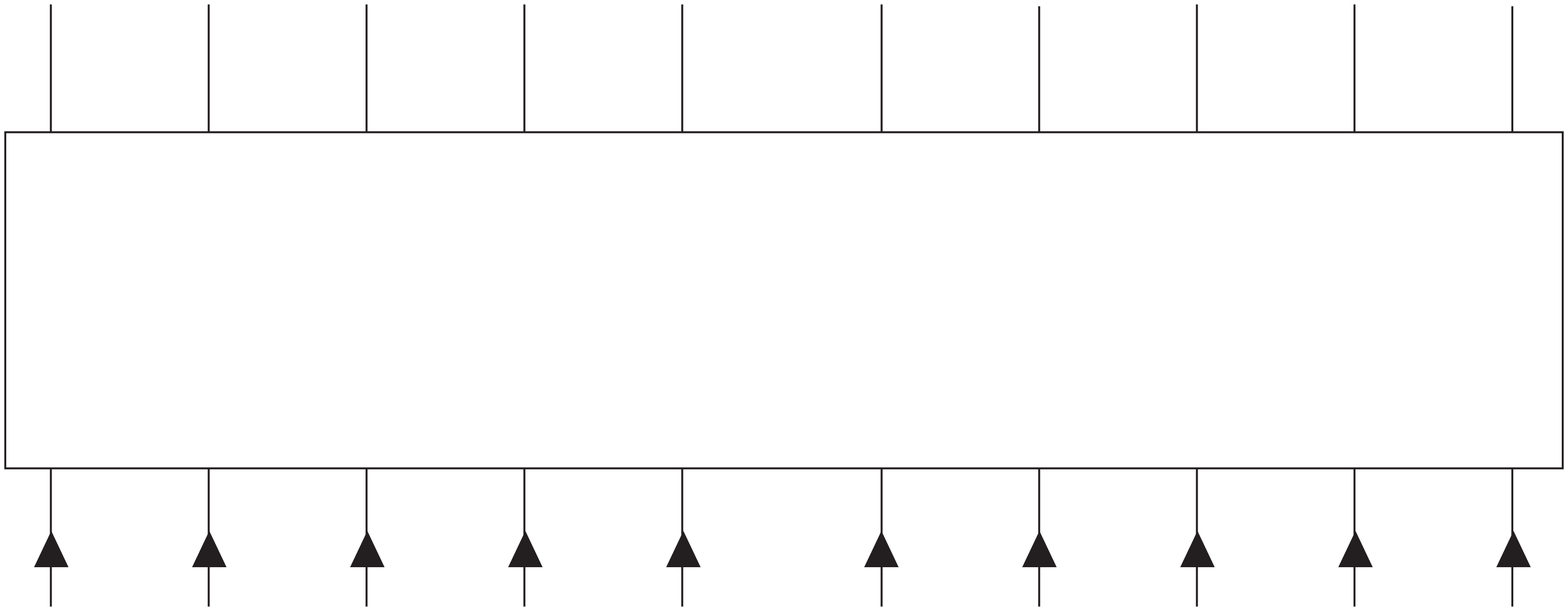}
\put(84,30){\large$\Sigma$}
\put(75,-13){$\overline{1}$} \put(57,-13){$\overline{2}$} 
\put(25,-11){$\cdots$} \put(2,-11){$\overline{n}$}
\put(98,-11){$1$} \put(116,-12){$2$} 
\put(139,-11){$\cdots$} \put(171,-10){$n$}
\end{overpic}} 
\hspace{2.0cm}
\raisebox{-20 pt}{\begin{overpic}[width=180pt]{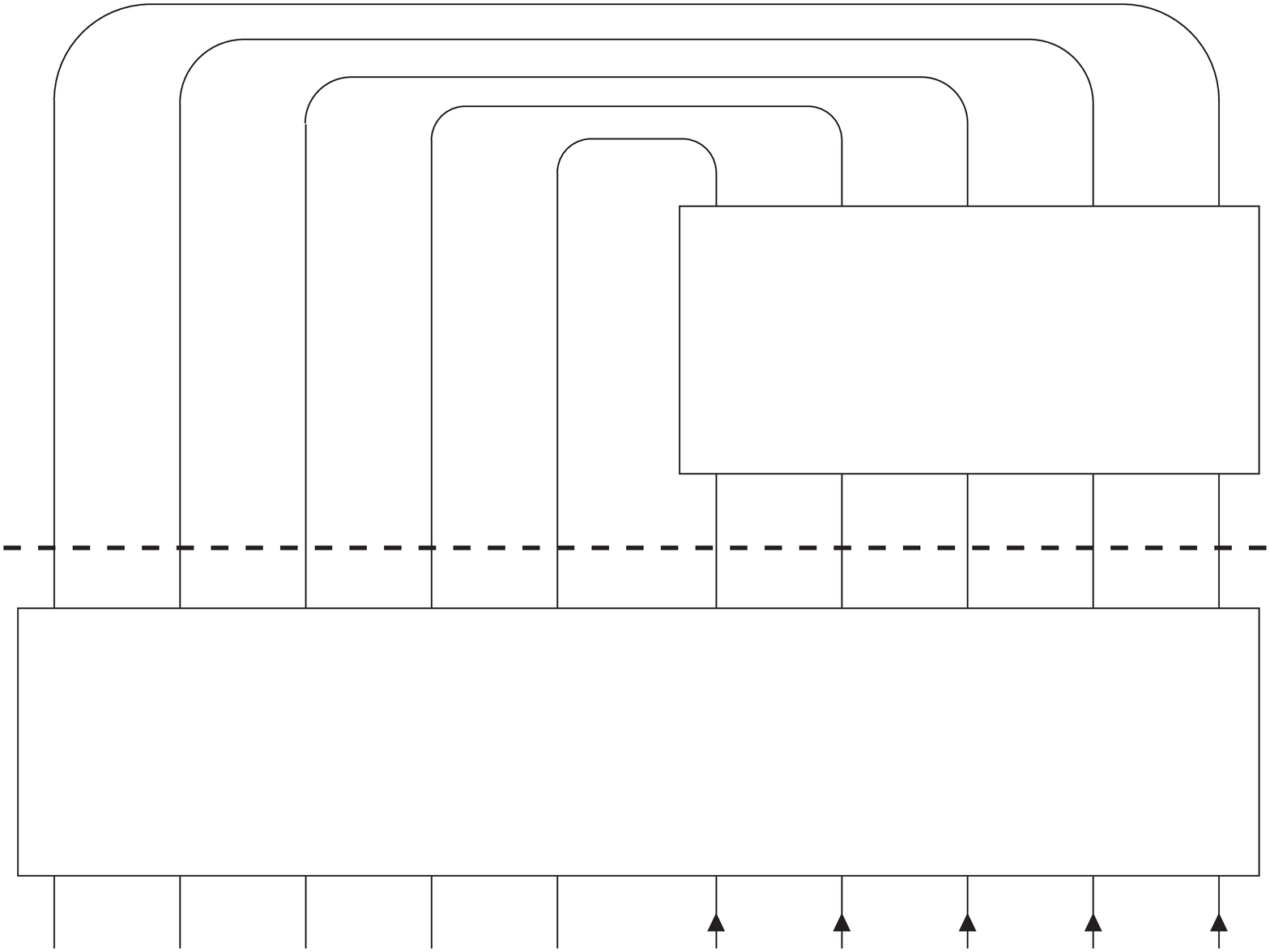}
\put(133,83){\large$\sigma$} \put(84,24){\large$\Sigma$}
\put(98,-12){$1$} \put(130,-11){$\cdots$} \put(169,-11){$n$}
\end{overpic}} 
$$ 
\vspace{0.1cm}
\caption{An action of $\mathscr{H}(2n)$ on $\mathscr{H}(n)$.}\label{action}
\end{figure}

Let $(\overline{x}_j,\overline{x}_j)_i$ ($1 \leq i \neq j \leq n$) be a $2n$-string link as illustrated in Figure \ref{partialconj}. The left figure is the case $j<i$ and the right figure is the case $i<j$.

\begin{figure}[h]
\vspace{0.4cm}
$$
\raisebox{-20 pt}{\begin{overpic}[width=183pt]{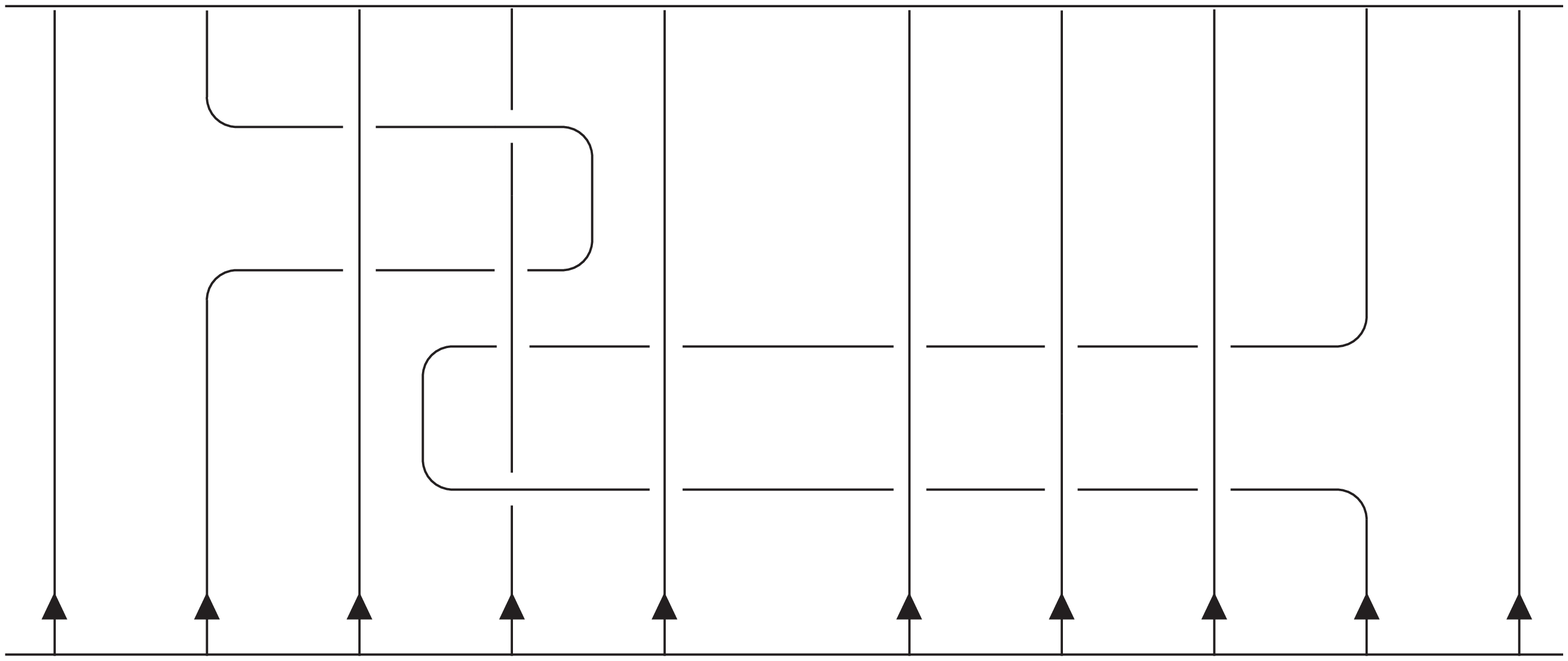}
\put(-2,-13){$\cdots$} \put(21,-14){$\overline{i}$} \put(34,-13){$\cdots$}\put(57,-14){$\overline{j}$} \put(70,-13){$\cdots$}
\put(97,-13){$\cdots$} \put(121,-14){$j$} \put(135,-13){$\cdots$} \put(157,-14){$i$} \put(170,-13){$\cdots$}
\end{overpic}} 
\hspace{1.3cm}
\raisebox{-20 pt}{\begin{overpic}[width=183pt]{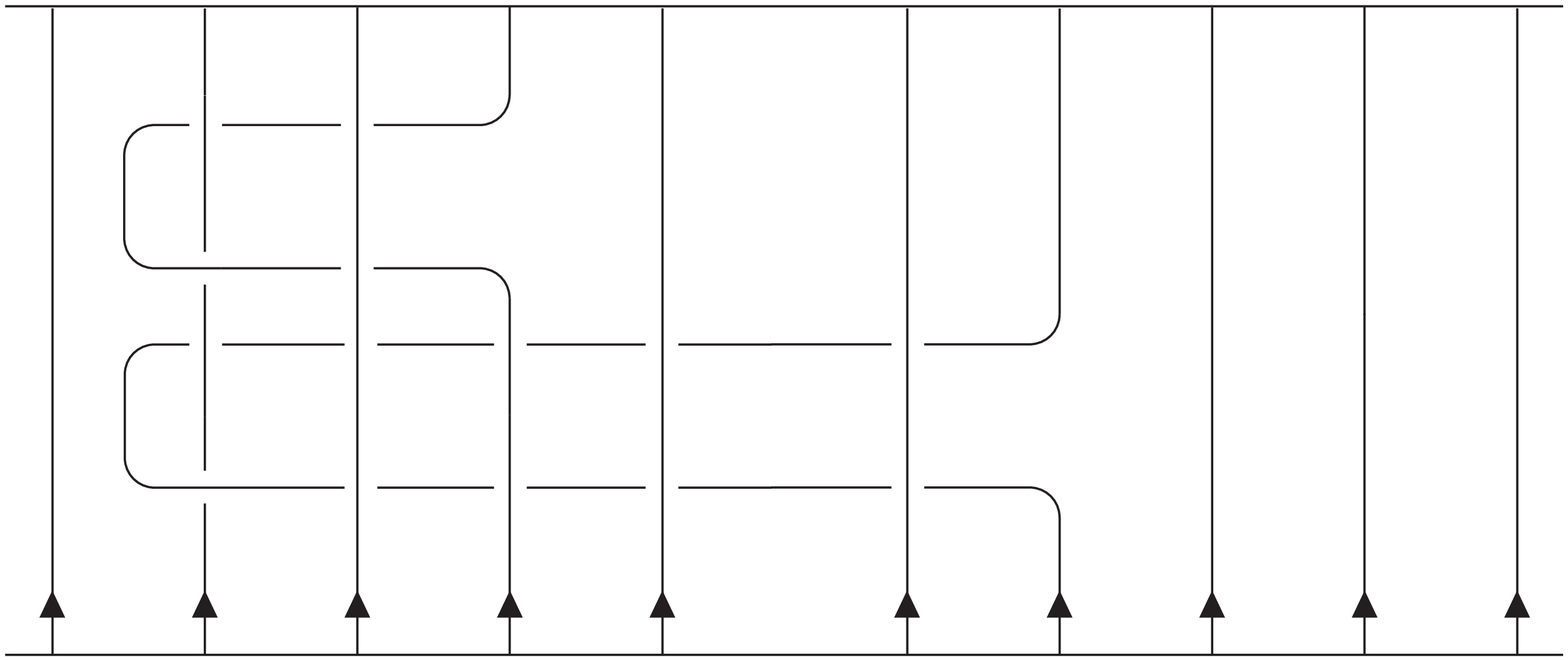}
\put(-2,-13){$\cdots$} \put(21,-14){$\overline{j}$} \put(34,-13){$\cdots$}\put(57,-14){$\overline{i}$} \put(70,-13){$\cdots$}
\put(97,-13){$\cdots$} \put(121,-14){$i$} \put(135,-13){$\cdots$} \put(157,-14){$j$} \put(170,-13){$\cdots$}
\end{overpic}} 
$$ 
\vspace{0.0cm}
%\caption{Generators $(\overline{x}_j,\overline{x}_j)_i$ and $(\overline{x}_i,\overline{x}_i)_j$ of partial conjugations.}\label{partialconj}
\caption{Generators $(\overline{x}_i,\overline{x}_i)_j$ of partial conjugations.
}\label{partialconj}
\end{figure}

\if0
\begin{figure}[h]
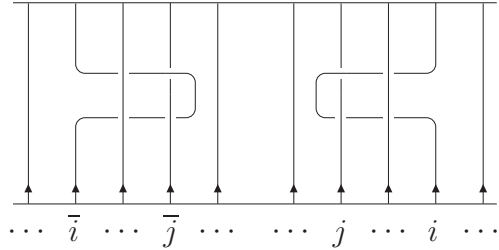

$$
\raisebox{-20 pt}{\begin{overpic}[width=180pt]{H2-elements-06.eps}
\put(-6,-13){$\cdots$} \put(17,-14){$\overline{j}$} \put(31,-13){$\cdots$} \put(55,-14){$\overline{i}$} \put(68,-13){$\cdots$}
\put(97,-13){$\cdots$} \put(121,-14){$i$} \put(135,-13){$\cdots$} \put(157,-14){$j$} \put(170,-13){$\cdots$}
\end{overpic}} 
\hspace{1.5cm}
\raisebox{-20 pt}{\begin{overpic}[width=180pt]{H2-elements-08.eps}
\put(-6,-13){$\cdots$} \put(17,-14){$\overline{i}$} \put(31,-13){$\cdots$} \put(55,-14){$\overline{j}$} \put(68,-13){$\cdots$}
\put(97,-13){$\cdots$} \put(121,-14){$j$} \put(135,-13){$\cdots$} \put(157,-14){$i$} \put(170,-13){$\cdots$}
\end{overpic}} 
$$ 
\vspace{0.0cm}
\caption{A generator $(x_j,x_j)_i$ of conjugations.}\label{partialconj2}
%\caption{An element $(x_j,x_j)_i$ of $\mathscr{H}(2n)$.}\label{partialconj2}
\end{figure}
\fi

The action of a subgroup generated by $(\overline{x}_j, \overline{x}_j)_i$ ($1 \leq i \neq j \leq n$) on $\mathscr{H}(n)$ induces the partial conjugations.
In fact, for a decomposed string link $\theta_i g_i$, we have $(\overline{x}_j,\overline{x}_j)_i \cdot \theta_i g_i  = \theta_i g_i^{x_j}$. Therefore for any $(\overline{x}_j,\overline{x}_j)_i$ and string link $\sigma$, the action give a partial conjugation of $\sigma$ and it is clear that any partial conjugation $\theta_i g_i^h$ of $\theta_i g_i$ is induced by $(\overline{x}_j,\overline{x}_j)_i$ for $j=1, \cdots i-1, i+1, \cdots, n$.
For the use in Section \ref{test}, we also introduce generators $(\overline{x}_j,x_j)_i$ of conjugations in $\mathscr{H}(2n)$ as illustrated in Figure \ref{conjugation}. 

%[[conjを載せるなら文での説明が必要]]
\begin{figure}[h]
$$
\raisebox{-20 pt}{\begin{overpic}[width=183pt]{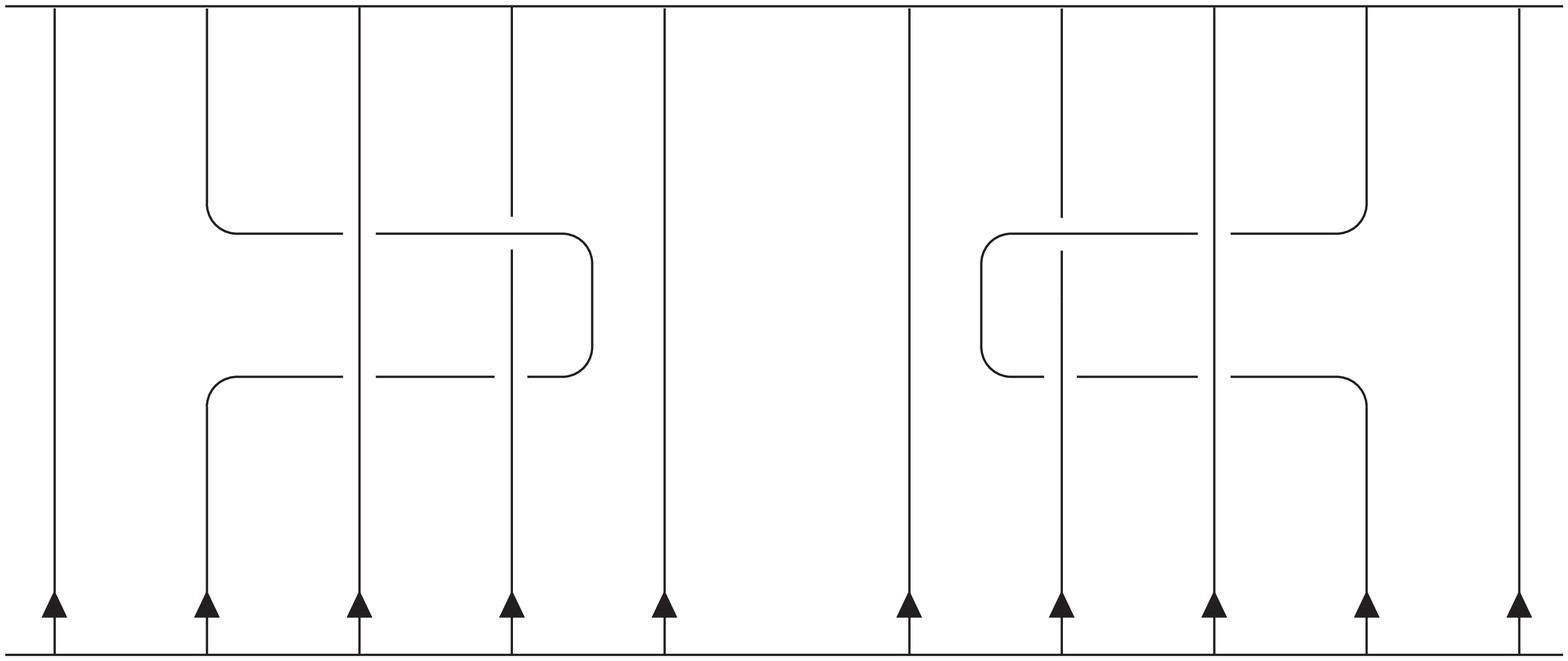}
\put(-2,-13){$\cdots$} \put(21,-14){$\overline{i}$} \put(34,-13){$\cdots$}\put(57,-14){$\overline{j}$} \put(70,-13){$\cdots$}
\put(98,-13){$\cdots$} \put(121,-14){$j$} \put(135,-13){$\cdots$} \put(157,-14){$i$} \put(170,-13){$\cdots$}
\end{overpic}}
%\if0
%\hspace{1.5cm}
%\raisebox{-20 pt}{\begin{overpic}[width=180pt]{
%H2-elements-04.eps}
%\put(-6,-13){$\cdots$} \put(17,-14){$\overline{i}$} \put(31,-13){$\cdots$} \put(55,-14){$\overline{j}$} \put(68,-13){$\cdots$}
%\put(97,-13){$\cdots$} \put(121,-14){$j$} \put(135,-13){$\cdots$} \put(157,-14){$i$} \put(170,-13){$\cdots$}
%\end{overpic}} 
%\fi
$$ 
\vspace{0.0cm}
\caption{A generator $(\overline{x}_j,x_j)_i$ of conjugations.}\label{conjugation}
%\caption{An element $(\overline{x}_j,x_j)_i$ of $\mathscr{H}(2n)$.}\label{conjugation}
\end{figure}

%%%%%%%%%%%%%%%%
\subsection{Claspers} 
\par
In this section, we introduce the clasper theory. For general definitions of claspers, we refer to Habiro \cite{Ha}. 

\begin{definition}\label{clasper}(\cite{Ha})
An embedded disk $T$ in $S^3$ ($D^2 \times [0,1]$) is said to be a {\em simple tree clasper} 
for a (string) link $L$ if it satisfies the following three conditions:
\begin{enumerate} 
\item The embedded disk $T$ is decomposed into bands and disks as follows. Each band connects two distinct disks and each disk attaches either one or three bands, where a disk attached at exactly one band is called a {\em leaf}. 
\item The embedded disk $T$ intersects the (string) link $L$ transversely so that each intersection is contained in the interior of the leaf. 
\item Each leaf intersects the (string) link $L$ at exactly one point.
\end{enumerate}
In this paper, we call a simple tree clasper with $k+1$ leaves a \emph{$C_k$-tree}. For simplicity, we express disks and bands as shown in Figure \ref{diskband} and call the band an {\em edge}. Furthermore, we express a positive (or negative) half-twist of a band by a circle with plus (or minus, respectively) on the edge, see Figure \ref{diskband}.

\begin{figure}[h]

$$\raisebox{-20 pt}{\begin{overpic}[width=40
pt]{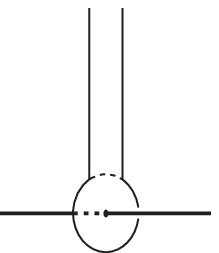}
\end{overpic}}
  \hspace{0.2cm} \mbox{\large$\longrightarrow$}
\hspace{0.2cm}
\raisebox{-20 pt}{\begin{overpic}[width=40
pt]{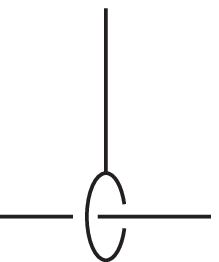}
\end{overpic}}\,,
\hspace{1.0cm}
\raisebox{-15 pt}{\begin{overpic}[width=50
pt]{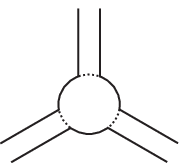}
\end{overpic}}
  \hspace{0.2cm} \mbox{\large$\longrightarrow$}
\hspace{0.2cm}
\raisebox{-15 pt}{\begin{overpic}[width=50
pt]{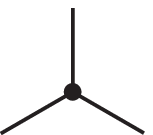}
\end{overpic}}\,
\hspace{1.0cm}
\raisebox{-20 pt}{\begin{overpic}[height=50
pt]{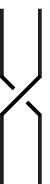}
\end{overpic}}
  \hspace{0.2cm} \mbox{\large$\longrightarrow$}
\hspace{0.2cm}
\raisebox{-20 pt}{\begin{overpic}[height=50
pt]{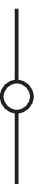}
\put(0,22){$+$}
\end{overpic}}\,
\hspace{1.0cm}
\raisebox{-20 pt}{\begin{overpic}[height=50
pt]{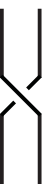}
\end{overpic}}
  \hspace{0.2cm} \mbox{\large$\longrightarrow$}
\hspace{0.2cm}
\raisebox{-20 pt}{\begin{overpic}[height=50
pt]{half-twist01.eps}
\put(0,22){$-$}
\end{overpic}}\,
\if0
\raisebox{-20 pt}{\begin{overpic}[height=50
pt]{half-twist02.eps}
\end{overpic}}
  \hspace{0.2cm} \mbox{$\overset{\mbox{l.h.}}{\sim}$}
\hspace{0.2cm}
\raisebox{-20 pt}{\begin{overpic}[height=50
pt]{half-twist03.eps}
\end{overpic}}
\fi
$$
    \caption{A band and a disk.}
    \label{diskband}
\end{figure}

\par
For a given $C_k$-tree $T$ for a (string) link $L$, there is a procedure to construct a framed link in a regular neighborhood of $T$. By surgery along the framed link, $L$ is changed into another (string) link in $S^3$ ($D^2 \times I$), so we can regard the surgery as a local move on $L$ as showed in Figure~\ref{Ck-tree}, called a $C_k$-move. We denote the obtained (string) link by $L_T$.
We remark that a $C_1$-move and a $C_2$-move take a band sum of a Hopf link and a Borromean ring respectively. 
The {\it $C_k$-equivalence} is the equivalence relation on (string) links generated by $C_k$-moves and ambient isotopies (with fixing boundary).
\begin{figure}[ht]
$$\raisebox{-15 pt}{\begin{overpic}[bb=0 0 217 73, width=160
pt]{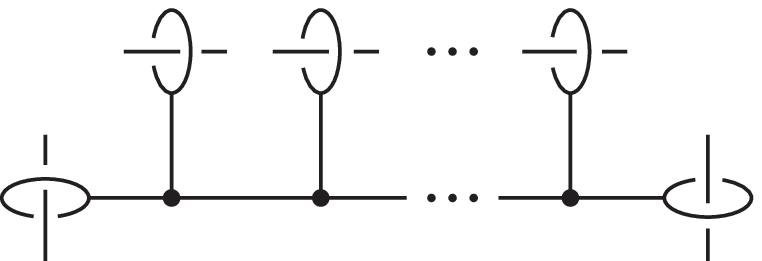}
\end{overpic}}
  \hspace{0.5cm} \mbox{\large$\xrightarrow[]{\rm surgery}$}
\hspace{0.5cm}
\raisebox{-33 pt}{\begin{overpic}[bb=0 0 229 9, width=180
pt]{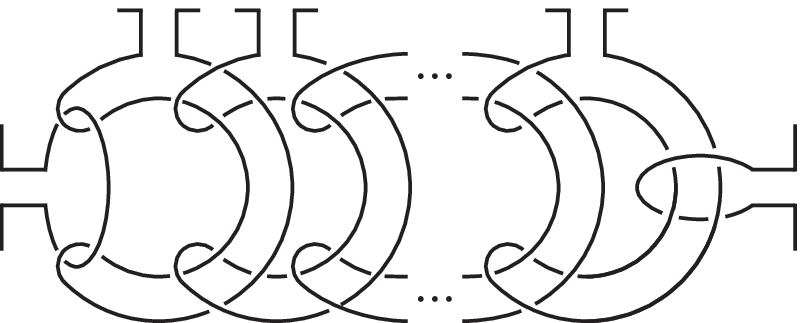}
\end{overpic}}$$
    \caption{A $C_k$-tree and a local move.}
    \label{Ck-tree}
\end{figure}
\end{definition}

The relation in Figure \ref{crossing-clasper} holds up to ambient isotopy. Thus a crossing change can be expressed by a $C_1$-tree. 
\begin{figure}[h]
$$
\raisebox{-20 pt}{\begin{overpic}[width=120pt]{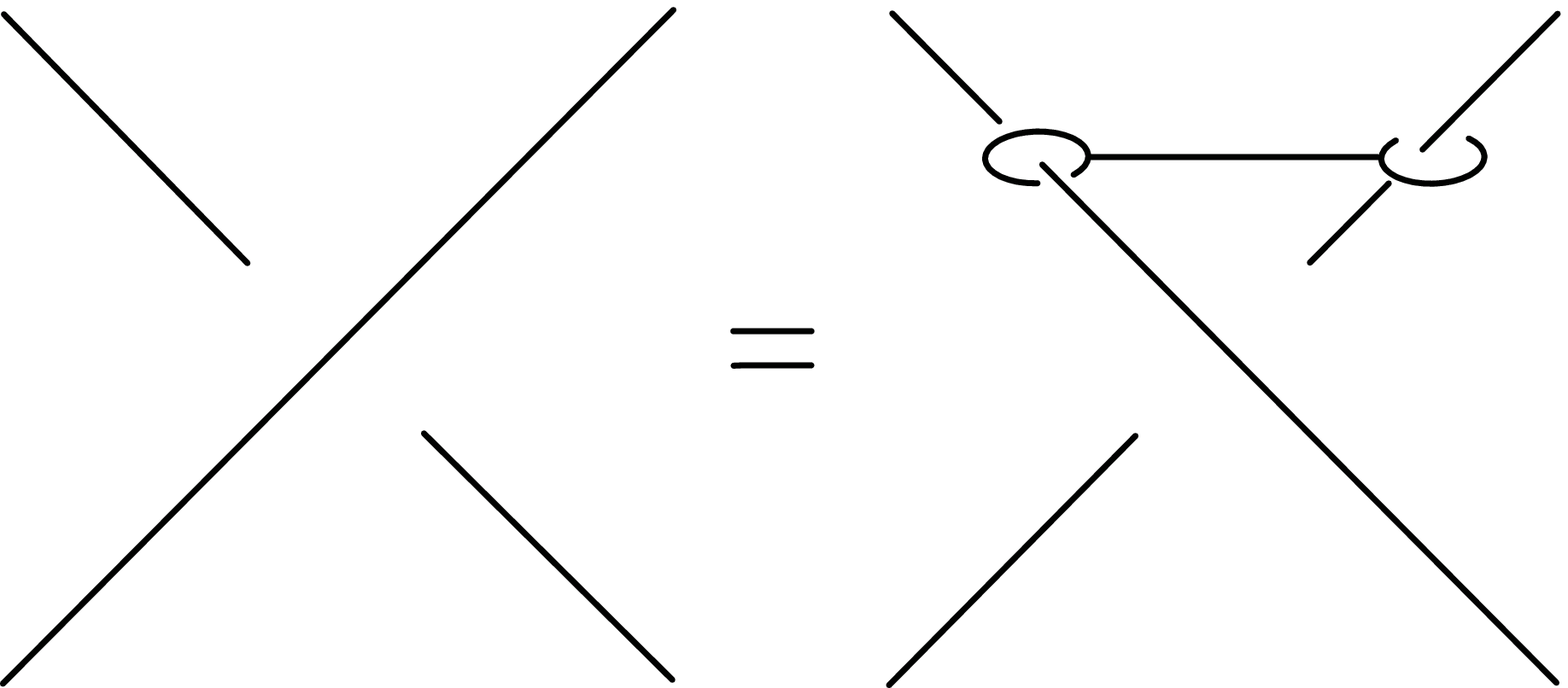}
\end{overpic}} 
$$ 
    \caption{A crossing change and a $C_1$-tree.}
    \label{crossing-clasper}
\end{figure}
%\end{lemma}

\par
Since a surgery along $C_k$-trees corresponds to a band sum of a Hopf link or a Brunnian link, the following two lemmas hold.
 
\begin{lemma} \label{vanish}
If two leaves of a clasper are at the same component of a (string) link, the clasper vanishes up to link-homotopy.  
\end{lemma}

\begin{lemma} \label{edge-self-change}
A crossing change between edges belonging to a $C_k$-tree is achieved by link-homotopy.
%If a $C_k$-tree have a leaf intersecting the $i$-th component, then a crossing change between an edge of the $C_k$-tree and an arc of the $i$-th component is achieved up to link-homotopy. 
\end{lemma}

We recall the relations of the claspers up to link-homotopy. 
We prepare the following lemma about local moves by claspers to transform shapes of (string) links up to link-homotopy (see also \cite{KM2}).
%, which showed by  \cite{Ha}, \cite{Y}, \cite{M} (we also refer to \cite{KM}).
\begin{lemma}[\cite{Ha,MY}, \cite{M} for (6)] \label{clasperlemma}
The following equations hold  up to link-homotopy. \\
$(1)$ A crossing change between a $C_i$-tree and a $C_j$-tree makes a new $C_{i+j+2}$-tree as shown in Figure~$\ref{c1clasper}(1)$. \\
$(2)$ An exchange of leaves of a $C_i$-tree and a $C_j$-tree makes a new $C_{i+j+1}$-tree as shown in Figure~$\ref{c1clasper} (2)$.\\
$(3)$ A full-twist of an edge of a $C_i$-tree vanishes $($Figure $\ref{c1clasper}(3))$. \\
$(4)$ A half-twist of an edge of a $C_i$-tree can be moved to another incident edge $($Figure~$\ref{c1clasper}(4))$. \\
$(5)$ If two $C_i$-trees are parallel where one of them is twisted and the other is non-twisted, then they are canceled $($Figure~$\ref{c1clasper}(5))$. \\
$(6)$ The relation shown in Figure~$\ref{c1clasper}(6)$ holds. We refer the relation by an \textit{IHX relation}.
$(7)$ A crossing change between a $C_k$-tree and a link component makes a new $C_{k+1}$-tree as shown in Figure~$\ref{c1clasper}(7)$, where the bold arcs are link components, the new $C_{k+1}$-tree in the right-hand side is a copy of the $C_k$-tree with a new vertex and a new edge connecting the vertex with the new leaf intersecting with the component and the new $C_{k+1}$-tree has an extra minus half-twist.
%\begin{lemma}[\cite{Ha,FY}] \label{edge-comp-change} as showed in Figure~$\ref{c1clasper}(7)$. \\
%A crossing change between a $C_k$-tree and a link component makes a new $C_{k+1}$-tree up to link-homotopy as follows. 
%$$
%\raisebox{-23 pt}{\begin{overpic}[width=80pt]{comp-clasp-change-5-1.eps}
%\put(2,46){\footnotesize $C_k$}
%\end{overpic}}
%\,\,\,\,
%\mbox{$\overset{\mbox{\rm l.h.}}{\sim}$}
%\,\,\,\,
%\raisebox{-28 pt}{\begin{overpic}[width=80pt]{comp-clasp-change-5-2.eps}
%\put(2,53){\footnotesize $C_k$}\put(2,26){\footnotesize $C_{k+1}$}
%\end{overpic}}
%$$
%\end{lemma}

\begin{figure}[h]
$$(1)\hspace{0.4cm}\raisebox{-27 pt}{\begin{overpic}[width=150pt]{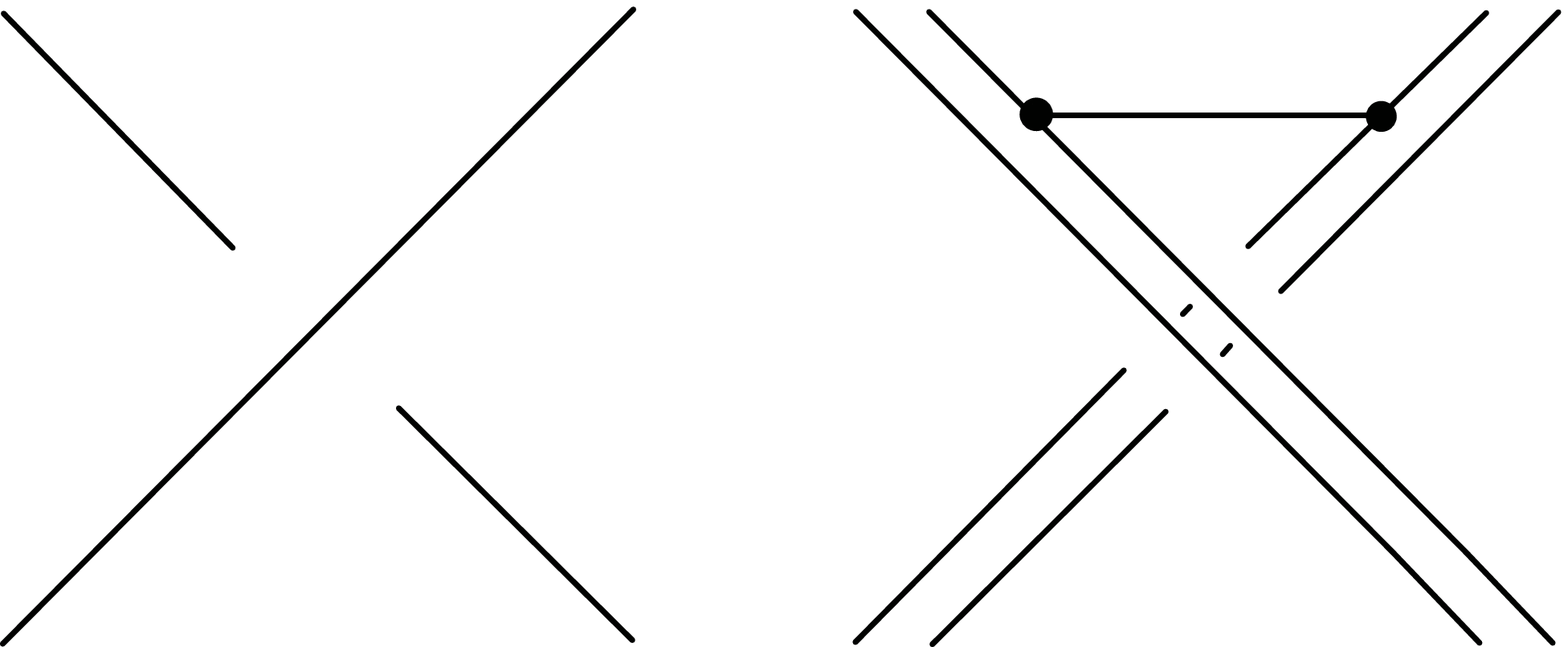}
\put(67,28){$=$}\put(3,60){$C_i$}\put(41,60){$C_j$}\put(100,60){$C_{i+j+2}$}
\end{overpic}}
\hspace{1.0cm}
(2)\hspace{0.4cm}
\raisebox{-32 pt}{\begin{overpic}[width=170pt]{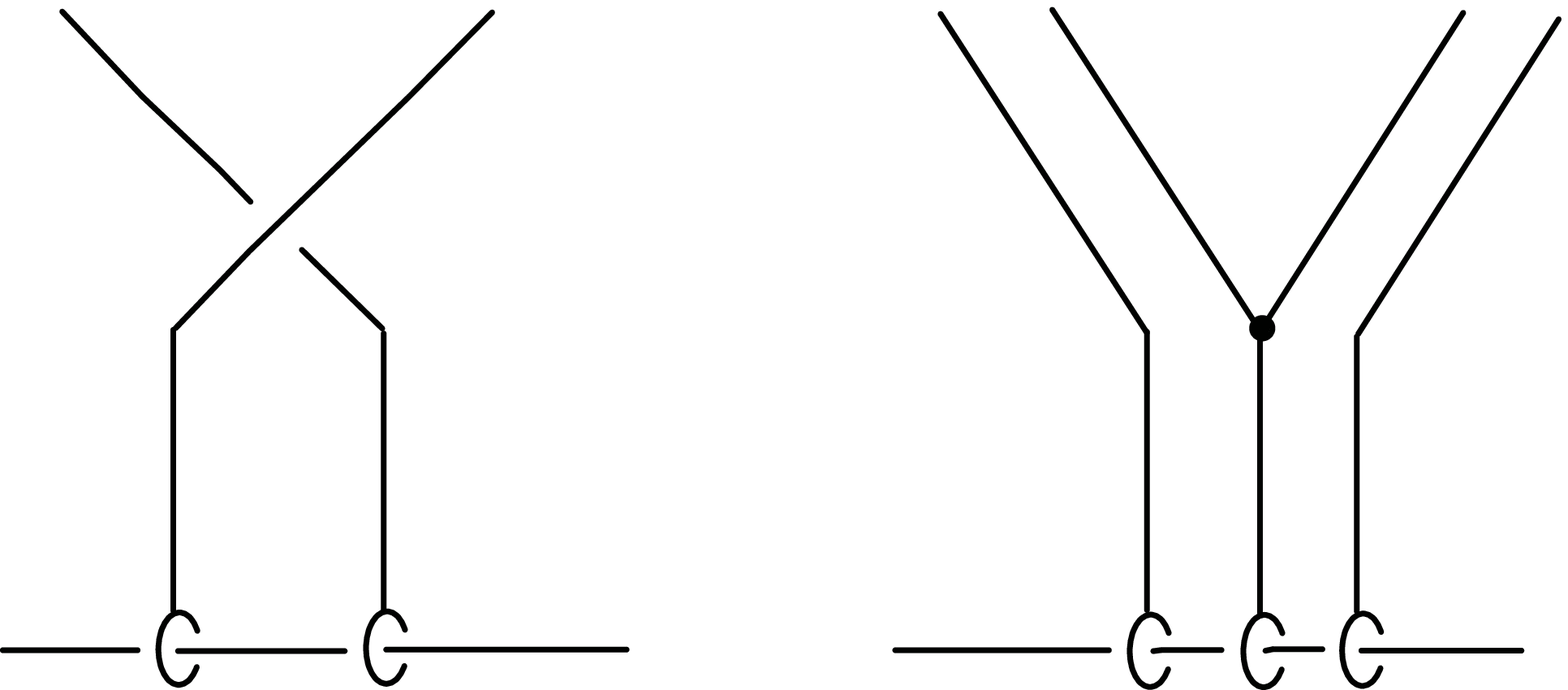}
\put(78,33){$=$}\put(0,60){$C_i$}\put(48,60){$C_j$}\put(120,67){$C_{i+j+1}$}
\end{overpic}}
$$
\vspace{0.3cm}
$$
(3)\hspace{0.4cm}
\raisebox{-28 pt}{\begin{overpic}[width=150pt]{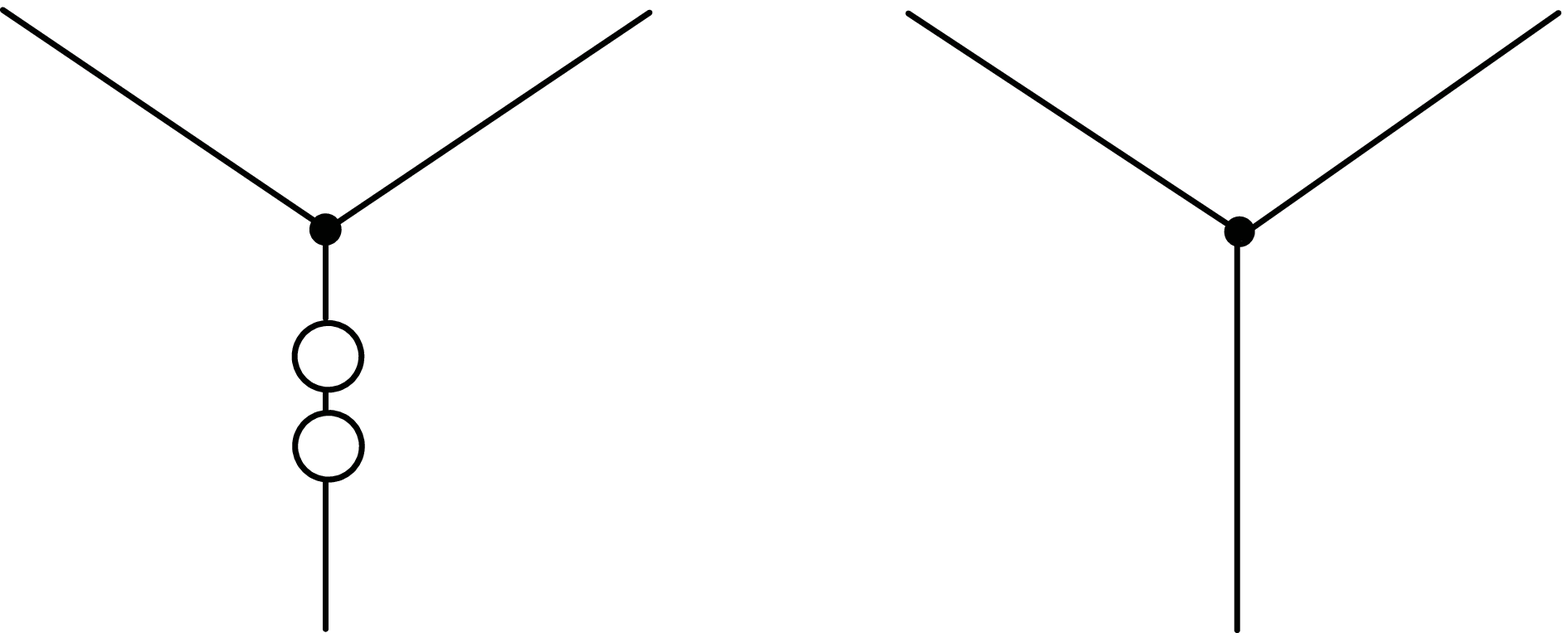}
\put(71,28){$=$}
\put(28,24.5){\tiny$+$}
\put(28,16){\tiny$+$}
\end{overpic}}
\hspace{1.0cm}
(4)\hspace{0.4cm}
\raisebox{-28 pt}{\begin{overpic}[width=150pt]{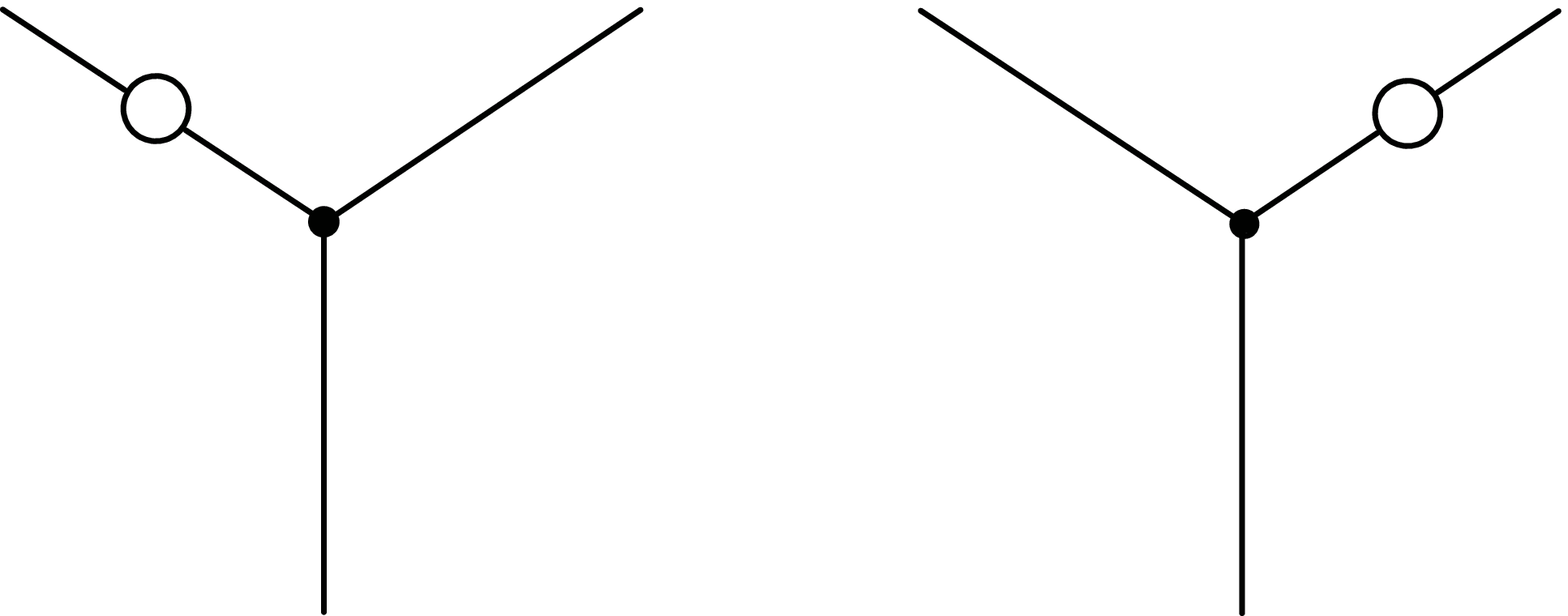}
\put(71,28){$=$}
\put(11.5,46.5){\tiny$+$} \put(131.5,46.5){\tiny$+$}
\end{overpic}}
$$
\vspace{0.3cm}
$$
(5)\hspace{0.4cm}
\raisebox{-32 pt}{\begin{overpic}[width=105pt]{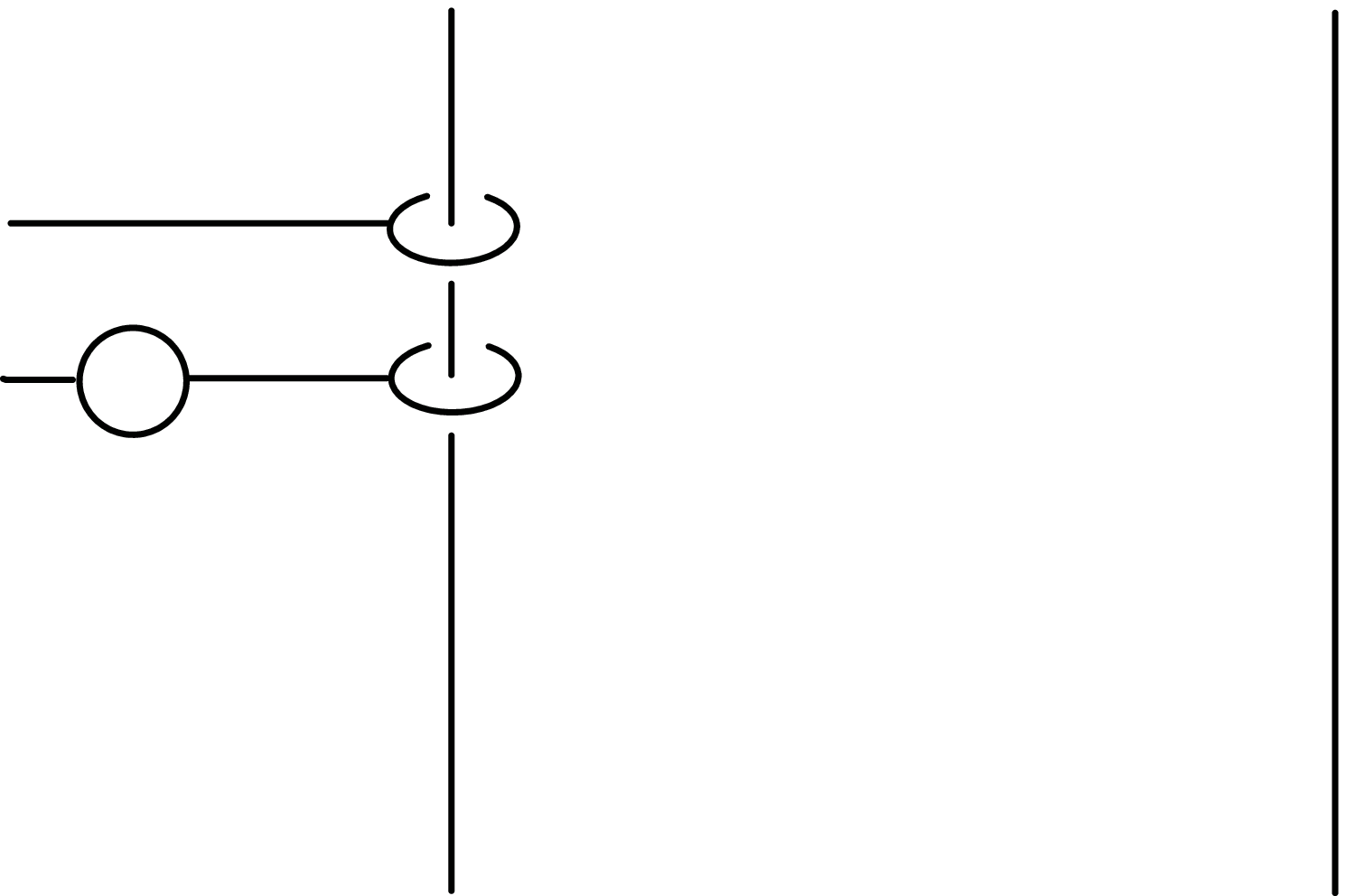}
\put(67,33){$=$} \put(6,37.7){\small$+$}
\end{overpic}} 
\hspace{1.5cm}
(6)\hspace{0.4cm}
\raisebox{-32 pt}{\begin{overpic}[width=170pt]{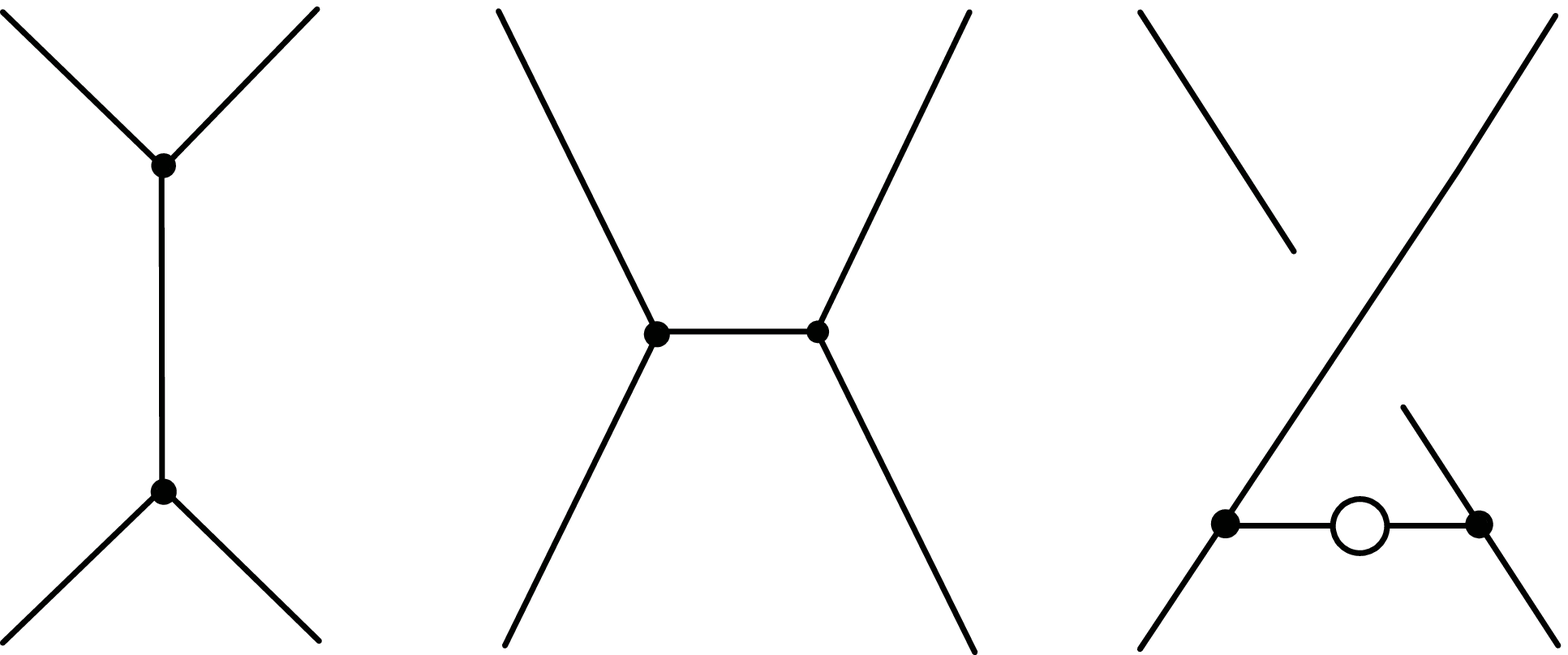}
\put(40,33){$=$}
\put(110,33){$+$} \put(144.5,12.2){\tiny$+$}
\end{overpic}}
$$    
\vspace{0.6cm}
$$
(7)\hspace{0.4cm}
\raisebox{-19 pt}{\begin{overpic}[width=75pt]{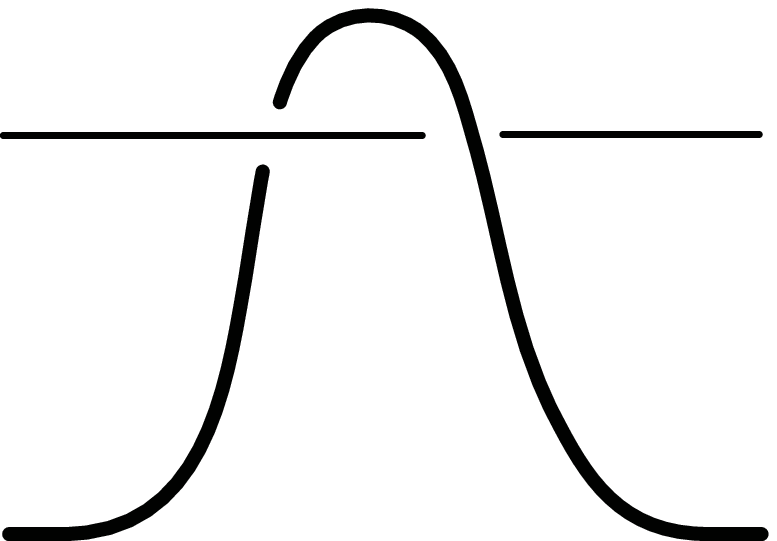}
\put(83,20){$=$}
\end{overpic}} 
\hspace{1.0cm}
\raisebox{-21 pt}{\begin{overpic}[width=75pt]{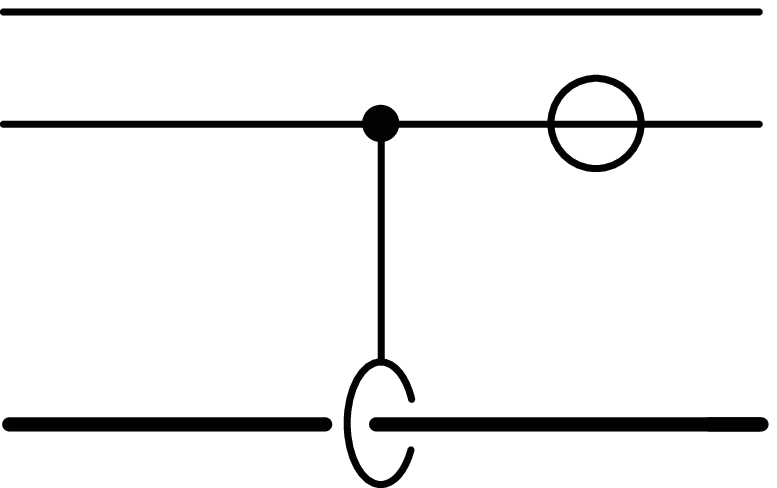}
\end{overpic}} 
$$    
\caption{Equations up to link-homotopy.}
\label{c1clasper}
\end{figure}
\end{lemma}

\begin{remark}
In Lemma \ref{clasperlemma}(1), if the $C_i$-tree and the $C_j$-tree in the left-hand side have leaves at the same component, the new $C_{i+j+2}$-tree in the right-hand side vanishes by Lemma \ref{vanish}.
In Lemma \ref{clasperlemma}(2), if the $C_i$-tree and the $C_j$-tree in the left-hand side have leaves at the same component other than the ones shown in Figure \ref{c1clasper}(2), the new $C_{i+j+1}$-tree in the right-hand side vanishes by Lemma \ref{vanish}.
By Lemma \ref{clasperlemma}(3), a positive half-twist and a negative one are equivalent up to link-homotopy. So we express those by just circles for simplicity.   
\end{remark}

%%%%%%%%%%%%%%%%%%%%%%%%%%%%%%%%%%
\subsection{Canonical form for link-homotopy classes of string links} \label{canonicalform}
%%%%%%%%%%%%%%%%%%%%%%%%%%%%%%%%%%
In \cite{HL} Habegger and Lin essentially defined the Milnor invariant of string links and showed that this invariant decides the link-homotopy class of a string link. 
After that, in \cite{Y} Yasuhara gave a canonical form for link-homotopy classes of string links by using claspers as follows.  In this subsection, we introduce a link-homotopy classification of string links by claspers. 
%We mention another version in \cite{MY} which is equivalent to original one in \cite{Ya}. 

%\begin{theorem}\label{Yasuhara}
%Any $n$-string link $\sigma$ is link-homotopic to $\sigma=\sigma_{(1)} \sigma_{(n-1)}$, where 
%$\sigma_{(i)}$
%\end{theorem}

%\begin{figure}[h]\label{canonicalform}
%\caption{An canonical formula.}
%\end{figure}

Let $\mathcal{I}_{k,n}$ be the set of sequences $i_1i_2 \ldots i_{k}$ of length $k$ on $\{ 1, \cdots , n\}$ such that $i_1 < \cdots < i_{k}$ and ${S}_k$ be the symmetric group of degree $k$. 
Let $\mathcal{J}_{k,n}$ be the set of all sequences $i_1 {s}(i_2 \cdots i_k) i_{k+1}$ where $i_1i_2 \ldots i_{k+1}$ in $\mathcal{I}_{k+1,n}$ and ${s}$ in ${S}_{k-1}$.
For any $J$ in $\mathcal{J}_{k,n}$, let $T_J$ and $T_J^{-1}$ be the $C_{k}$-trees for $1_n$ as illustrated in Figure \ref{C_k-trees}, where $b_{J}$ is the positive $(k-1)$-braid defined by a permutation ${s}$ such that every pair of strings crosses at most once, and each edges of the tree overpasses all strings of $1_n$. Then we have the following canonical forms for string links.

\begin{figure}[h]
$$\raisebox{-0 pt}{\begin{overpic}[width=190pt]{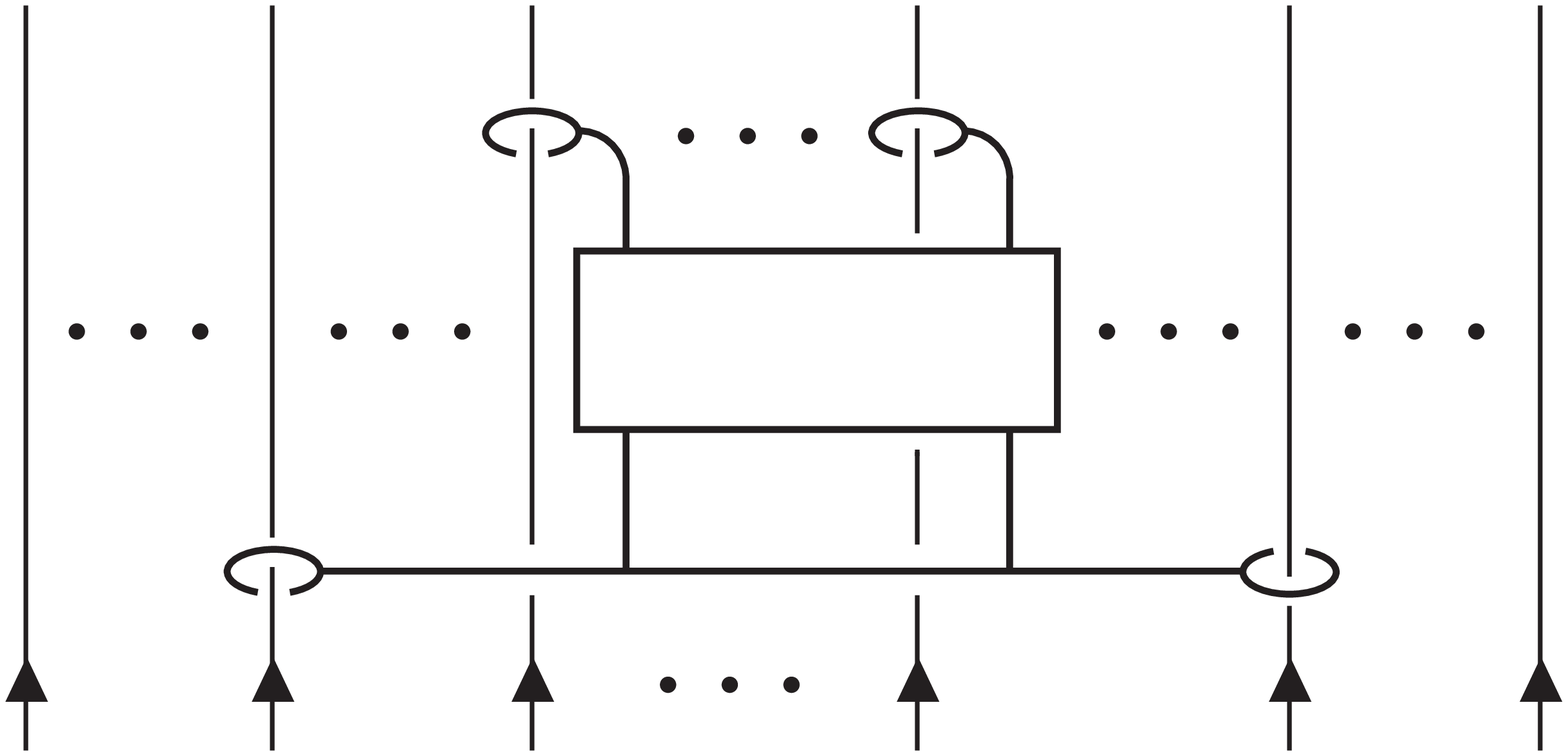}
\put(95,46){$b_J$}
\put(-1,-14){$1$} \put(29,-14){$i_1$} \put(61,-14){$i_2$} 
\put(108,-14){$i_{k}$} \put(148,-14){$i_{k+1}$} \put(184,-14){$n$} 
\end{overpic}}
\hspace{1.5cm}
\raisebox{-0 pt}{\begin{overpic}[width=190pt]{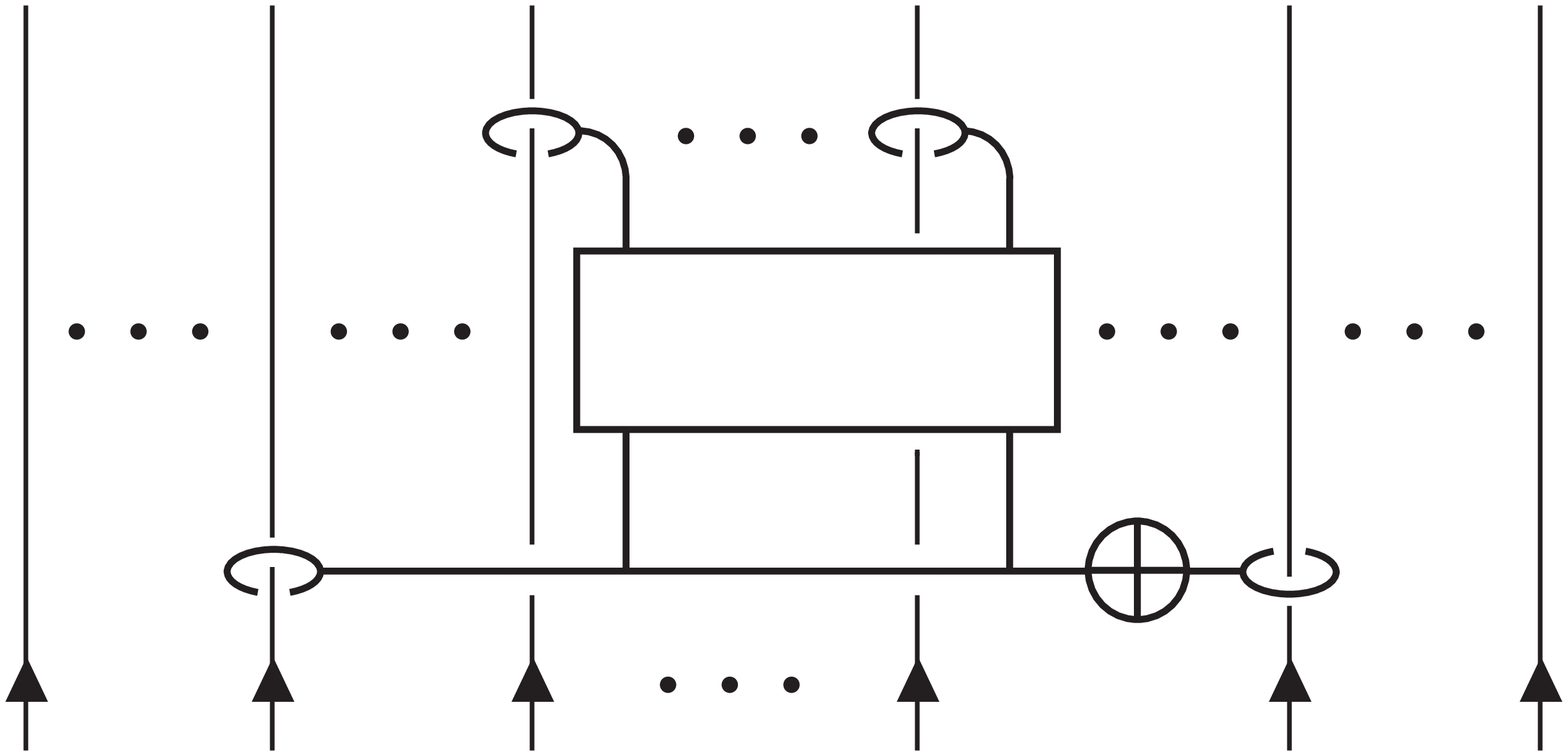}
\put(95,46){$b_J$}
\put(-1,-14){$1$} \put(29,-14){$i_1$} \put(61,-14){$i_2$} 
\put(108,-14){$i_{k}$} \put(148,-14){$i_{k+1}$} \put(184,-14){$n$} 
\end{overpic}}
$$
\vspace{0.0cm}
    \caption{$C_{k}$-trees $T_J$ and $T_J^{-1}$ for $1_n$.}
    \label{C_k-trees}
\end{figure}

%\begin{theorem}[{\cite[Theorem 4.1]{MY2}}, cf.{\cite[Theorem 4.3]{Y}}] \label{standerd form}
\begin{theorem}[{\cite{MY2}}, cf.{\cite{Y}}] \label{standerd form}
Any $n$-string link $\sigma$ is link-homotopic to the string link $l_1 \cdot  l_2 \cdot \cdots \cdot l_{n-1}$, where 
\[
  l_k= \prod_{J \in \mathcal{J}_{k,n}} {(({1_n})_{T_J^{\varepsilon}})}^{|x_J|}, 
    x_J =\left\{ 
          \begin{array}{ll}
               \mu_\sigma(J) & \text{ if } k=1 \\
                 (-1)^{k-1} \{ \mu_\sigma(J)-\mu_{l_1 l_2\cdots l_{k-1}}(J) \}  & \text{ if } k \geq 2,  
          \end{array} 
            \right. 
\]            
where $\mu$ is Milnor's invariant for string links, $J$ in the product appears in the lexicographic order of the lenght 2 subsequence $i_1i_{k+1}$ of $J$ and if they are the same then in the lexicographic order of ${s}(i_2 \cdots i_k)$, and $\varepsilon=1$ if $x_J>0$ and otherwise $\varepsilon=-1$.   
\end{theorem}

\begin{remark}\label{MYrem}
In \cite{MY}, when $k\geq 2$, the sign $(-1)^{k-1}$ of $\mu_L(M)-\mu_{l_1 l_2\cdots l_{k-1}}(M)$ seems to be missing (see \cite{K}). 
\end{remark}

\begin{remark}\label{parallel}
A product of $p$ copies of a $C_k$-tree $T_J^\pm$ for ${\bf 1}_n$ can be expressed by $p$ parallel $C_k$-trees as illustrated in Figure~\ref{parallel trees} by Lemma~\ref{clasperlemma}(1), (2) and Lemma~\ref{vanish}.
We denote it by $T_J^{\pm p}$.
\end{remark}

\begin{figure}[ht]
$$
\raisebox{-25 pt}{\begin{overpic}[height=60pt]{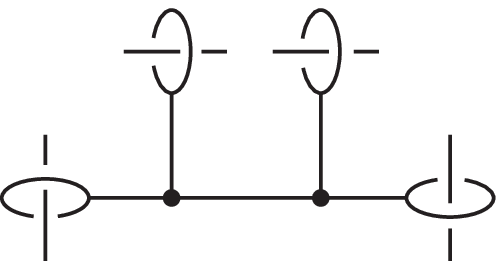}
\put(56,5){$p$}
\end{overpic}}
\,\,\,\,\longrightarrow\,\,\,\,
\raisebox{-25 pt}{\begin{overpic}[height=60pt]{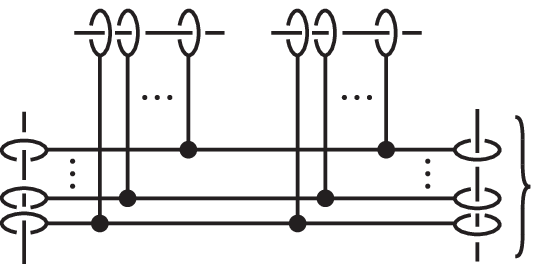}
\put(128,15){$p$}
\end{overpic}}
$$
\caption{$p$ parallel trees.} \label{parallel trees}
\end{figure}

\begin{figure}[ht]
$$
\raisebox{-40 pt}{\begin{overpic}[width=120 pt]{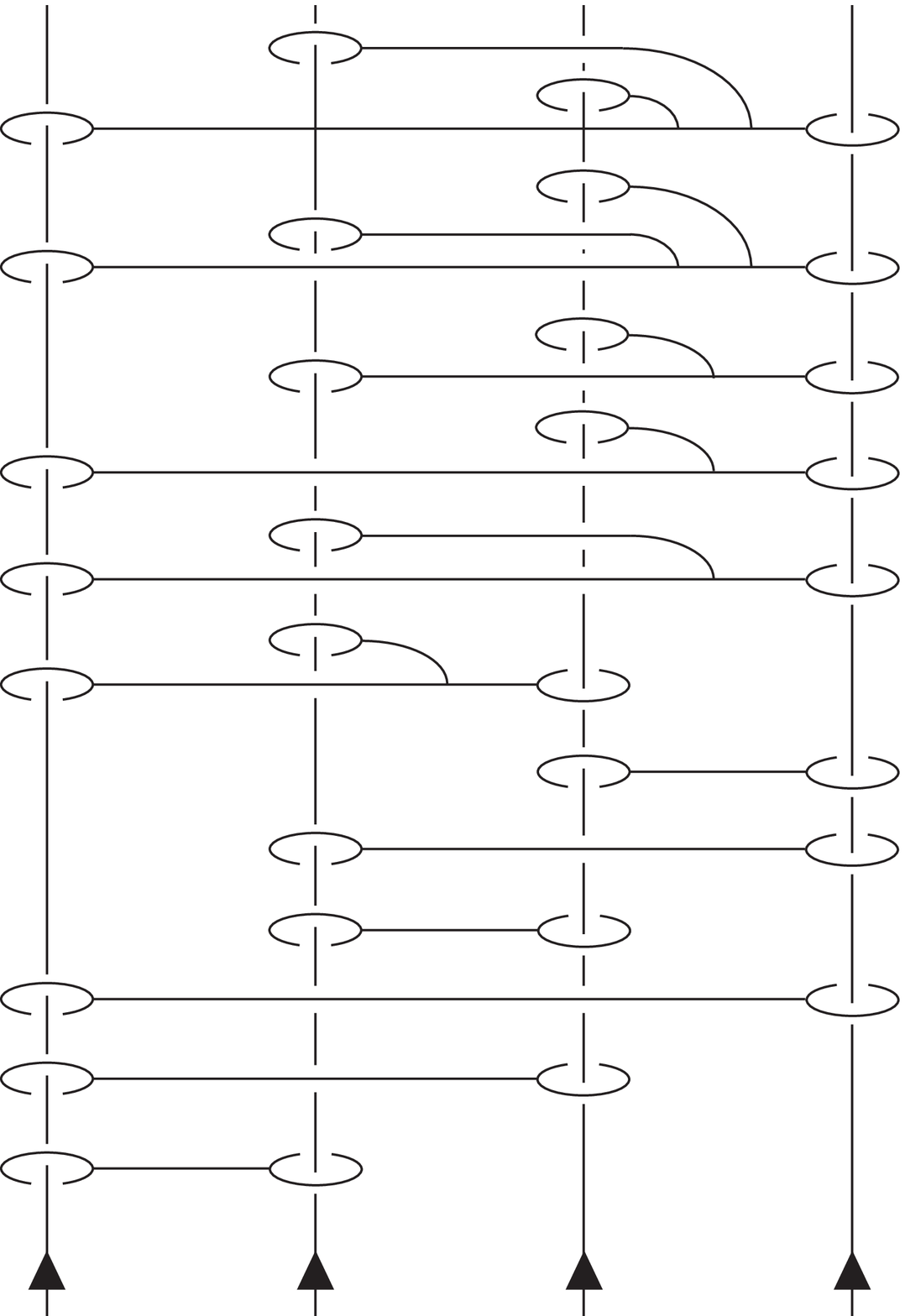}
\put(4,-13){$1$}\put(39,-13){$2$}
\put(75,-13){$3$}\put(111,-13){$4$}
\put(51,18){$y_{12}$}\put(87,29){$y_{13}$}
\put(123,40){$y_{14}$}\put(87,51){$y_{23}$}
\put(123,60){$y_{24}$}\put(123,73){$y_{34}$}
\put(87,83){$y_{123}$}\put(123,97){$y_{124}$}
\put(122,111){$y_{134}$}\put(122,124){$y_{124}$}
\put(122,139){$y_{1234}$}\put(122,157){$y_{1324}$}
\end{overpic}} 
$$
\vspace{0.0cm}
\caption{The canonical form for 4-component string links} \label{4-compSF}
\end{figure}

\begin{figure}[ht]
$$
\raisebox{-20 pt}{\begin{overpic}[width=90 pt]{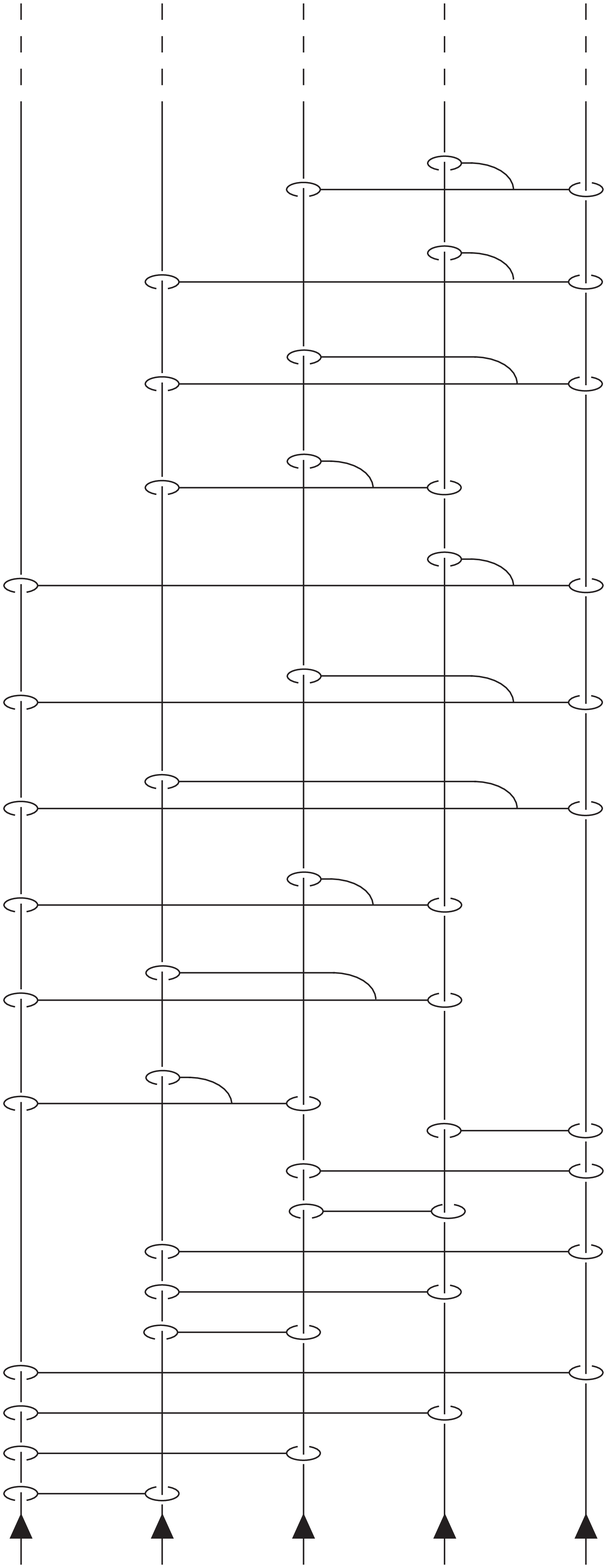}
\put(0,-10){\footnotesize $1$}\put(21,-10){\footnotesize $2$}
\put(42,-10){\footnotesize $3$}\put(63,-10){\footnotesize $4$}
\put(85,-10){\footnotesize $5$}
\put(28,10){\scriptsize $y_{12}$}\put(49,16){\scriptsize $y_{13}$}
\put(71,21){\scriptsize $y_{14}$}\put(92,27){\scriptsize $y_{15}$}
\put(49,34){\scriptsize $y_{23}$}\put(71,39){\scriptsize $y_{24}$}
\put(92,45){\scriptsize $y_{25}$}\put(71,52){\scriptsize $y_{34}$}
\put(92,57){\scriptsize $y_{35}$}\put(92,65){\scriptsize $y_{45}$}
\put(49,70){\footnotesize $y_{123}$}\put(70,84){\footnotesize $y_{124}$}
\put(70,98){\footnotesize $y_{134}$}\put(70,160){\footnotesize $y_{234}$}
\put(92,112){\footnotesize $y_{125}$}\put(92,128){\footnotesize $y_{135}$}
\put(92,146){\footnotesize $y_{145}$}\put(92,175){\footnotesize $y_{235}$}
\put(92,191){\footnotesize $y_{245}$}\put(92,205){\footnotesize $y_{345}$}
\end{overpic}}
\hspace{1.5cm}
\raisebox{-20 pt}{\begin{overpic}[width=90 pt]{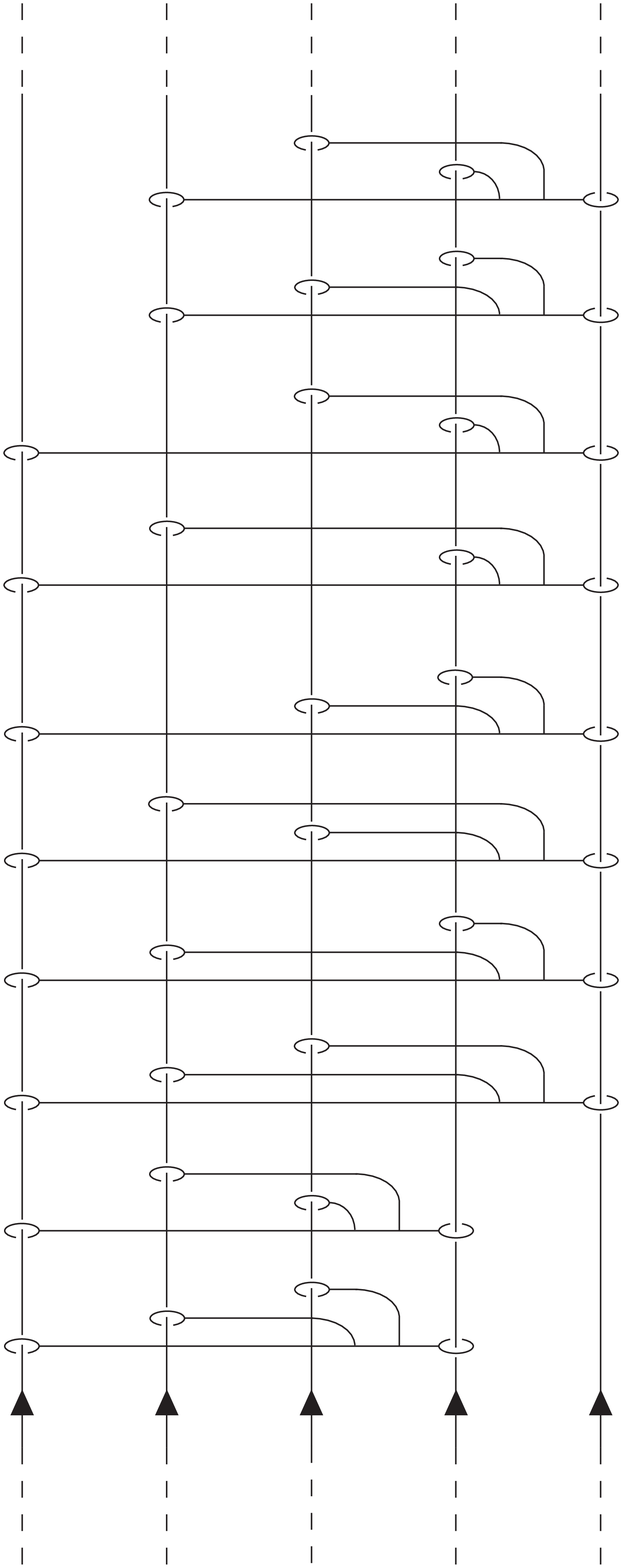}
\put(71,32){\footnotesize $y_{1234}$}\put(71,48){\footnotesize $y_{1324}$}
\put(92,67){\footnotesize $y_{1235}$}\put(92,85){\footnotesize $y_{1245}$}
\put(92,103){\footnotesize $y_{1325}$}\put(92,121){\footnotesize $y_{1345}$}
\put(92,143){\footnotesize $y_{1425}$}\put(92,162){\footnotesize $y_{1435}$}
\put(92,182){\footnotesize $y_{2345}$}\put(92,199){\footnotesize $y_{2435}$}
\end{overpic}}
\hspace{1.5cm}
\raisebox{-20 pt}{\begin{overpic}[width=90 pt]{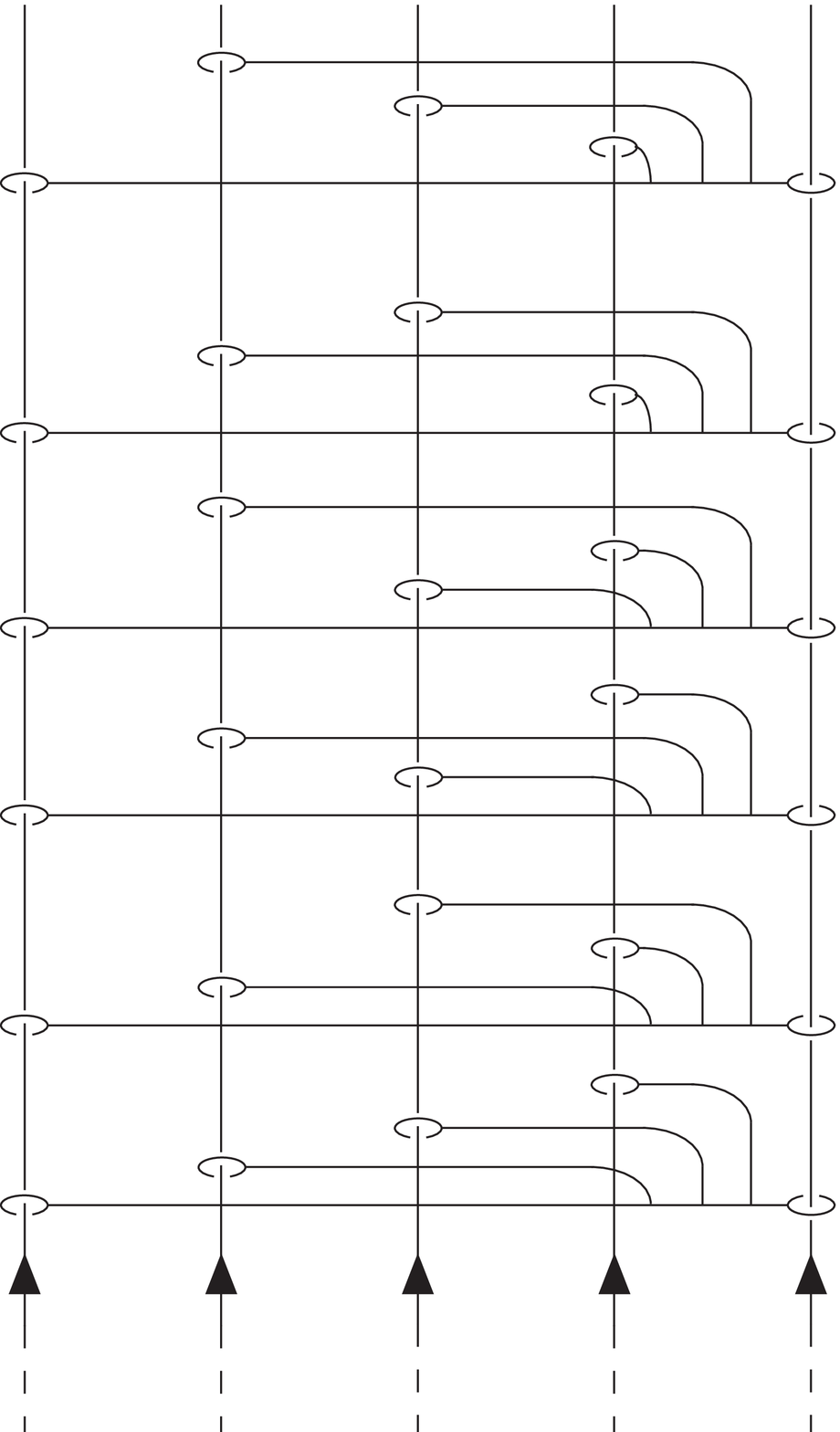}
\put(92,24){\footnotesize $y_{12345}$}\put(92,43){\footnotesize $y_{12435}$}
\put(92,66){\footnotesize $y_{13245}$}\put(92,86){\footnotesize $y_{13425}$}
\put(92,106){\footnotesize $y_{14235}$}\put(92,133){\footnotesize $y_{14325}$}
\end{overpic}}
$$
\vspace{0.0cm}
\caption{The canonical form of 5-component string links} \label{5-compSF}
\end{figure}

%%%%%%%%%%%%%%%%%%%%%%%%%%%%%%%%%%
\section{calculation of partial conjugations} \label{CalcParConj}
%%%%%%%%%%%%%%%%%%%%%%%%%%%%%%%%%%
%前a?論文との関係をe¨?a??

In this section, we calculate the actions of the generators $(\overline{x}_i, \overline{x}_i)_j$ of partial conjugations for 4- and 5-component string links. We abbreviate $(\overline{x}_i, \overline{x}_i)_j$ to $\overline{x}_{ij}$. 
These results give presentations of the link-homotopy classes of 4- and 5-component links, see Theorems \ref{Rep4-compLink} and \ref{Rep5-compLink}. In the 4-component case, 
the result gives an alternative proof of Levine's classification of the link-homotopy classes of 4-component links \cite{Le2}, see Remark \ref{Rem-for-Levine's-result}.  
%the presentation coincides with the one in authors' previous work \cite{KM}. It was proved that the presentation in \cite{KM} is equivalent to Levine's classification \cite{L} of the link-homotopy-classes of 4-component links. 
%obtained by modifying Levine's classification of the link-homotopy classes of 4-component links \cite{L}. 
%Thus our calculation gives an alternative proof of Levine's result.
\par 
Theorem \ref{standerd form} gives canonical forms for string links. The canonical forms for the 4- and 5-component cases are as in Figure \ref{4-compSF} and \ref{5-compSF} respectively. 
For the canonical form, the action of the partial conjugation $\overline{x}_{ij}$ in Figure \ref{partialconj} is shown as in Figure \ref{ActSF} by using claspers, where $b$ presents the canonical forms.  

\begin{figure}[ht] 
\raisebox{-25 pt}{\begin{overpic}[height=120pt]{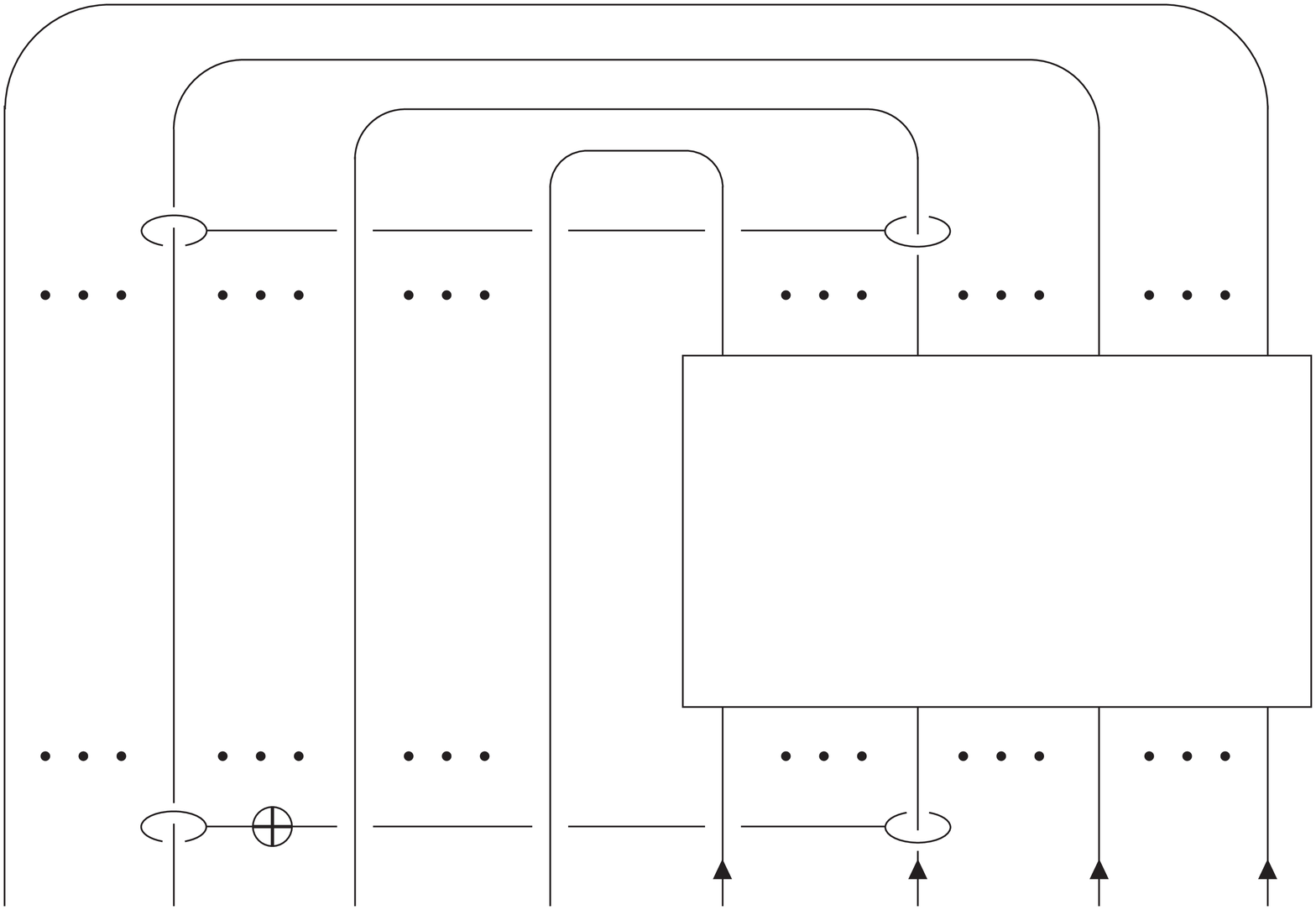}
\put(21,-12){$i$} \put(118,-12){$j$} \put(130,45){\Large $b$}
\end{overpic}}
\vspace{0.3cm}
\caption{Partial conjugation $\overline{x}_{ij}$ for the canonical form $b$} \label{ActSF}
\end{figure}

We transform the shape of Figure \ref{ActSF} back to the canonical form by using the relations in Section \ref{preparations}. Then the numbers of claspers change and the differences from the previous canonical form present the action of the partial conjugation $\overline{x}_{ij}$. For the generators $\{\overline{x}_{ij}\mid 1\leq i,j\leq n, i\neq j\}$ of the partial conjugations for $n$-component string links, we call $\overline{x}_{ij}$ a commutator of \textit{degree $1$} and $[\overline{x}_{ij}, C]=\overline{x}_{ij}C\overline{x}_{ij}^{-1} C^{-1}$ of \textit{degree $k$} if $C$ is a degree $k-1$ commutator. 
In \cite{Le2}, Levine classified the link-homotopy classes of 4-component links. In the process, he simplified the result by using the commutators of the actions corresponding to the partial conjugations. In \cite{HL}, Habegger and Lin pointed out that the actions of the $k$-th degree commutators commute up to those of $(k+1)$-th degree commutators and $(n-1)$-th degree commutators are identities for an $n$-component link. 
We calculate the actions of commutators $[\overline{x}_{ij}, \overline{x}_{kl}]=\overline{x}_{ij}\overline{x}_{kl}(\overline{x}_{ij})^{-1}(\overline{x}_{kl})^{-1}$ of the partial conjugations $\overline{x}_{ij}$ for the 4- and 5-component cases and additionally $[\overline{x}_{ij}, [\overline{x}_{kl}, \overline{x}_{st}]]$ for the 5-component case. By using the action of the commutators, we modify the actions of $\overline{x}_{ij}$ to simplified ones $\overline{x}'_{ij}$.
\par 
The results of this section are also on the second author's web page \cite{Mweb} as excel files and Mathematica files. 

\subsection{4-component case}
The results of the calculations of the actions of $\overline{x}_{ij}$ are in Table \ref{Act4-compSF}, where the entries are the difference of the number of the claspers. Note that $y_{ij}$ are invariants (which are equivalent to the linking numbers) and we omit them in the table. An example of a calculation for $\overline{x}_{ij}$ is shown in Appendix \ref{example-calc}.

\begin{table}[htb] 
\begin{center}
\caption{Partial conjugations for 4-component string links} \label{Act4-compSF}
{
  \begin{tabular}{|c|l|l|l|l|l|l|} \hline
  & $\overline{x}_{12}$ & $\overline{x}_{13}$ & $\overline{x}_{14}$ & $\overline{x}_{21}$ \\ \hline
  
 $y_{123}$ & $y_{23}$ & $-y_{23}$  & 0 &  $-y_{13}$  \\ \hline 
 
 $y_{124}$ & $y_{24}$ & 0 & $-y_{24}$ & $-y_{14}$  \\ \hline 
 
 $y_{134}$ & $0$ & $y_{34}$ & $-y_{34}$ & 0  \\ \hline
 
 $y_{234}$ & $0$ & 0 & 0 & 0 \\ \hline
 
 $y_{1234}$ & $y_{234}-y_{24}y_{34}$ & $y_{34}-y_{23}y_{34}$ & $-y_{234}-y_{34}$ & $-y_{134}+y_{14}y_{34}$  \\ 
 
 & $+y_{23}y_{34}$ & & $+y_{24}y_{34}$ & $-y_{13}y_{34}+y_{13}y_{14}$   \\ \hline
 
 $y_{1324}$ & $-y_{23}y_{24}$ & $-y_{234} +y_{23}y_{24}$ & $y_{234}-y_{24}$ & $-y_{13}y_{14}$  \\ \hline \hline
  
   & $\overline{x}_{23}$ & $\overline{x}_{24}$ & $\overline{x}_{31}$ & $\overline{x}_{32}$\\ \hline
  
 $y_{123}$  & $y_{13}$ & 0 & $y_{12}$ & $-y_{12}$\\ \hline 
 
 $y_{124}$  & 0 & $y_{14}$ & 0 & 0 \\ \hline 
 
 $y_{134}$  & 0 & 0 & $-y_{14}$ & 0\\ \hline
 
 $y_{234}$  & $y_{34}$ & $-y_{34}$ & $0$ & $-y_{24}$\\ \hline
 
 $y_{1234}$  & $y_{134}+y_{13}y_{34}$ & $y_{14}-y_{14}y_{34}$ & $-y_{14}-y_{12}y_{14}$ & $-y_{124}+y_{12}y_{14}$\\ 
 
 & $-y_{13}y_{14}$ &  &  &   \\ \hline
 
 $y_{1324}$ & $-y_{134}+y_{13}y_{14}$ & $y_{134}-y_{14}$  & $-y_{124}+y_{14}$ & $y_{124}-y_{12}y_{14}$\\ \hline \hline
 
\if0
   & $\overline{x}_{31}$ & $\overline{x}_{32}$ & $\overline{x}_{34}$ & $\overline{x}_{41}$ & $\overline{x}_{42}$ & $\overline{x}_{43}$ \\ \hline 
  
   $y_{123}$ & $y_{12}$ & $-y_{12}$ & 0 & 0 & 0 & $0$ \\ \hline 
   
 $y_{124}$ & 0 & 0 & 0 & $y_{12}$ & $-y_{12}$ & $0$ \\ \hline 
 
 $y_{134}$ & $-y_{14}$ & 0 & $y_{14}$ & $y_{13}$ & 0 & $-y_{13}$  \\ \hline
 
 $y_{234}$ & $0$ & $-y_{24}$ & $y_{24}$ & 0 & $y_{23}$ & $-y_{23}$ \\ \hline
 
 $y_{1234}$ & $-y_{14}-y_{12}y_{14}$ & $-y_{124}+y_{12}y_{14}$ & $y_{124}$ & $y_{123}+y_{12}$ & $-y_{12}y_{23}-y_{12}$ & $-y_{123}+y_{13}y_{23}$ \\ 
 
  &  &  &  & $-y_{13}y_{23}+y_{12}y_{23}$ & &  \\ \hline

 $y_{1324}$ & $-y_{124}+y_{14}$ & $y_{124}-y_{12}y_{14}$ & $-y_{14}y_{24}$ & $-y_{123}+y_{13}$ & $y_{123}-y_{12}y_{13}$ & $-y_{13}y_{23}$ \\ 
 
  & $+y_{12}y_{14}-y_{12}y_{24}$ & $+y_{12}y_{24}$ & & $+y_{12}y_{13}-y_{12}y_{23}$ & $+y_{12}y_{23}$ &  \\ 
  
  & $+y_{14}y_{24}$ &  & & $+y_{13}y_{23}$ &  &  \\ \hline \hline
\fi

  & $\overline{x}_{34}$ & $\overline{x}_{41}$ & $\overline{x}_{42}$ & $\overline{x}_{43}$ \\ \hline 
  
   $y_{123}$ & 0 & 0 & 0 & $0$ \\ \hline 
   
 $y_{124}$ & 0 & $y_{12}$ & $-y_{12}$ & $0$ \\ \hline 
 
 $y_{134}$ & $y_{14}$ & $y_{13}$ & 0 & $-y_{13}$  \\ \hline
 
 $y_{234}$ & $y_{24}$ & 0 & $y_{23}$ & $-y_{23}$ \\ \hline
 
 $y_{1234}$ & $y_{124}$ & $y_{123}+y_{12}$ & $-y_{12}y_{23}-y_{12}$ & $-y_{123}+y_{13}y_{23}$ \\ 
 
  &  & $-y_{13}y_{23}+y_{12}y_{23}$ & &  \\ \hline

 $y_{1324}$  & $-y_{14}y_{24}$ & $-y_{123}+y_{13}$ & $y_{123}-y_{12}y_{13}$ & $-y_{13}y_{23}$ \\ 
 
  & & $+y_{12}y_{13}-y_{12}y_{23}$ & $+y_{12}y_{23}$ &  \\ 
  
  & & $+y_{13}y_{23}$ &  &  \\ \hline
  \end{tabular}
}
  \end{center}
\end{table}

\if0
\begin{center}
\begin{tabular}{|c|c|} \hline
  & $\overline{x}_{12}$ \\ \hline 
  $y_{123}$ & a \\ \hline 
  $y_{124}$ & a \\ \hline 
  $y_{134}$ & a \\ \hline 
  $y_{234}$ & a \\ \hline 
  $y_{1234}$ & a \\ \hline 
  $y_{1324}$ & a \\ \hline 
  \end{tabular} 
\end{center}
\fi

We also calculate the actions of the commutators $[\overline{x}_{ij}, \overline{x}_{kl}]$. The results are in Table \ref{ActCom4-compSF}. The other commutators which do not appear the table are equal to the identity or the ones in the table.

\begin{table}[htb] 
\begin{center}
   \caption{Commutators of $\overline{x}_{ij}$ for 4-component string links} \label{ActCom4-compSF}
  \begin{tabular}{|c|c|c|c|c|c|c|} \hline
      & $[\overline{x}_{21}, \overline{x}_{31}]$ & $[\overline{x}_{21}, \overline{x}_{41}]$ & $[\overline{x}_{31}, \overline{x}_{41}]$ & $[\overline{x}_{12}, \overline{x}_{32}]$ & $[\overline{x}_{12}, \overline{x}_{42}]$ & $[\overline{x}_{13}, \overline{x}_{23}]$ \\ \hline
    $y_{123}$ & 0 & 0 & 0 & 0 & 0 & 0 \\
    \hline 
    $y_{124}$ & 0 & 0 & 0 & 0 & 0 & 0 \\
    \hline
    $y_{134}$ & 0 & 0 & 0 & 0 & 0 & 0 \\
    \hline 
    $y_{234}$ & 0 & 0 & 0 & 0 & 0 & 0 \\
    \hline
    $y_{1234}$ & $-y_{14}$ & 0 & $y_{12}$ & 0 & $-y_{23}$ & $y_{34}$ \\
    \hline 
    %$[\overline{x}_{32}, \overline{x}_{42}]$ & 0 & 0 & 0 & 0 & $-y_{12}$ & 0 \\
    %\hline
    $y_{1324}$ & $y_{14}$ & $y_{13}$ & 0 & $y_{24}$ & $y_{23}$ & 0 \\
    \hline
    %$[\overline{x}_{13}, \overline{x}_{43}]$ & 0 & 0 & 0 & 0 & $y_{23}$ & $-y_{23}$ \\
    %\hline 
    %$[\overline{x}_{23}, \overline{x}_{43}]$ & 0 & 0 & 0 & 0 & 0 & $-y_{13}$ \\
    %\hline
    %$[\overline{x}_{14}, \overline{x}_{24}]$ & 0 & 0 & 0 & 0 & $-y_{34}$ & 0 \\
    %\hline
    %$[\overline{x}_{14}, \overline{x}_{34}]$ & 0 & 0 & 0 & 0 & 0 & $-y_{24}$ \\
    %\hline 
    %$[\overline{x}_{24}, \overline{x}_{34}]$ & 0 & 0 & 0 & 0 & $y_{14}$ & $-y_{14}$ \\
    %\hline
  \end{tabular} 
\end{center}
\end{table}

\begin{remark}
We have that, for the 4-component case, $[\overline{x}_{ik}, \overline{x}_{jk}]=[\overline{x}_{il}, \overline{x}_{jl}]$ where $i,j,k,l$ are distinct integers in $\{1,2,3,4\}$. 
\end{remark}

We can simplify the relation $\overline{x}_{ij}$ in Table \ref{Act4-compSF} by using the commutators in Table \ref{ActCom4-compSF} to the simplified relations $\overline{x}'_{ij}$. The actions of simplified partial conjugations $\overline{x}'_{ij}$ are in Table \ref{ModAct4-compSF}. For examples, the first two $\overline{x}'_{ij}$ are obtained as follows,
\begin{align*}
\overline{x}'_{12}=[\overline{x}_{13}, \overline{x}_{23}]^{y_{24}-y_{23}} \circ [\overline{x}_{12}, \overline{x}_{32}]^{y_{23}}\circ\overline{x}_{12}, \quad
\overline{x}'_{13}=[\overline{x}_{13}, \overline{x}_{23}]^{y_{23}} \circ [\overline{x}_{12}, \overline{x}_{32}]^{-y_{23}}\circ\overline{x}_{13}.
\end{align*}

\begin{table}[htb] 
\begin{center}
\caption{Simplified partial conjugations $\overline{x}'_{ij}$ for 4-component string links} \label{ModAct4-compSF}
{
  \begin{tabular}{|c|c|c|c||c|c|c|} \hline
  & $\overline{x}'_{12}$ & $\overline{x}'_{13}$ & $\overline{x}'_{14}$ & $\overline{x}'_{21}$ & $\overline{x}'_{23}$ & $\overline{x}'_{24}$ \\ \hline 
  
 $y_{123}$ & $y_{23}$ & $-y_{23}$  & 0 &  $-y_{13}$ & $y_{13}$ & 0 \\ \hline  
 
 $y_{124}$ & $y_{24}$ & 0 & $-y_{24}$ & $-y_{14}$ & 0 & $y_{14}$ \\ \hline
 
 $y_{134}$ & $0$ & $y_{34}$ & $-y_{34}$ & 0 & 0 & 0 \\ \hline
 
 $y_{234}$ & $0$ & 0 & 0 & 0 & $y_{34}$ & $-y_{34}$ \\ \hline
 
 $y_{1234}$ & $y_{234}$ & $0$ & $-y_{234}$ & $-y_{134}$ & $y_{134}$ & $0$ \\ \hline
 
 $y_{1324}$ & $0$ & $-y_{234}$ & $y_{234}$ & $0$ & $-y_{134}$ & $y_{134}$ \\ \hline\hline
  
  & $\overline{x}'_{31}$ & $\overline{x}'_{32}$ & $\overline{x}'_{34}$ & $\overline{x}'_{41}$ & $\overline{x}'_{42}$ & $\overline{x}'_{43}$ \\ \hline 
  
   $y_{123}$ & $y_{12}$ & $-y_{12}$ & 0 & 0 & 0 & $0$ \\ \hline 
   
 $y_{124}$ & 0 & 0 & 0 & $y_{12}$ & $-y_{12}$ & $0$ \\ \hline 
 
 $y_{134}$ & $-y_{14}$ & 0 & $y_{14}$ & $y_{13}$ & 0 & $-y_{13}$  \\ \hline
 
 $y_{234}$ & $0$ & $-y_{24}$ & $y_{24}$ & 0 & $y_{23}$ & $-y_{23}$ \\ \hline
 
 $y_{1234}$ & $0$ & $-y_{124}$ & $y_{124}$ & $y_{123}$ & $0$ & $-y_{123}$ \\ \hline

 $y_{1324}$ & $-y_{124}$ & $y_{124}$ & $0$ & $-y_{123}$ & $y_{123}$ & $0$ \\ \hline
  \end{tabular}
}
  \end{center}
\end{table}

We remark that, by replacing $\overline{x}_{ij}$ to $\overline{x}'_{ij}$, the actions of 
the commutators in Table \ref{ActCom4-compSF} do not change; $[\overline{x}_{ij},\overline{x}_{kl}]=[\overline{x}'_{ij},\overline{x}'_{kl}]$. Since $\displaystyle \prod_{\substack{j; j \neq i}} \overline{x}'_{ij}={\rm id}$ holds for each $1\leq i \leq 4$, we can erase 4 relations. We erase $\{\overline{x}'_{14}, \overline{x}'_{24}, \overline{x}'_{34}, \overline{x}'_{43}\}$. From the changes of $y_{ijk}$ in Table \ref{ModAct4-compSF}, we see that no more relations can be erased. For example, if $\overline{x}'_{12}$ is erased additionally, we can not change $y_{124}$ by $y_{24}$ by using the other $\overline{x}'_{ij}$. Thus the smallest number of the relations is 8.
Note that, for the commutators in Table \ref{ActCom4-compSF}, we do not use $\overline{x}_{ij}$ corresponding to the erased relations. 

\if0 
Finally  
we have the table of the relations in Table \ref{SmpfModAct4-compSF} which generate the changes of numbers of the claspers raised by the partial conjugations. 

\begin{table}[htb] 
\begin{center}
   \caption{Simplified partial conjugations $\overline{x}'_{ij}$ for 4-component string links} \label{SmpfModAct4-compSF}
  \begin{tabular}{|c|c|c|c|c||c|c|} \hline
      & $y_{123}$ & $y_{124}$ & $y_{134}$ & $y_{234}$ & $y_{1234}$ & $y_{1324}$ \\ \hline 
    $\overline{x}'_{12}$ & $y_{23}$ & $y_{24}$ & 0 & 0 & $y_{234}$ & $0$ \\ \hline
    $\overline{x}'_{13}$ & $-y_{23}$ & 0 & $y_{34}$ & 0 & 0 & $-y_{234}$ \\ \hline 
    %$\overline{x}'_{14}$ & 0 & $-y_{24}$ & $-y_{34}$ & 0 & $-y_{234}$ & $y_{234}$ \\ \hline 
    $\overline{x}'_{21}$ & $-y_{13}$ & $-y_{14}$ & 0 & 0 & $-y_{134}$ & $0$ \\ \hline 
    $\overline{x}'_{23}$ & $y_{13}$ & 0 & 0 & $y_{34}$ & $y_{134}$ & $-y_{134}$ \\ \hline 
    %$\overline{x}'_{24}$ & 0 & $y_{14}$ & 0 & $-y_{34}$ & $0$ & $y_{134}$ \\ \hline 
    $\overline{x}'_{31}$ & $y_{12}$ & 0 & $-y_{14}$ & $0$ & 0 & $-y_{124}$ \\ \hline 
    $\overline{x}'_{32}$ & $-y_{12}$ & 0 & 0 & $-y_{24}$ & $-y_{124}$ & $y_{124}$ \\ \hline
    %$\overline{x}'_{34}$ & 0 & 0 & $y_{14}$ & $y_{24}$ & $y_{124}$ & $0$ \\ \hline 
    $\overline{x}'_{41}$ & 0 & $y_{12}$ & $y_{13}$ & 0 & $y_{123}$ & $-y_{123}$ \\ \hline
    $\overline{x}'_{42}$ & 0 & $-y_{12}$ & 0 & $y_{23}$ & 0 & $y_{123}$ \\ \hline 
    %$\overline{x}'_{43}$ & 0 & 0 & $-y_{13}$ & $-y_{23}$ & $-y_{123}$ & $0$ \\ \hline 
  \end{tabular} 
  \end{center}
\end{table}
\fi

As a conclusion, we have a presentation of the link-homotopy classes of 4-component links. 

\begin{theorem} \label{Rep4-compLink} 
The set $\mathcal{L}_4$ of link-homotopy classes of 4-component links has a presentation which is the set of 12-tuples of integers (i.e. the numbers of claspers) modulo the 
the relations $X_4$ generated by the
8 relations $\{\overline{x}'_{12},\overline{x}'_{13}, \overline{x}'_{21}, \overline{x}'_{23}, \overline{x}'_{31}\overline{x}'_{32}, \overline{x}'_{41}, \overline{x}'_{42} \}$ in Table \ref{ModAct4-compSF}. Namely, 
$$\mathcal{L}_4=\mathbb{Z}^{12}/{X_4}.$$
\end{theorem}

\if0
\begin{remark} \label{Rem-for-Levine's-result}
In \cite{Le}, Levine classified $\mathcal{L}_4$ via the calculation of the geometric automorphism of the reduced group of 4-component links. In \cite{KM2}, the authors had another presentation of $\mathcal{L}_4$ by introducing the ``tetrahedron'' standard form by using the claspers and proved that it is equivalent to Levine's result. 
This presentation coincides with the one in Theorem \ref{Rep4-compLink}. Thus the above calculations give an alternative proof of the Levine's result. 
\end{remark}

\begin{remark} \label{Graff's-result}
In \cite{Gra}, Graff also gave the same presentation of $\mathcal{L}_4$ by calculating the action of partial conjugations independently. 
\end{remark}
\fi

\begin{remark}
The functions of 12 integer variables which 
 are invariants under the relations of Table \ref{ModAct4-compSF} are invariants of $\mathcal{L}_4$. Some examples are given in \cite{Le2} and \cite{KM2} for subsets of $\mathcal{L}_4$.
%As mentioned in \cite{L} and \cite{KM2}, the Milnor homotopy invariants $\overline{\mu}_L(I)$ are obtained via Table \ref{ModAct4-compSF}. 
\end{remark}

\subsection{5-component case}
We also calculate the actions of $\overline{x}_{ij}$ for the 5-component case. In this subsection, some tables of the calculations are moved to Appendix \ref{result-tables} since they occupy spaces. 
The results of the actions of $\overline{x}_{ij}$ are in the Table \ref{Act5-compSF}, where the entries are the difference of the number of the claspers. 
Note that $y_{ij}$ are invariants (which are equivalent to the linking numbers) and we omit them in the table. 
%The example of calculation for $\overline{x}_{ij}$ are shown in Appendix \ref{example-calc}. 
\par
We also calculate the actions of the commutators $[\overline{x}_{ij}, \overline{x}_{kl}]$ in Table \ref{ActCom5-compSF} and $[\overline{x}_{ij}, [\overline{x}_{kl}, \overline{x}_{st}]]$ in Table \ref{Act3Com5-compSF}. 
The results of $[\overline{x}_{ij}, \overline{x}_{kl}]$ can be simplified by using $[\overline{x}_{ij}, [\overline{x}_{kl}, \overline{x}_{st}]]$. The simplified first commutators $[\overline{x}_{ij}, \overline{x}_{kl}]'$ are shown in Table \ref{ModComm5-compSF}. Then we can simplify the actions of $\overline{x}_{ij}$ by using $[\overline{x}_{ij}, \overline{x}_{kl}]'$ and $[\overline{x}_{ij}, [\overline{x}_{kl}, \overline{x}_{st}]]$. The simplified partial conjugations $\overline{x}'_{ij}$ are shown in Table \ref{ModAct5-compSF}. For example, 
\begin{align*}
\overline{x}'_{12}=&
[\overline{x}_{34}, [\overline{x}_{14}, \overline{x}_{24}]]^{y_{25}y_{35}-y_{24}y_{34}-y_{23}y_{35}+y_{23}y_{34}}
\circ[\overline{x}_{31}, [\overline{x}_{14}, \overline{x}_{24}]]^{-y_{23}y_{25}+y_{23}y_{24}}\\
&\circ[\overline{x}_{52}, [\overline{x}_{13}, \overline{x}_{23}]]^{y_{24}y_{25}-y_{23}y_{25}}
\circ[\overline{x}_{43}, [\overline{x}_{13}, \overline{x}_{23}]]^{y_{24}+y_{24}y_{34}-y_{23}y_{34}+y_{23}y_{24}}\\
&\circ[\overline{x}_{41}, [\overline{x}_{13}, \overline{x}_{23}]]^{-y_{24}y_{25}+y_{23}y_{24}}
\circ[\overline{x}_{54}, [\overline{x}_{12}, \overline{x}_{42}]]^{y_{145}}\\
&\circ[\overline{x}_{53}, [\overline{x}_{12}, \overline{x}_{32}]]^{y_{135}}
\circ[\overline{x}_{42}, [\overline{x}_{12}, \overline{x}_{32}]]^{y_{23}y_{24}}
\circ\overline{x}_{12}.
\end{align*}
In the equation we only use the commutators $[\overline{x}_{ij}, [\overline{x}_{kl}, \overline{x}_{st}]]$ of degree 3. However, in general, to simplify $\overline{x}_{ij}$, we also use modified commutators $[\overline{x}_{ij}, \overline{x}_{kl}]'$ of degree 2. 

\if0
By using the degree 3 commutators in Table \ref{Act3Com5-compS}, we can simplify $[\overline{x}_{ij}, \overline{x}_{kl}]$ in Table \ref{Act2Com5-compSF} to $[\overline{x}_{ij}, \overline{x}_{kl}]'$ in Table \ref{ActMod2Com5-compSF}.
\fi

\begin{remark}
In Table \ref{ActCom5-compSF}, the other commutators $[\overline{x}_{ij}, \overline{x}_{kl}]$ which do not appear the table are equal to the identity or to the ones in the table up to the relations of the commutators $[\overline{x}_{ij}, [\overline{x}_{kl}, \overline{x}_{st}]]$ in Table \ref{Act3Com5-compSF}. 
In Table \ref{Act3Com5-compSF}, the other commutators $[\overline{x}_{ij}, [\overline{x}_{kl}, \overline{x}_{st}]]$ which do not appear the table are equal to the identity or the one in the table. 
\end{remark}

\if0
Then, we can simplify $\overline{x}_{ij}$ in Table \ref{Act5-compSF} to the relation $\overline{x}'_{ij}$ in Table \ref{SmpfModAct5-compSF} by using the relations in Table \ref{ActMod2Com5-compSF} and Table \ref{Act3Com5-compSF}.

\begin{table}[htb] 
\begin{center}
   \caption{The table of relations in Theorem \ref{mainthm}} \label{SmpfModAct5-compSF}
  \begin{tabular}{|c|c|c|c|c|c|c|} \hline
      & $f_1$ & $f_2$ & $f_3$ & $f_4$ & $t_1$ & $t_2$ \\ \hline 
    $\psi_{21}$ & 0 & 0 & $c_5$ & $-c_1$ & $f_1$ & $0$ \\ \hline 
    $\psi_{41}$ & 0 & $c_6$ & $-c_5$ & $0$ & $0$ & $-f_1$ \\ \hline 
    $\psi_{12}$ & 0 & 0 & $-c_4$ & $c_2$ & $f_2$ & $0$ \\ \hline 
    $\psi_{32}$ & $c_6$ & 0 & $0$ & $-c_2$ & $0$ & $-f_2$ \\ \hline 
    $\psi_{43}$ & $c_5$ & $-c_4$ & $0$ & $0$ & $f_3$ & $0$ \\ \hline 
    $\psi_{23}$ & $-c_5$ & 0 & $0$ & $c_3$ & $0$ & $-f_3$ \\ \hline 
    $\psi_{34}$ & $-c_1$ & $c_2$ & $0$ & $0$ & $f_4$ & $0$ \\ \hline 
    $\psi_{14}$ & 0 & $-c_2$ & $c_3$ & $0$ & $0$ & $-f_4$ \\ \hline 
  \end{tabular} \\
\end{center}
\end{table}
\fi

Remark that, when we replace $\overline{x}_{ij}$ to $\overline{x}'_{ij}$, the actions of the commutators in Table \ref{Act3Com5-compSF} do not change; $[\overline{x}_{ij}, [\overline{x}_{kl}, \overline{x}_{st}]]=[\overline{x}'_{ij}, [\overline{x}'_{kl}, \overline{x}'_{st}]]$, and the commutators $[\overline{x}'_{ij}, \overline{x}'_{kl}]$ can be simplified to the corresponding actions in Table \ref{ModComm5-compSF}. 
%; $[\overline{x}_{ij}, \overline{x}_{kl}]'=[\overline{x}'_{ij}, \overline{x}'_{kl}]'$. 
Since, $\displaystyle \prod_{\substack{j; j \neq i}} \overline{x}'_{ij}={\rm id}$ holds for each $1\leq i\leq 5$, we can erase 5 relations. We erase $\{\overline{x}'_{15}, \overline{x}'_{25}, \overline{x}'_{35}, \overline{x}'_{45}, \overline{x}'_{54}\}$. From the changes of $y_{ijk}$ in Table \ref{ModAct5-compSF}, we see that no more relations can be erased. For example, if $\overline{x}'_{12}$ is erased additionally, we can not change $y_{125}$ by $y_{25}$ by using the other $\overline{x}'_{ij}$. 
Thus the smallest number of the relations is 15. 
Note that, for the commutators in Tables \ref{ActCom5-compSF} and \ref{Act3Com5-compSF}, we do not use $\overline{x}_{ij}$ corresponding to the erased relations. 
Then we have a presentation of $\mathcal{L}_5$. 

\begin{theorem} \label{Rep5-compLink} 
The set $\mathcal{L}_5$ of link-homotopy classes of 5-component links has a presentation which is 36-tuples of integers (i.e. the numbers of claspers) modulo the relations $X_5$ generated by the 15 relations 
$\{\overline{x}'_{12}, \overline{x}'_{13}, \overline{x}'_{14}, 
\overline{x}'_{21}, \overline{x}'_{23}, \overline{x}'_{24}, 
\overline{x}'_{31}, \overline{x}'_{32}, \overline{x}'_{34}, 
\overline{x}'_{41}, \overline{x}'_{42}, \overline{x}'_{43},
\overline{x}'_{51}, \overline{x}'_{52}, \overline{x}'_{53}
\}$ 
in Table \ref{ModAct5-compSF}. Namely, 
$$\mathcal{L}_5=\mathbb{Z}^{36}/{X_5}.$$
\end{theorem}

\begin{remark}
The functions of 36 integer variables which 
 are invariants under the relations of Table \ref{ModAct5-compSF} are invariants of $\mathcal{L}_5$. 
%As mentioned in \cite{L} and \cite{KM2}, the Milnor homotopy invariants $\overline{\mu}_L(I)$ are obtained via Table \ref{ModAct4-compSF}. 
\end{remark}

%In both 4- and 5-component cases, for modified partial conjugations $\overline{x}'_{ij}$ in Table \ref{ModAct4-compSF} and \ref{ModAct5-compSF}, 
%$$\displaystyle \prod_{j \neq i} \overline{x}'_{ij}=id$$
%holds. By this relation, we deduce the number of the generators of partial conjugations by $4$ or $5$ respectively. We conjecture that it holds in general.   

\subsection{Conjecture}
In the both 4- and 5-component cases, $\displaystyle \prod_{j; j \neq i} \overline{x}_{ij}={\rm id}$ holds up to commutators for each $1\leq i \leq n$. We can check the similar relations for the 2- and 3-component cases. We state a conjecture. 

\begin{conjecture}
For $n$-component string links, $\displaystyle \prod_{j; j \neq i} \overline{x}_{ij}={\rm id}$ holds up to commutators for each $1\leq i \leq n$. Thus the number of the generators of partial conjugations is reduced to $n(n-2)$. 
\end{conjecture}
%\newpage
%%%%%%%%%%%%%%%%%% Actions of the simplified partial conjugation  (start) %%%%%%%%%%%%%%%%%%%%%%%%%%%

\begin{table}[phtb] 
   \caption{Simplified partial conjugations $\overline{x}'_{ij}$ for 5-component string links} \label{ModAct5-compSF}
\begin{center}
\begin{tabular}{|c|l|l|l|l|} \hline
      & $\overline{x}'_{12}$ & $\overline{x}'_{13}$ & $\overline{x}'_{14}$ & $\overline{x}'_{15}$ \\ \hline 
    $y_{123}$ & $y_{23}$ & $-y_{23}$ & $0$ & $0$ \\ \hline 
    $y_{124}$ & $y_{24}$ & $0$ & $-y_{24}$ & $0$\\ \hline
    $y_{134}$ & $0$ & $y_{34}$ & $-y_{34}$ & $0$\\ \hline
    $y_{125}$ & $y_{25}$ & $0$ & $0$ & $-y_{25}$\\ \hline
    $y_{135}$ & $0$ & $y_{35}$ & $0$ & $-y_{35}$\\ \hline
    $y_{145}$ & $0$ & $0$ & $y_{45}$ & $-y_{45}$\\ \hline
    $y_{234}$ & $0$ & $0$ & $0$ & $0$\\
\hline
  \end{tabular}
\end{center}
\end{table}

\begin{table}[htb] 
\begin{center}
  \begin{tabular}{|c|l|l|l|l|l|} \hline
      & $\overline{x}'_{21}$ & $\overline{x}'_{23}$ & $\overline{x}'_{24}$ & $\overline{x}'_{25}$\\ \hline 
    $y_{123}$ & $-y_{13}$ & $y_{13}$ & $0$ & $0$\\ \hline 
    $y_{124}$ & $-y_{14}$ & $0$ & $y_{14}$ & $0$\\ \hline
    $y_{134}$ & $0$ & $0$ & $0$ & $0$\\ \hline
    $y_{125}$ & $-y_{15}$ & $0$ & $0$ & $y_{15}$\\ \hline
    $y_{135}$ & $0$ & $0$ & $0$ & $0$\\ \hline
    $y_{145}$ & $0$ & $0$ & $0$ & $0$\\ \hline
    $y_{234}$ & $0$ & $y_{34}$ & $-y_{34}$ & $0$\\
\hline
  \end{tabular} 
\end{center}
\end{table}

\begin{table}[htb] 
\begin{center}
  \begin{tabular}{|c|l|l|l|l|} \hline
      & $\overline{x}'_{31}$ & $\overline{x}'_{32}$ & $\overline{x}'_{34}$ & $\overline{x}'_{35}$ \\ \hline 
    $y_{123}$ & $y_{12}$ & $-y_{12}$ & $0$ & $0$\\ \hline 
    $y_{124}$ & $0$ & $0$ & $0$ & $0$\\ \hline
    $y_{134}$ & $-y_{14}$ & $0$ & $y_{14}$ & $0$\\ \hline
    $y_{125}$ & $0$ & $0$ & $0$ & $0$\\ \hline
    $y_{135}$ & $-y_{15}$ & $0$ & $0$ & $y_{15}$\\ \hline
    $y_{145}$ & $0$ & $0$ & $0$ & $0$\\ \hline
    $y_{234}$ & $0$ & $-y_{24}$ & $y_{24}$ & $0$\\
\hline
  \end{tabular} 
\end{center}
\end{table}

\begin{table}[htb] 
\begin{center}
  \begin{tabular}{|c|l|l|l|l|} \hline
      & $\overline{x}'_{41}$ & $\overline{x}'_{42}$ & $\overline{x}'_{43}$ & $\overline{x}'_{45}$\\ \hline 
    $y_{123}$ & $0$ & $0$ & $0$ & $0$\\ \hline 
    $y_{124}$ & $y_{12}$ & $-y_{12}$ & $0$ & $0$\\ \hline
    $y_{134}$ & $y_{13}$ & $0$ & $-y_{13}$ & $0$\\ \hline
    $y_{125}$ & $0$ & $0$ & $0$ & $0$\\ \hline
    $y_{135}$ & $0$ & $0$ & $0$ & $0$\\ \hline
    $y_{145}$ & $-y_{15}$ & $0$ & $0$ & $y_{15}$\\ \hline
    $y_{234}$ & $0$ & $y_{23}$ & $-y_{23}$ & $0$\\
\hline
  \end{tabular} 
\end{center}
\end{table}

\begin{table}[htb] 
\begin{center}
  \begin{tabular}{|c|l|l|l|l|} \hline
      & $\overline{x}'_{51}$ & $\overline{x}'_{52}$ & $\overline{x}'_{53}$ & $\overline{x}'_{54}$\\ \hline 
    $y_{123}$ & $0$ & $0$ & $0$ & $0$\\ \hline 
    $y_{124}$ & $0$ & $0$ & $0$ & $0$\\ \hline
    $y_{134}$ & $0$ & $0$ & $0$ & $0$\\ \hline
    $y_{125}$ & $y_{12}$ & $-y_{12}$ & $0$ & $0$\\ \hline
    $y_{135}$ & $y_{13}$ & $0$ & $-y_{13}$ & $0$\\ \hline
    $y_{145}$ & $y_{14}$ & $0$ & $0$ & $-y_{14}$\\ \hline
    $y_{234}$ & $0$ & $0$ & $0$ & $0$\\
\hline
  \end{tabular} 
\end{center}
\end{table}
%%%%%%%%%%%%%%%%%% Actions of the simplified partial conjugation  (end) %%%%%%%%%%%%%%%%%%%%%%%%%%%

%%%%%%%%%%%%%%%%%% Actions of the simplified first commutators  (start) %%%%%%%%%%%%%%%%%%%%%%%%%%%
\begin{table}[htb] 
\begin{center}
   \caption{Modified commutators of $\overline{x}_{ij}$ for 5-component string links} \label{ModComm5-compSF}
  \begin{tabular}{|c|l|l|l|l|l|l|l|} \hline
      & $[\overline{x}_{21}, \overline{x}_{31}]'$ & $[\overline{x}_{21}, \overline{x}_{41}]'$ & $[\overline{x}_{21}, \overline{x}_{51}]'$ & $[\overline{x}_{31}, \overline{x}_{41}]'$ & $[\overline{x}_{31}, \overline{x}_{51}]'$  \\ \hline 
    $y_{1234}$ & $y_{14}$ & $0$ & $0$ & $-y_{12}$ & $0$ \\ \hline 
    $y_{1324}$ & $-y_{14}$ & $-y_{13}$ & $0$ & $0$ & $0$ \\ \hline
    $y_{1235}$ & $y_{15}$ & $0$ & $0$ & $0$ & $-y_{12}$ \\ \hline
    $y_{1245}$ & $0$ & $y_{15}$ & $0$ & $0$ & $0$ \\ \hline
    $y_{1325}$ & $-y_{15}$ & $0$ & $-y_{13}$ & $0$ & $0$ \\ \hline
    $y_{1345}$ & $0$ & $0$ & $0$ & $y_{15}$ & $0$ \\ \hline
    $y_{1425}$ & $0$ & $-y_{15}$ & $-y_{14}$ & $0$ & $0$ \\ \hline
    $y_{1435}$ & $0$ & $0$ & $0$ & $-y_{15}$ & $-y_{14}$ \\ \hline
    $y_{2345}$ & $0$ & $0$ & $0$ & $0$ & $0$ \\ \hline
    $y_{2435}$ & $0$ & $0$ & $0$ & $0$ & $0$ \\ \hline\hline
    $y_{12345}$ & $y_{145}$ & $0$ & $0$ & $0$ & $0$ \\ \hline 
    $y_{12435}$ & $0$ & $y_{135}$ & $0$ & $0$ & $- y_{124}$ \\ \hline
    $y_{13245}$ & $-y_{145}$ & $0$ & $0$ & $0$ & $0$ \\ \hline
    $y_{13425}$ & $0$ & $0$ & $- y_{134}$ & $y_{125}$ & $0$ \\ \hline
    $y_{14235}$ & $0$ & $- y_{135}$ & $0$ & $0$ & $y_{124}$ \\ \hline
    $y_{14325}$ & $0$ & $0$ & $y_{134}$ & $- y_{125}$& $0$ \\ \hline
  \end{tabular} \\
\end{center}

\begin{center}
  \begin{tabular}{|c|l|l|l|l|l|l|l|l|l|} \hline
    & $[\overline{x}_{41}, \overline{x}_{51}]'$ & $[\overline{x}_{12}, \overline{x}_{32}]'$  & $[\overline{x}_{12}, \overline{x}_{42}]'$ & $[\overline{x}_{12}, \overline{x}_{52}]'$ & $[\overline{x}_{32}, \overline{x}_{42}]'$  \\ \hline 
    $y_{1234}$ & $0$ & $0$ & $y_{23}$ & $0$ & $y_{12}$ \\ \hline  
    $y_{1324}$ & $0$ & $-y_{24}$ & $-y_{23}$ & $0$ & $0$ \\ \hline
    $y_{1235}$ & $0$ & $0$ & $0$ & $y_{23}$ & $0$\\ \hline
    $y_{1245}$ & $-y_{12}$ & $0$ & $0$ & $y_{24}$ & $0$\\ \hline
    $y_{1325}$ & $0$ & $-y_{25}$ & $0$ & $-y_{23}$ & $0$\\ \hline
    $y_{1345}$ & $-y_{13}$ & $0$ & $0$ & $0$ & $0$\\ \hline
    $y_{1425}$ & $0$ & $0$ & $-y_{25}$ & $-y_{24}$ & $0$\\ \hline
    $y_{1435}$ & $0$ & $0$ & $0$ & $0$ & $0$\\ \hline
    $y_{2345}$ & $0$ & $0$ & $0$ & $0$ & $y_{25}$\\ \hline
    $y_{2435}$ & $0$ & $0$ & $0$ & $0$ & $-y_{25}$ \\ \hline\hline
    $y_{12345}$ & $- y_{123}$ & $0$ & $0$ & $y_{234}$ & $y_{125}$\\ \hline 
    $y_{12435}$ & $0$ & $0$ & $0$ & $- y_{234}$ & $- y_{125}$\\ \hline
    $y_{13245}$ & $y_{123}$ & $- y_{245}$ & $0$ & $0$ & $0$\\ \hline
    $y_{13425}$ & $0$ & $0$ & $0$ & $- y_{234}$ & $- y_{125}$ \\ \hline
    $y_{14235}$ & $0$ & $0$ & $- y_{235}$ & $0$ & $0$ \\ \hline
    $y_{14325}$ & $0$ & $0$ & $0$ & $y_{234}$ & $y_{125}$ \\ \hline
  \end{tabular} \\
\end{center}
\end{table}

\begin{table}[htb] 
\begin{center}
  \begin{tabular}{|c|l|l|l|l|l|l|l|} \hline
       & $[\overline{x}_{32}, \overline{x}_{52}]'$ & $[\overline{x}_{42}, \overline{x}_{52}]'$ & $[\overline{x}_{13}, \overline{x}_{23}]'$ & $[\overline{x}_{13}, \overline{x}_{43}]'$ & $[\overline{x}_{13}, \overline{x}_{53}]'$ \\ \hline 
    $y_{1234}$ & $0$ & $0$ & $-y_{34}$ & $-y_{23}$ & $0$ \\ \hline  
    $y_{1324}$ & $0$ & $0$ & $0$ & $y_{23}$ & $0$ \\ \hline
    $y_{1235}$ & $y_{12}$ & $0$ & $-y_{35}$ & $0$ & $-y_{23}$ \\ \hline
    $y_{1245}$ & $0$ & $y_{12}$ & $0$ & $0$ & $0$ \\ \hline
    $y_{1325}$ & $0$ & $0$ & $0$ & $0$ & $y_{23}$ \\ \hline
    $y_{1345}$ & $0$ & $0$ & $0$ & $0$ & $y_{34}$ \\ \hline
    $y_{1425}$ & $0$ & $0$ & $0$ & $0$ & $0$ \\ \hline
    $y_{1435}$ & $0$ & $0$ & $0$ & $-y_{35}$ & $-y_{34}$ \\ \hline
    $y_{2345}$ & $0$ & $-y_{23}$ & $0$ & $0$ & $0$ \\ \hline
    $y_{2435}$ & $-y_{24}$ & $0$ & $0$ & $0$ & $0$ \\ \hline\hline
    $y_{12345}$ & $0$ & $0$ & $- y_{345}$ & $0$ & $0$ \\ \hline 
    $y_{12435}$ & $0$ & $0$ & $0$ & $0$ & $y_{234}$ \\ \hline
    $y_{13245}$ & $0$ & $- y_{123}$ & $0$ & $0$ & $- y_{234}$ \\ \hline
    $y_{13425}$ & $0$ & $0$ & $0$ & $0$ & $y_{234}$ \\ \hline
    $y_{14235}$ & $- y_{124}$ & $0$ & $0$ & $0$ & $- y_{234}$ \\ \hline
    $y_{14325}$ & $0$ & $0$ & $0$ & $y_{235}$ & $0$ \\ \hline
  \end{tabular} \\
\end{center}

\begin{center}
  \begin{tabular}{|c|l|l|l|l|l|l|} \hline
      & $[\overline{x}_{23}, \overline{x}_{43}]'$ & $[\overline{x}_{23}, \overline{x}_{53}]'$ & $[\overline{x}_{14}, \overline{x}_{24}]'$ & $[\overline{x}_{14}, \overline{x}_{34}]'$ & $[\overline{x}_{24}, \overline{x}_{34}]'$ \\ \hline  
    $y_{1234}$ & $0$ & $0$ & $y_{34}$ & $0$ & $-y_{14}$ \\ \hline 
    $y_{1324}$ & $y_{13}$ & $0$ & $0$ & $y_{24}$ & $y_{14}$ \\ \hline
    $y_{1235}$ & $0$ & $0$ & $0$ & $0$ & $0$\\ \hline
    $y_{1245}$ & $0$ & $0$ & $-y_{45}$ & $0$ & $0$\\ \hline
    $y_{1325}$ & $0$ & $y_{13}$ & $0$ & $0$ & $0$\\ \hline
    $y_{1345}$ & $0$ & $0$ & $0$ & $-y_{45}$ & $0$\\ \hline
    $y_{1425}$ & $0$ & $0$ & $0$ & $0$ & $0$\\ \hline
    $y_{1435}$ & $0$ & $0$ & $0$ & $0$ & $0$\\ \hline
    $y_{2345}$ & $0$ & $y_{34}$ & $0$ & $0$ & $-y_{45}$\\ \hline
    $y_{2435}$ & $-y_{35}$ & $-y_{34}$ & $0$ & $0$ & $0$\\ \hline\hline
    $y_{12345}$ & $0$ & $0$ & $0$ & $0$ & $- y_{145}$\\ \hline 
    $y_{12435}$ & $- y_{135}$ & $0$ & $y_{345}$ & $0$ & $0$\\ \hline 
    $y_{13245}$ & $y_{135}$ & $0$ & $0$ & $0$ & $y_{145}$\\ \hline
    $y_{13425}$ & $- y_{135}$ & $0$ & $0$ & $y_{245}$ & $0$\\ \hline
    $y_{14235}$ & $y_{135}$ & $0$ & $0$ & $0$ & $y_{145}$\\ \hline
    $y_{14325}$ & $0$ & $- y_{134}$ & $0$ & $0$ & $- y_{145}$\\ \hline
  \end{tabular} \\
\end{center}
\end{table}
%%%%%%%%%%%%%%%%%% Actions of the simplified first commutators  (end) %%%%%%%%%%%%%%%%%%%%%%%%%%%

%%%%%%%%%%%%%%%%%% Actions of the second commutators  (start) %%%%%%%%%%%%%%%%%%%%%%%%%%%
\begin{table}[htb] 
\begin{center}
   \caption{3-commutators for 5-component string links} \label{Act3Com5-compSF}
  \begin{tabular}{|c|c|c|c|c|c|} \hline
        & $[\overline{x}_{41},[\overline{x}_{21},\overline{x}_{31}]]$ 
        & $[\overline{x}_{43},[\overline{x}_{21},\overline{x}_{31}]]$ 
        & $[\overline{x}_{52},[\overline{x}_{21},\overline{x}_{31}]]$ 
        & $[\overline{x}_{53},[\overline{x}_{21},\overline{x}_{31}]]$ 
        & $[\overline{x}_{52},[\overline{x}_{21},\overline{x}_{41}]]$\\ \hline 
    $y_{12345}$ & $y_{15}$ & $-y_{15}$ & $0$ & $0$ & $0$\\ \hline 
    $y_{12435}$ & $0$ & $y_{15}$ & $0$ & $0$ & $0$\\ \hline 
    $y_{13245}$ & $-y_{15}$ & $0$ & $0$ & $0$ & $0$\\ \hline 
    $y_{13425}$ & $0$ & $y_{15}$ & $0$ & $0$ & $-y_{13}$\\ \hline 
    $y_{14235}$ & $-y_{15}$ & $0$ & $0$ & $y_{14}$ & $0$\\ \hline 
    $y_{14325}$ & $y_{15}$ & $-y_{15}$ & $-y_{14}$ & $0$ & $0$\\ \hline 
  \end{tabular} \\
\end{center}

\begin{center}
  \begin{tabular}{|c|c|c|c|c|c|} \hline
        & $[\overline{x}_{54},[\overline{x}_{21},\overline{x}_{41}]]$ 
        & $[\overline{x}_{53},[\overline{x}_{31},\overline{x}_{41}]]$ 
        & $[\overline{x}_{54},[\overline{x}_{31},\overline{x}_{41}]]$ 
        & $[\overline{x}_{41},[\overline{x}_{12},\overline{x}_{32}]]$ 
        & $[\overline{x}_{42},[\overline{x}_{12},\overline{x}_{32}]]$\\ \hline 
    $y_{12345}$ & $0$ & $0$ & $y_{12}$ & $0$ & $0$\\ \hline 
    $y_{12435}$ & $0$ & $-y_{12}$ & $0$ & $0$ & $0$\\ \hline 
    $y_{13245}$ & $y_{13}$ & $0$ & $0$ & $0$ & $0$\\ \hline 
    $y_{13425}$ & $0$ & $0$ & $0$ & $0$ & $-y_{25}$\\ \hline 
    $y_{14235}$ & $0$ & $0$ & $0$ & $0$ & $0$\\ \hline 
    $y_{14325}$ & $0$ & $0$ & $0$ & $y_{25}$ & $0$\\ \hline 
  \end{tabular} \\
\end{center}

\begin{center}
  \begin{tabular}{|c|c|c|c|c|c|} \hline
        & $[\overline{x}_{52},[\overline{x}_{12},\overline{x}_{32}]]$ 
        & $[\overline{x}_{53},[\overline{x}_{12},\overline{x}_{32}]]$ 
        & $[\overline{x}_{52},[\overline{x}_{12},\overline{x}_{42}]]$ 
        & $[\overline{x}_{54},[\overline{x}_{12},\overline{x}_{42}]]$ 
        & $[\overline{x}_{41},[\overline{x}_{13},\overline{x}_{23}]]$\\ \hline 
    $y_{12345}$ & $0$ & $0$ & $0$ & $-y_{23}$ & $0$\\ \hline 
    $y_{12435}$ & $0$ & $-y_{24}$ & $0$ & $0$ & $0$\\ \hline 
    $y_{13245}$ & $y_{24}$ & $0$ & $0$ & $y_{23}$ & $0$\\ \hline 
    $y_{13425}$ & $-y_{24}$ & $0$ & $0$ & $0$ & $0$\\ \hline 
    $y_{14235}$ & $0$ & $y_{24}$ & $y_{23}$ & $0$ & $y_{35}$\\ \hline 
    $y_{14325}$ & $0$ & $0$ & $-y_{23}$ & $0$ & $0$\\ \hline 
  \end{tabular} \\
\end{center}

\begin{center}
  \begin{tabular}{|c|c|c|c|c|c|} \hline
        & $[\overline{x}_{43},[\overline{x}_{13},\overline{x}_{23}]]$ 
        & $[\overline{x}_{52},[\overline{x}_{13},\overline{x}_{23}]]$ 
        & $[\overline{x}_{53},[\overline{x}_{13},\overline{x}_{23}]]$ 
        & $[\overline{x}_{31},[\overline{x}_{14},\overline{x}_{24}]]$ 
        & $[\overline{x}_{34},[\overline{x}_{14},\overline{x}_{24}]]$\\ \hline 
    $y_{12345}$ & $0$ & $0$ & $y_{34}$ & $0$ & $-y_{45}$\\ \hline 
    $y_{12435}$ & $-y_{35}$ & $0$ & $-y_{34}$ & $0$ & $0$\\ \hline 
    $y_{13245}$ & $0$ & $0$ & $0$ & $y_{45}$ & $0$\\ \hline 
    $y_{13425}$ & $0$ & $-y_{34}$ & $0$ & $0$ & $0$\\ \hline 
    $y_{14235}$ & $0$ & $0$ & $0$ & $0$ & $0$\\ \hline 
    $y_{14325}$ & $0$ & $y_{34}$ & $0$ & $0$ & $0$\\ \hline 
  \end{tabular} \\
\end{center}
\end{table}
%%%%%%%%%%%%%%%%%% Actions of the second commutators  (end) %%%%%%%%%%%%%%%%%%%%%%%%%%%

%%%%%%%%%%%%%%%%%%%%%%%%%%%%%%%%%%
\section{Algorithm} 
%%%%%%%%%%%%%%%%%%%%%%%%%%%%%%%%%%

In \cite{HL}, Habegger and Lin showed that there is an algorithm which determines whether given two links are link-homotopic or not by using the actions of the partial conjugations. However, in the paper, the actions are not calculated explicitly. As an application of Section \ref{CalcParConj}, we can run the algorithm. 
\par
Let $L$ and $L'$ be two $n$-component links. By cutting them at appropriate points, we have string links corresponding to them. Then we transform them to Meilhan and Yasuhara's canonical form \cite{MY}. Let $Y_i$ (resp. $Y'_i$) be the sequence of the numbers of $i$-th degree claspers of $L$ (resp. $L'$). For example, for 4-component links $L$ (resp. $L'$), $Y_1=(y_{12}, y_{13}, y_{14}, y_{23}, y_{24}, y_{34})$, $Y_2=(y_{123}, y_{124}, y_{134}, y_{234})$ and $Y_3=(y_{1234}, y_{1324})$ (resp. $Y'_1=(y'_{12}, y'_{13}, y'_{14}, y'_{23}, y'_{24}, y'_{34})$, $Y'_2=(y'_{123}, y'_{124}, y'_{134}, y'_{234})$ and $Y'_3=(y'_{1234}, y'_{1324})$).  
 Then, for the links $L=(Y_1, Y_2, \dots, Y_{n-1})$ and $L'=(Y'_1, Y'_2, \dots, Y'_{n-1})$, Habegger and Lin's algorithm is translated in our notations as follows. We will change $Y_i$ to $Y'_i$ in order. Note that the partial conjugations do not change the numbers $y_{ij}$ of $C_1$-trees (they are invariants; the linking numbers). 

\begin{enumerate}
\renewcommand{\labelenumi}{\arabic{enumi}.}
\item Check $Y_1=Y'_1$. If it holds, then go to Step 2, otherwise $L$ and $L'$ are not link-homotopic. 
\item Let $\mathcal{S}_1$ be the set of partial conjugations. Find a relation $\Psi_{1} \in \mathcal{S}_1$ which satisfies $\Psi_{1}(Y_2)=Y'_2$ for $L$. If there is such a relation, then replace 
$$L=(Y_1, Y_2, Y_3, Y_4, \dots)$$ 
to 
\begin{align*}
L_1&=\Psi_1(L)\\ 
&= (Y_1, \Psi_1(Y_2), \Psi_1(Y_3), \Psi_1(Y_4), \dots)\\
&= (Y'_1, Y'_2, \Psi_1(Y_3), \Psi_1(Y_4), \dots)
\end{align*} 
and go to Step 3, otherwise $L$ and $L'$ are not link-homotopic. 
%\item Let $\mathcal{S}_1$ be the set of partial conjugations. 
%$\Psi_1=\prod_{\overline{x}_{ij}\in X_n} {\overline{x}'}_{ij}^{a_{ij}}$, where $X_n$ is the set of the partial conjugation in Theorems \ref{Rep4-compLink} and \ref{Rep5-compLink}, $a_{ij}\in \mathbb{Z}$ and the order of the product can be taken appropriately but is fixed. Find the exponents $a_{ij}$ satisfying $\Psi_1(Y_2)=Y'_2$. If there are such $a_{ij}$, then go to Step 3, otherwise $L$ and $L'$ are not link-homotopic. 
\item Let $\mathcal{S}_2$ be the set of the partial conjugations which do not change $Y'_2$ of $L_1$. Find a relation $\Psi_{2} \in \mathcal{S}_2$ which satisfies $\Psi_2\circ\Psi_1(Y_3)=Y'_3$ for $L_1$. If there is such a relation, then replace 
$$L_1=(Y'_1, Y'_2, \Psi_1(Y_3), \Psi_1(Y_4), \dots)$$ 
to 
\begin{align*}
L_2&=\Psi_2(L_1)\\
&=(Y'_1, Y'_2, \Psi_2\circ\Psi_1(Y_3), \Psi_2\circ\Psi_1(Y_4), \dots)\\
&=(Y'_1, Y'_2, Y'_3, \Psi_2\circ\Psi_1(Y_4), \dots)
\end{align*} 
and go to Step 4, otherwise $L$ and $L'$ are not link-homotopic. 
\item Let $\mathcal{S}_3$ be the set of the partial conjugations which do not change $Y'_2$ and $Y'_3$ of $L_2$. Find a relation $\Psi_{3} \in \mathcal{S}_3$ which satisfies $\Psi_3\circ\Psi_2\circ\Psi_1(Y_4)=Y'_4$ for $L_2$. If there is such a relation, then replace 
$$L_2=(Y'_1, Y'_2, Y'_3, \Psi_2\circ\Psi_1(Y_4), \dots)$$ 
to 
\begin{align*}
L_3&=\Psi_3(L_2)\\
&=(Y'_1, Y'_2, Y'_3, \Psi_3\circ\Psi_2\circ\Psi_1(Y_4), \dots)\\
&=(Y'_1, Y'_2, Y'_3, Y'_4, \dots)
\end{align*}
and go to the next step, otherwise $L$ and $L'$ are not link-homotopic. 
\item Repeat the operations until $Y_{n-1}$ changes to $Y'_{n-1}$. Then,  $L'=\Psi_{n-1}\circ\cdots\circ\Psi_1(L)$ and $L$ and $L'$ are link-homotopic. 
If it stops before that, $L$ and $L'$ are not link-homotopic. 
\end{enumerate}

\begin{remark} \label{algorithm-rem}
For the 4- and 5-component cases, the algorithm goes as follows. 
\par
The set $\mathcal{S}_1$ is generated by $X_n$ in Theorems \ref{Rep4-compLink} and \ref{Rep5-compLink} for the 4- and 5-component cases respectively. 
In Step 2, $\Psi_{1}$ can be presented as $\prod_{\overline{x}'_{ij}\in X_n} {\overline{x}'}_{ij}^{a_{ij}}$, where $a_{ij}\in \mathbb{Z}$. 
To find $\Psi_{1}$ such that $\Psi_{1}(Y_2)=Y'_2$, it is needed to find $a_{ij}$ satisfying
\begin{equation}
\prod_{\overline{x}'_{ij}\in X_n} {\overline{x}'}_{ij}^{a_{ij}}(y_{klm})=y'_{klm} \label{step1}
\end{equation}
for all pairs $(y_{klm}, y'_{klm})$ ($y_{klm}\in Y_2$, $y'_{klm}\in Y'_2$). This problem is a system of linear Diophantine equations and solvable if it has solutions. 
\par
In Step 3, the set $\mathcal{S}_2$ includes the commutators $[\overline{x}_{ij}, \overline{x}_{kl}]$ or the simplified commutators $[\overline{x}_{ij}, \overline{x}_{kl}]'$, see Tables \ref{ActCom4-compSF} or \ref{ModComm5-compSF}. The set $\mathcal{S}_2$ is generated by relations presented as 
$\prod_{\overline{x}'_{ij}\in X_n} {\overline{x}'}_{ij}^{b_{ij}}$, where $b_{ij}\in \mathbb{Z}$, which satisfy
$\prod_{\overline{x}'_{ij}\in X_n} {\overline{x}'}_{ij}^{b_{ij}}(Y'_2)=Y'_2$ and the (simplified) commutators in Table \ref{ActCom4-compSF} or \ref{ModComm5-compSF}.
To find the generators of $\mathcal{S}_2$ which are not the (simplified) commutators, it is needed to find $b_{ij}$ satisfying
\begin{equation*}
\prod_{\overline{x}'_{ij}\in X_n} {\overline{x}'}_{ij}^{b_{ij}}(y'_{klm})=y'_{klm} \label{step2-1}
\end{equation*}
for all $y'_{klm}\in Y'_2$. This problem is the homogeneous system of linear Diophantine equations corresponding to (\ref{step1}). 
\par
Let $\psi_i$ be the generators of $\mathcal{S}_2$. 
To find $\Psi_{2}$ such that $\Psi_2\circ\Psi_{1}(Y_3)=Y'_3$, it is needed to find $c_{i}$ satisfying
\begin{equation}
\prod_{\psi_i} {\psi}_{i}^{c_i}(y_{klmn})=y'_{klmn} \label{step2-2}
\end{equation}
for all pairs $(y_{klmn}, y'_{klmn})$ ($y_{klmn}\in Y_3$, $y'_{klmn}\in Y'_3$). This problem is a system of linear Diophantine equations. 
\par
For the 5-component case, the commutators $[\overline{x}_{ij},[\overline{x}_{kl},\overline{x}_{mn}]]$ do not change $Y'_3$, see Table \ref{Act3Com5-compSF}. 
Thus, in Step 4, the set $\mathcal{S}_3$ includes the commutators $[\overline{x}_{ij},[\overline{x}_{kl},\overline{x}_{mn}]]$. The set $\mathcal{S}_3$ is generated by relations presented as $\prod_{\psi_i} {\psi}_{i}^{d_i}$, where $d_i\in \mathbb{Z}$, which satisfy $\prod_{\psi_i} {\psi}_{i}^{d_i}(Y'_3)=Y'_3$ and $[\overline{x}_{ij},[\overline{x}_{kl},\overline{x}_{mn}]]$ in Table \ref{Act3Com5-compSF}. To find the generators of $\mathcal{S}_3$ which are not $[\overline{x}_{ij},[\overline{x}_{kl},\overline{x}_{mn}]]$, it is needed to find $d_{i}$ satisfying
\begin{equation*}
\prod_{\psi_i} {\psi}_{i}^{d_i}(y_{klmn})=y_{klmn}
\end{equation*}
for all $y'_{klmn}\in Y'_3$. This problem is the homogeneous system of linear Diophantine equations corresponding to (\ref{step2-2}). 
\end{remark}

\begin{remark}
In \cite{KM2}, the authors gave a presentation of $\mathcal{L}_4$ by using the ``tetrahedron'' standard form, which coincides with Theorem \ref{Rep4-compLink}, and showed examples of the algorithm. For explicit calculations of the algorithm, see Section 4 of the paper. 

\end{remark}

%%%%%%%%%%%%%%%%%%%%%%%%%%%%%%%%%%
\section{Test of the results} \label{test}
%%%%%%%%%%%%%%%%%%%%%%%%%%%%%%%%%%

The calculations of the partial conjugations $\overline{x}_{ij}$ are complicated and mistakes easily occur. 
In this section, we also calculate the actions of conjugations $(\overline{x}_j,x_j)_i$ in Figure \ref{conjugation} for the canonical forms and test the results in Table \ref{Act4-compSF} and \ref{Act5-compSF}. We abbreviate $(\overline{x}_j,x_j)_i$ to $cx_{ij}$ for simplicity. The conjugation $cx_{ij}$ $(1\leq i<j \leq n)$ for the canonical form $b$ of $n$-component string links is shown in Figure \ref{ConjFig}. We transform the shape of Figure \ref{ConjFig} back to the canonical form by using the relations in Section \ref{preparations}. Then the numbers of claspers change and the differences of the numbers from the previous canonical form present the action of the conjugation $cx_{ij}$. The results of the actions of $cx_{ij}$ are in Table \ref{ActConj4-compSF} for the 4-component case and Table \ref{ActConj5-compSF} for the 5-component case. We also use the notation $cx_{ji}$ $(1\leq i<j \leq n)$ with $cx_{ij}=cx_{ji}$.

\begin{figure}[ht] 
\raisebox{-36 pt}{\begin{overpic}[width=120 pt]{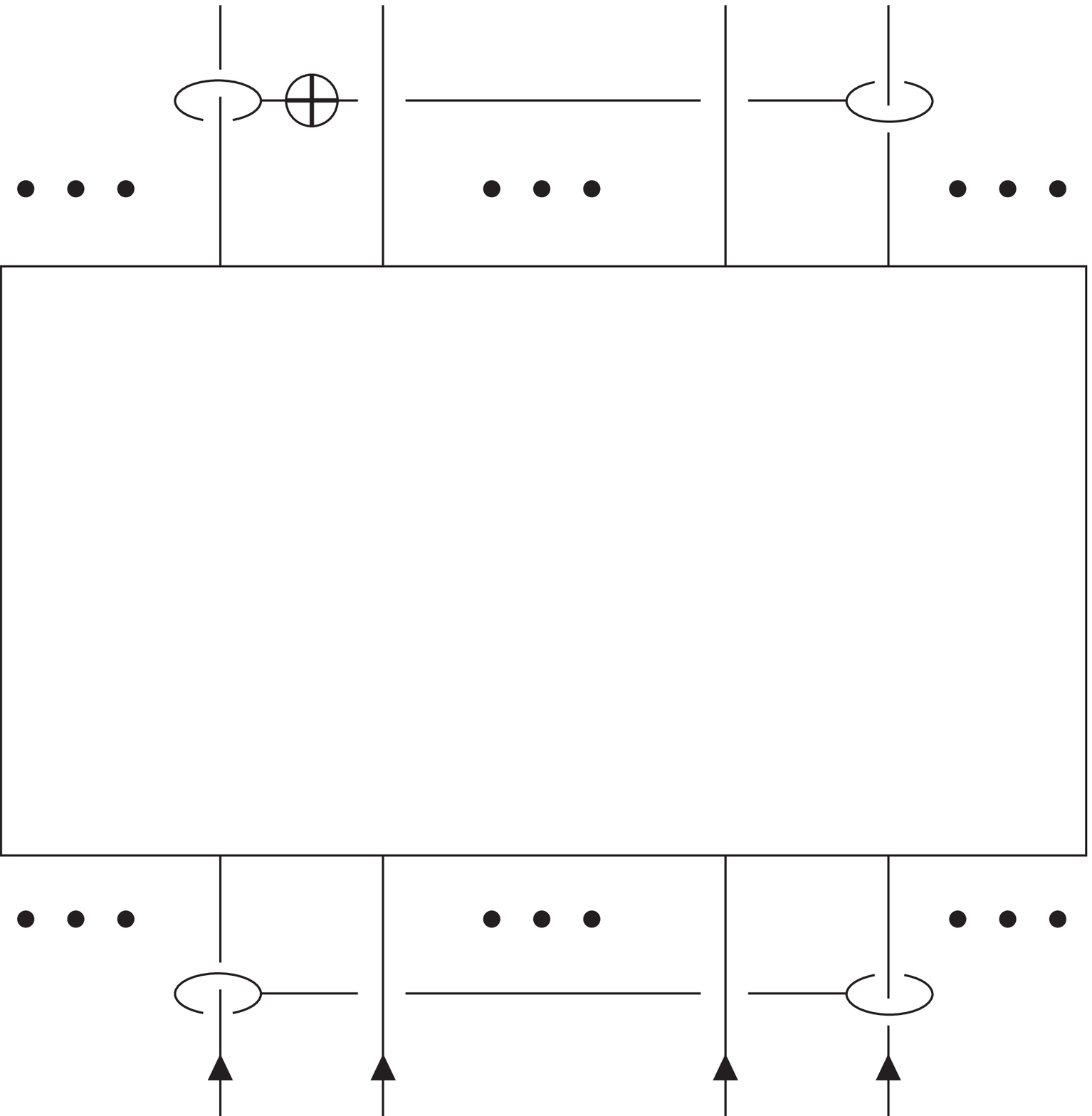}
\put(22,-13){$i$} \put(95,-13){$j$} \put(57,55){\Large $b$}
\end{overpic}}
\vspace{0.3cm}
\caption{Conjugation $cx_{ij}$ for the canonical form $b$} \label{ConjFig}
\end{figure}

\begin{proposition} \label{ConjRel}
For the actions of conjugations $cx_{ij}$ and partial conjugations $\overline{x}_{ij}$ for $n$-component string links, the following relations hold. 
\begin{enumerate}
\item $\displaystyle \prod_{k=2}^{n}CX_{k}= id$, where $\displaystyle CX_{k}= \prod_{i=1}^{k-1}cx_{ik}$.
\item $\displaystyle \prod_{\substack{i=1\\ i\neq j}}^{n}\overline{x}_{ij}=\prod_{\substack{i=1\\ i\neq j}}^{n}cx_{ij}$ for each $1\leq j\leq n$.
\end{enumerate}
\end{proposition}
\begin{proof}
(1) $\displaystyle \prod_{k=2}^{n}CX_{k}$ make left (resp. right) full-twists above (resp. under) a string link $\sigma$, see Figure \ref{ConjRelFig01} for the 4-component case. The twists cancel by isotopy and the action is the identity.
\par
(2) We show the relation for the 4-component case. See Figure \ref{ConjRelFig02}. The first figure presents $\displaystyle \prod_{\substack{i=1\\ i\neq 4}}^{4}\overline{x}_{i4}$ (i.e. $j=4$). Up to link-homotopy, it is transformed to the last figure which presents $\displaystyle \prod_{\substack{i=1\\ i\neq 4}}^{4}cx_{i4}$. For the cases $j\neq 4$, to see the relation, first move the $j$-th component to the outermost position under the other components, second transform the shape similar to the case $j=4$ and finally move back the $j$-th component to the original position. The relations for other than the 4-component case are shown similarly.
\end{proof}

\begin{figure}[ht] 
$
\raisebox{-45 pt}{\begin{overpic}[width=100pt]{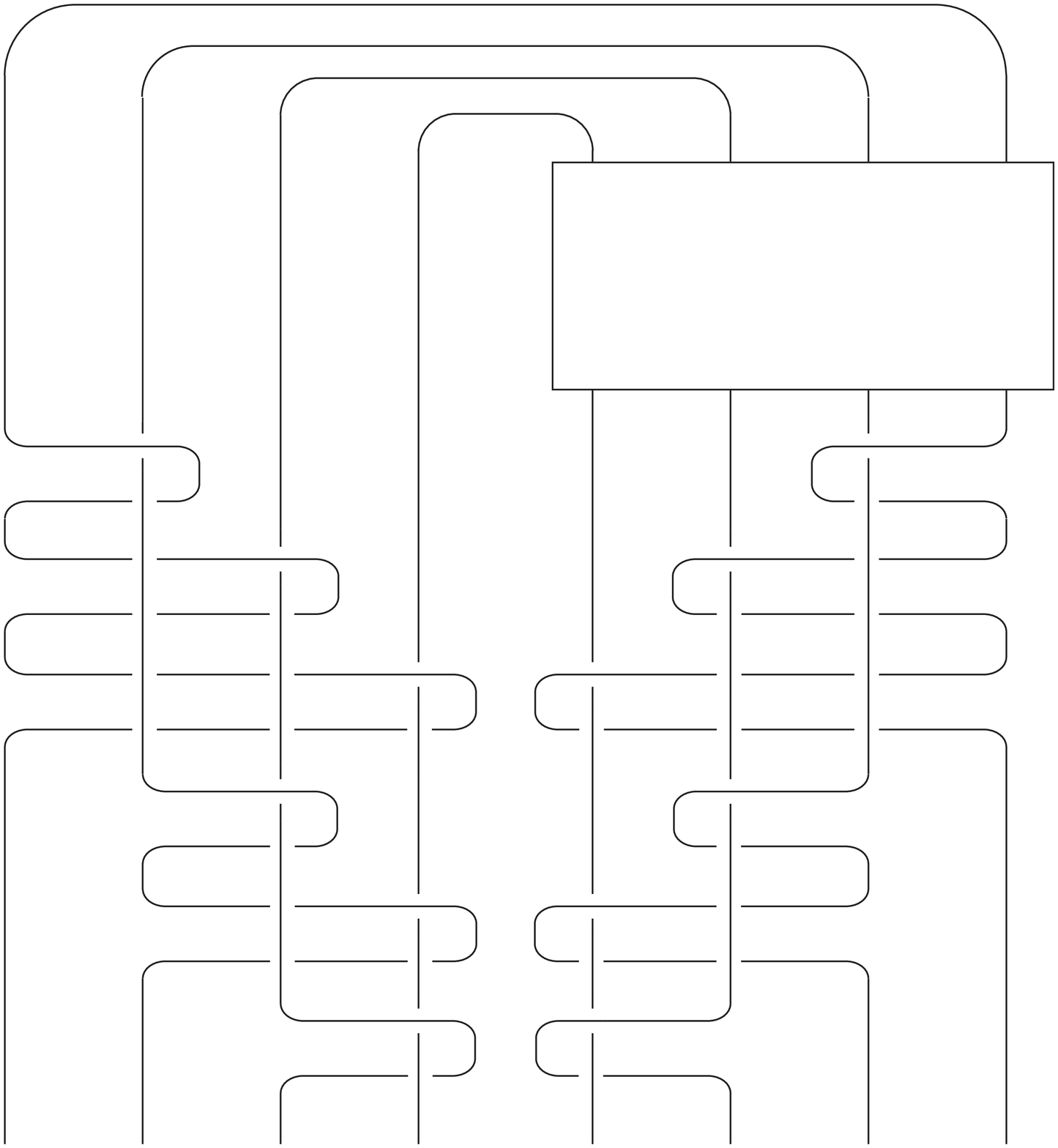}
\put(72,80){$\sigma$} 
\put(-27,50){\footnotesize$CX_4$}
\put(-27,24){\footnotesize$CX_3$} 
\put(-27,4){\footnotesize$CX_2$}
\end{overpic}}
\hspace{0.4cm}
\mbox{$=$}
\hspace{0.4cm}
\raisebox{-45 pt}{\begin{overpic}[width=100pt]{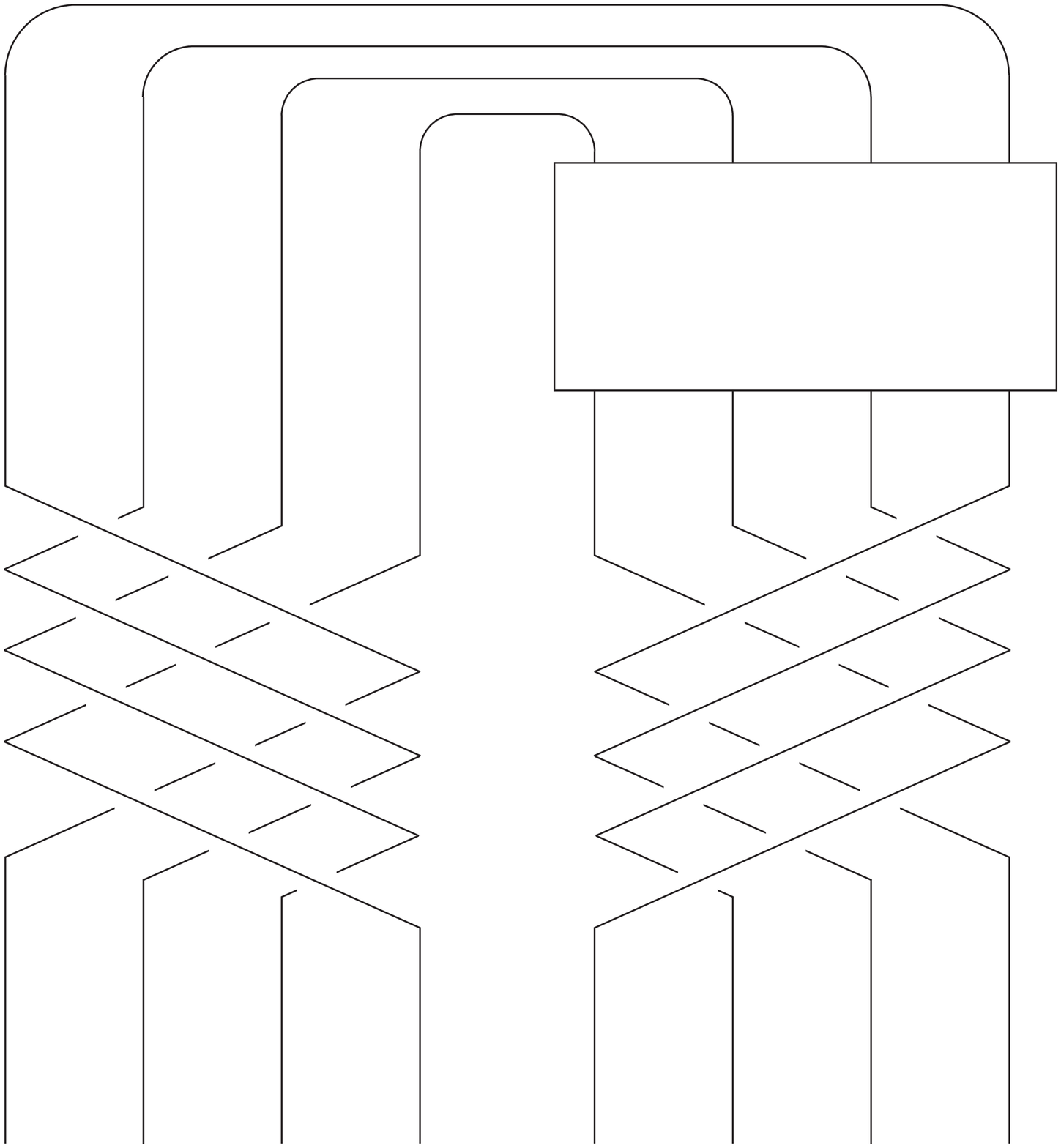}
\put(73,80){$\sigma$}
\end{overpic}}
$
\caption{Full twist} \label{ConjRelFig01}
\end{figure}

\begin{figure}[ht] 
$
\raisebox{-54 pt}{\begin{overpic}[width=100pt]{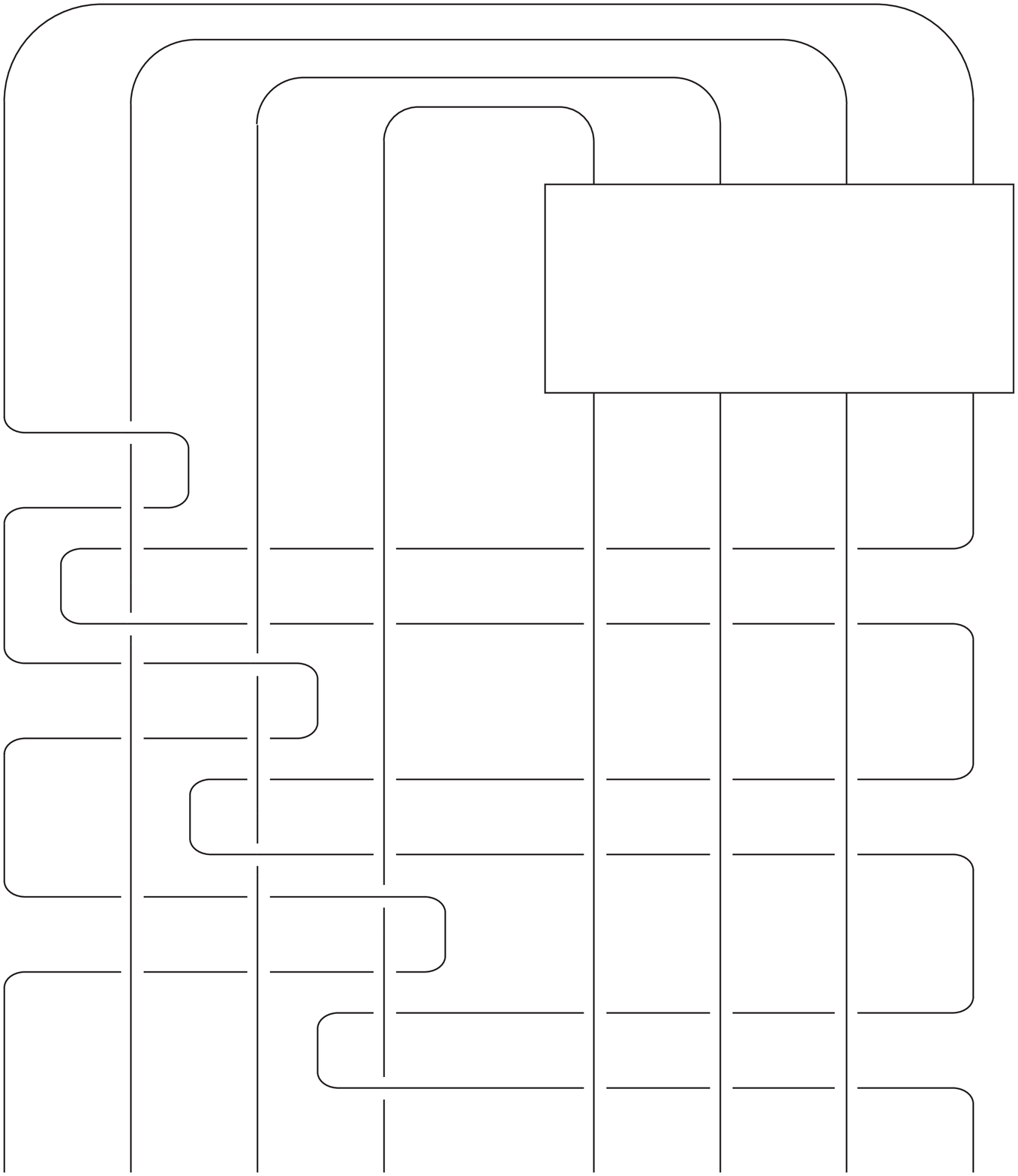}
%\put(){$\overline{x}_{}$}
\put(73,85){$\sigma$}
\put(-20,62){\small$x_{34}$}
\put(-20,39){\small$x_{24}$} 
\put(-20,16){\small$x_{14}$}
\end{overpic}}
\hspace{0.3cm}
\overset{\mbox{l.h.}}{\mbox{\large $\sim$}}
\hspace{0.3cm}
\raisebox{-39 pt}{\begin{overpic}[width=100pt]{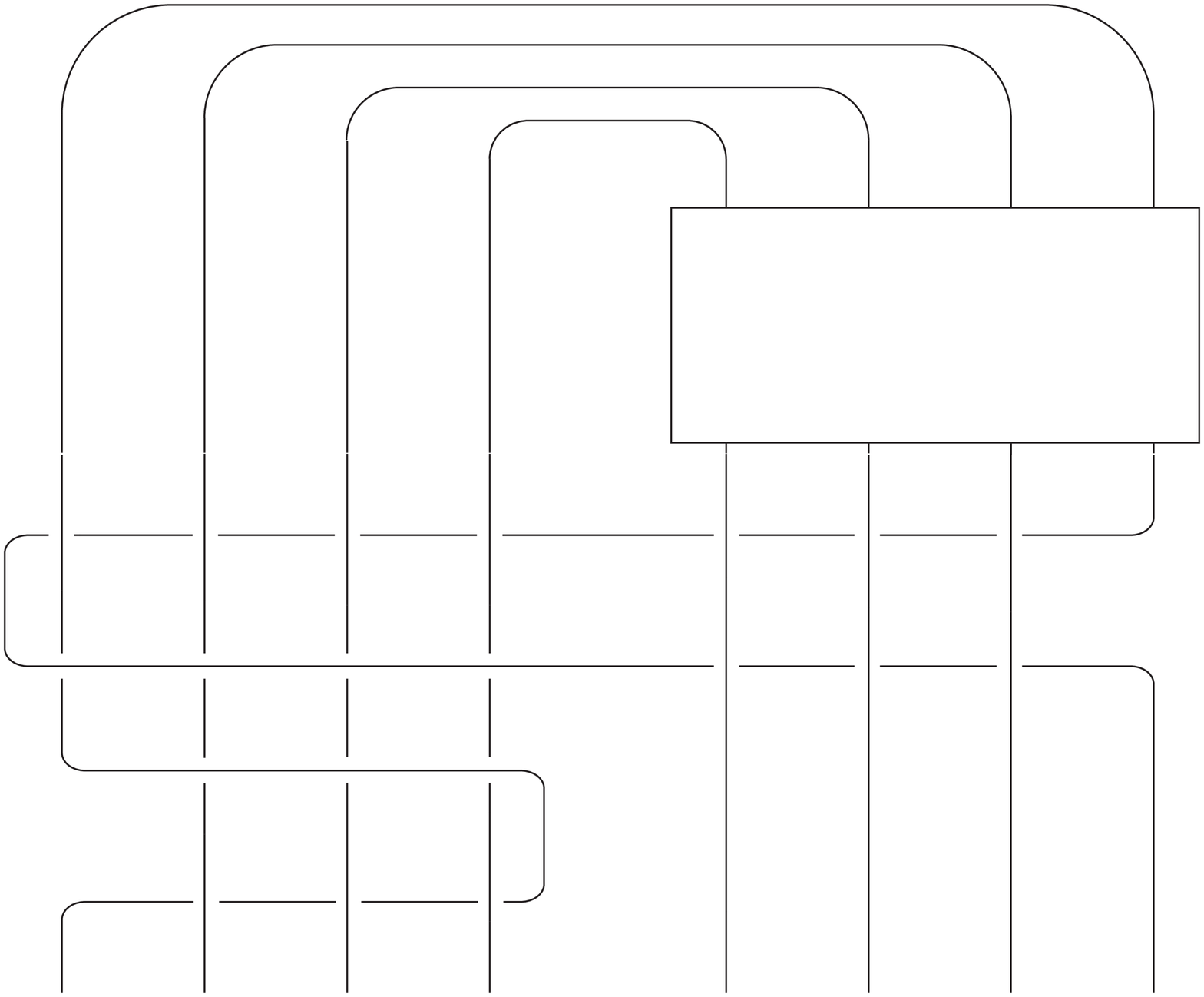}
\put(74,53){$\sigma$}
\end{overpic}}
\hspace{0.3cm}
=
\hspace{0.3cm}
\raisebox{-39 pt}{\begin{overpic}[width=100pt]{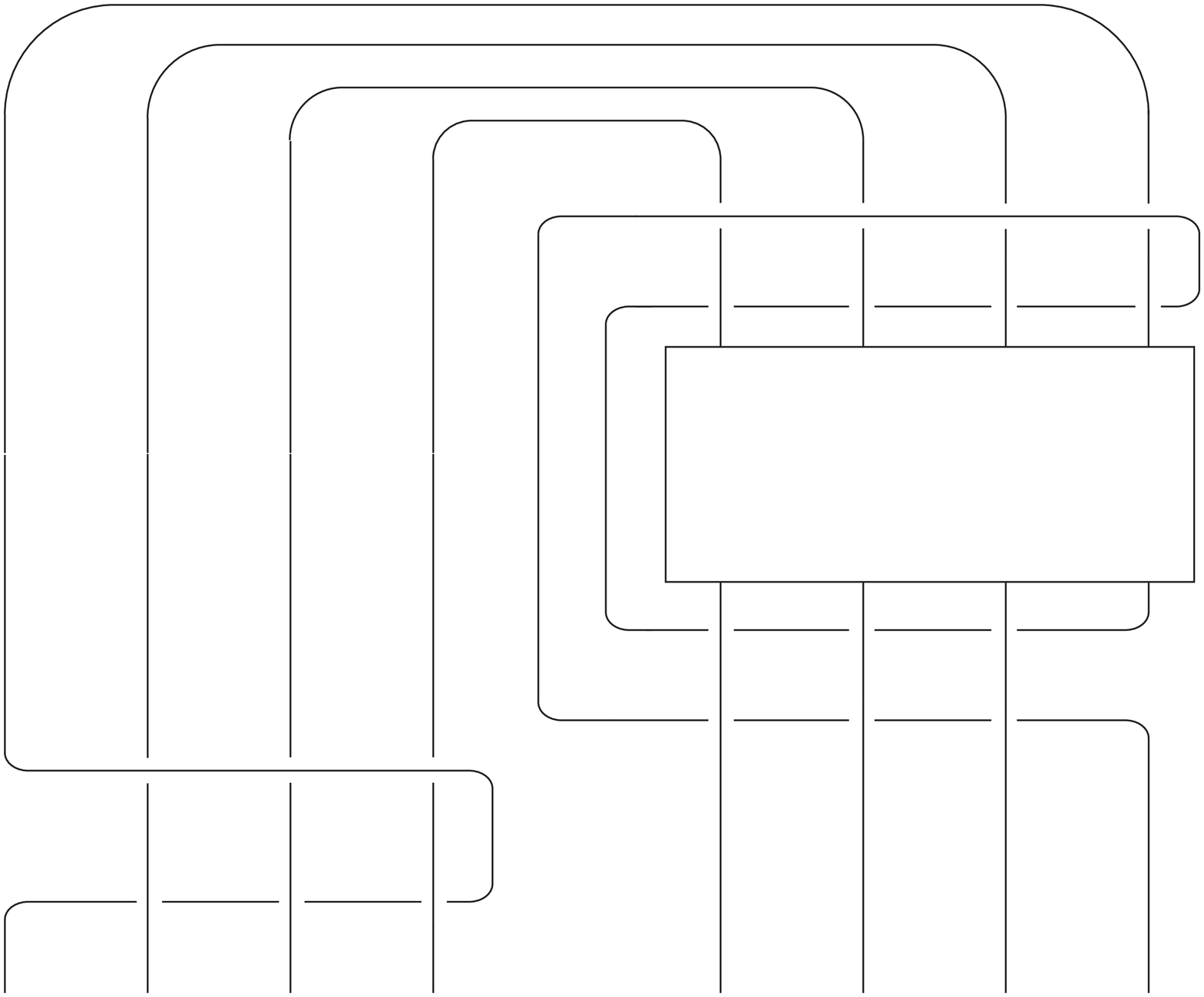}
\put(74,40){$\sigma$}
\end{overpic}}
$\\[20pt]

$
%\overset{\mbox{l.h.}}{\mbox{\large $\sim$}}
=
\hspace{0.3cm}
\raisebox{-37 pt}{\begin{overpic}[width=100pt]{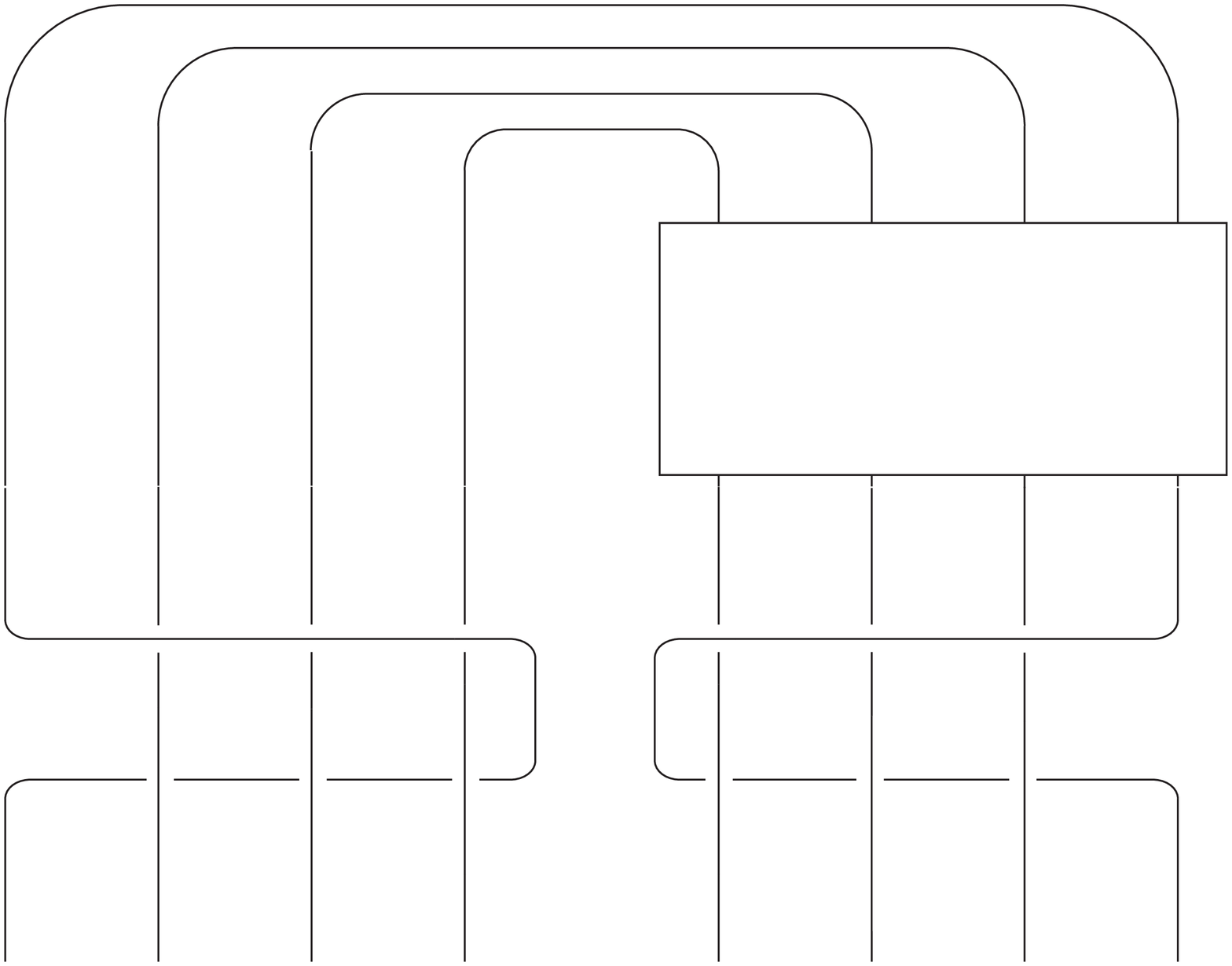}
\put(73,47){$\sigma$}
\end{overpic}}
\hspace{0.3cm}
=
\hspace{1.2cm}
\raisebox{-40 pt}{\begin{overpic}[width=100pt]{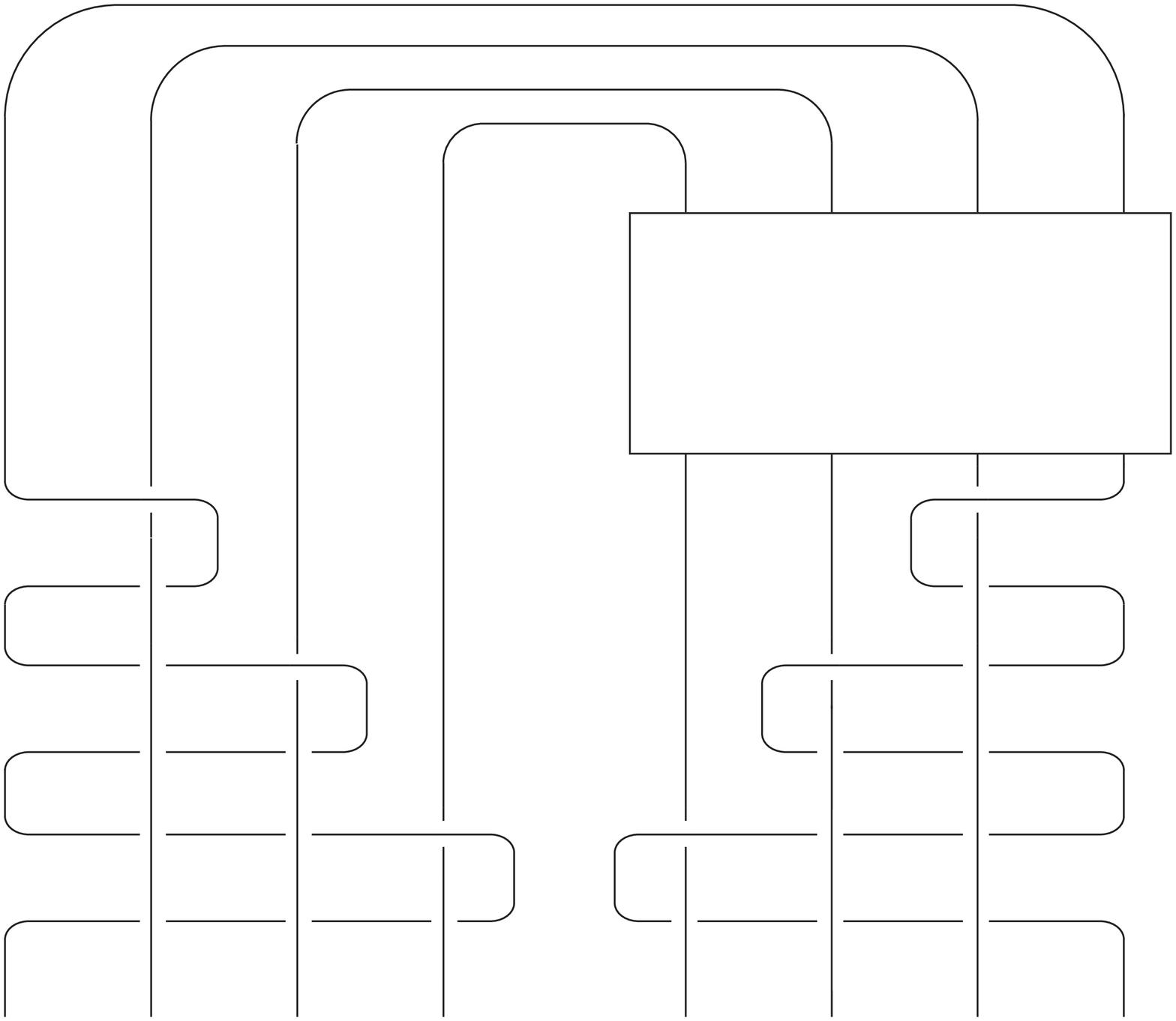}
\put(73,56){$\sigma$}
\put(-25,40){\small$cx_{34}$}
\put(-25,25){\small$cx_{24}$} 
\put(-25,10){\small$cx_{14}$}
\end{overpic}}
$
\caption{Relation between the partial conjugations and the conjugations.} \label{ConjRelFig02}
\end{figure}

The relations (1) and (2) in Proposition \ref{ConjRel} can be used to test the results of $cx_{ij}$ and $\overline{x}_{ij}$ respectively. 
We have checked these relations by using the software Mathematica. The files are in the web page \cite{Mweb}. 

%%%%%%%%%%%%%%%%%%%%%%%%%%%%%%%%%%
\appendix
%%%%%%%%%%%%%%%%%%%%%%%%%%%%%%%%%%

%%%%%%%%%%%%%%%%%%%%%%%%%%%%%%%%%%
\section{Examples of calculations for $\overline{x}_{ij}$} \label{example-calc}
%%%%%%%%%%%%%%%%%%%%%%%%%%%%%%%%%%
We show an example of the calculation of the action $\overline{x}_{13}$ for the 4-component canonical form. 
We use the relations of the claspers in Section \ref{preparations}. They also induce the relations in Figure \ref{clasper-rel-for-SF} up to link-homotopy, where the vertical lines are components of string links and the new claspers in the right-hand sides go parallel to the other claspers in outside the figures.

\begin{figure}[ht] 
$$(a)\hspace{0.5cm}\raisebox{-17pt}{\begin{overpic}[width=70pt]{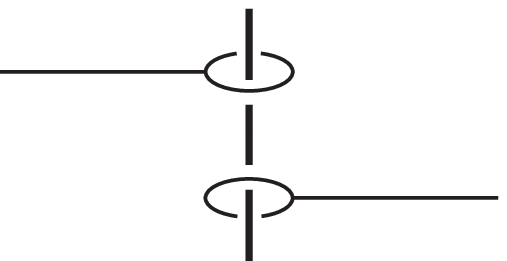}
\end{overpic}}
\hspace{0.3cm}
=
\hspace{0.3cm}
\raisebox{-24 pt}{\begin{overpic}[width=70pt]{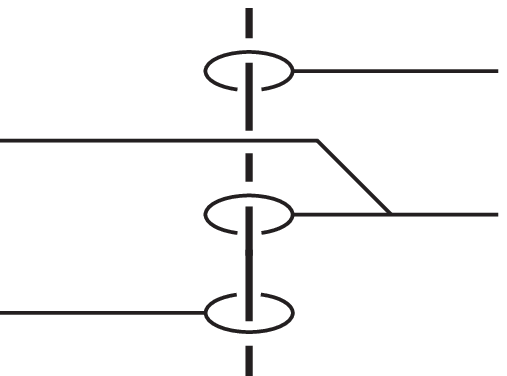}
\end{overpic}}
\hspace{1.0cm}
(b)\hspace{0.5cm}\raisebox{-17 pt}{\begin{overpic}[width=70pt]{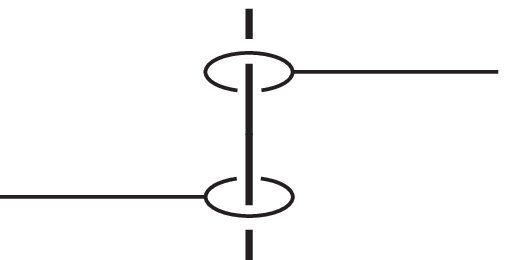}
\end{overpic}}
\hspace{0.3cm}
=
\hspace{0.3cm}
\raisebox{-24 pt}{\begin{overpic}[width=70pt]{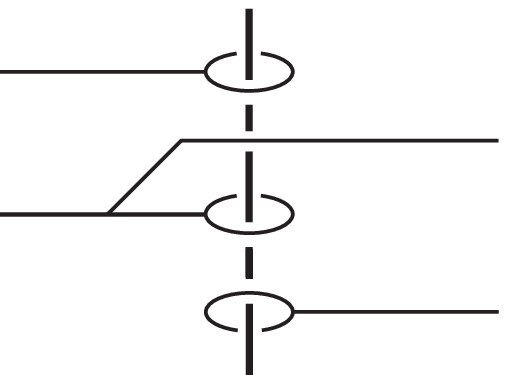}
\end{overpic}}
$$

$$(c)\hspace{0.4cm}\raisebox{-19 pt}{\begin{overpic}[width=70pt]{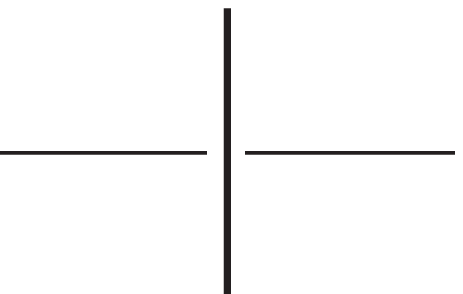}
\end{overpic}}
\hspace{0.3cm}
=
\hspace{0.3cm}
\raisebox{-19 pt}{\begin{overpic}[width=69pt]{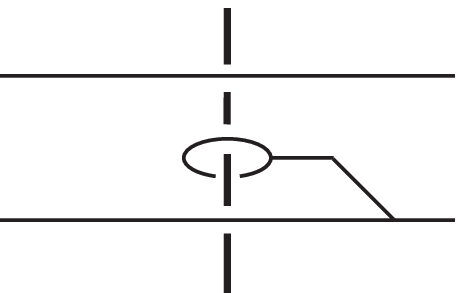}
\end{overpic}}
$$
\caption{Relations of the claspers for string links} \label{clasper-rel-for-SF}
\end{figure}

The calculation of the action $\overline{x}_{13}$ is in Figure \ref{Example01}, where $b$ is the canonical form for 4-component string links in Figure \ref{4-compSF}. The partial conjugation $\overline{x}_{13}$ is presented as the first figure. We transform it back to the canonical form. In the figures, we only write the numbers of claspers for the new claspers. 
First, we move the clasper over the canonical form $b$ down to under it. Then the clasper cancel with the twisted clasper bottommost. In the middle, by (2) in Lemma \ref{clasperlemma} and ($a$) in Figure \ref{clasper-rel-for-SF}, the new claspers appear which are shown in the second figure. 
Second, we move the leaves in the left-side to the right-side. 
Third, we start to move the new claspers to the positions of the canonical form. In the fourth figure, $-y_{234}$ parallel claspers moves to the position of the canonical form and, by ($c$) in Figure \ref{clasper-rel-for-SF}, new $+y_{34}$ $C_3$-claspers appear. In the fifth figure, the claspers in the fourth figure move to the positions of the canonical form. In the process, new claspers occur by (2) in Lemma \ref{clasperlemma} and $(b)$ in Figure \ref{clasper-rel-for-SF}. 
Finally, we obtain the new canonical form in the last figure. The difference of the numbers of the claspers from $b$ presents the action of $\overline{x}_{13}$. 
\par
By the similar calculations, we obtain the actions $\overline{x}_{ij}$
 for the 4-component case (Table \ref{Act4-compSF}) and the 5-component case (Table \ref{Act5-compSF}). 
 
 \begin{figure}[ht] 
$
\raisebox{-70 pt}{\begin{overpic}[width=123pt]{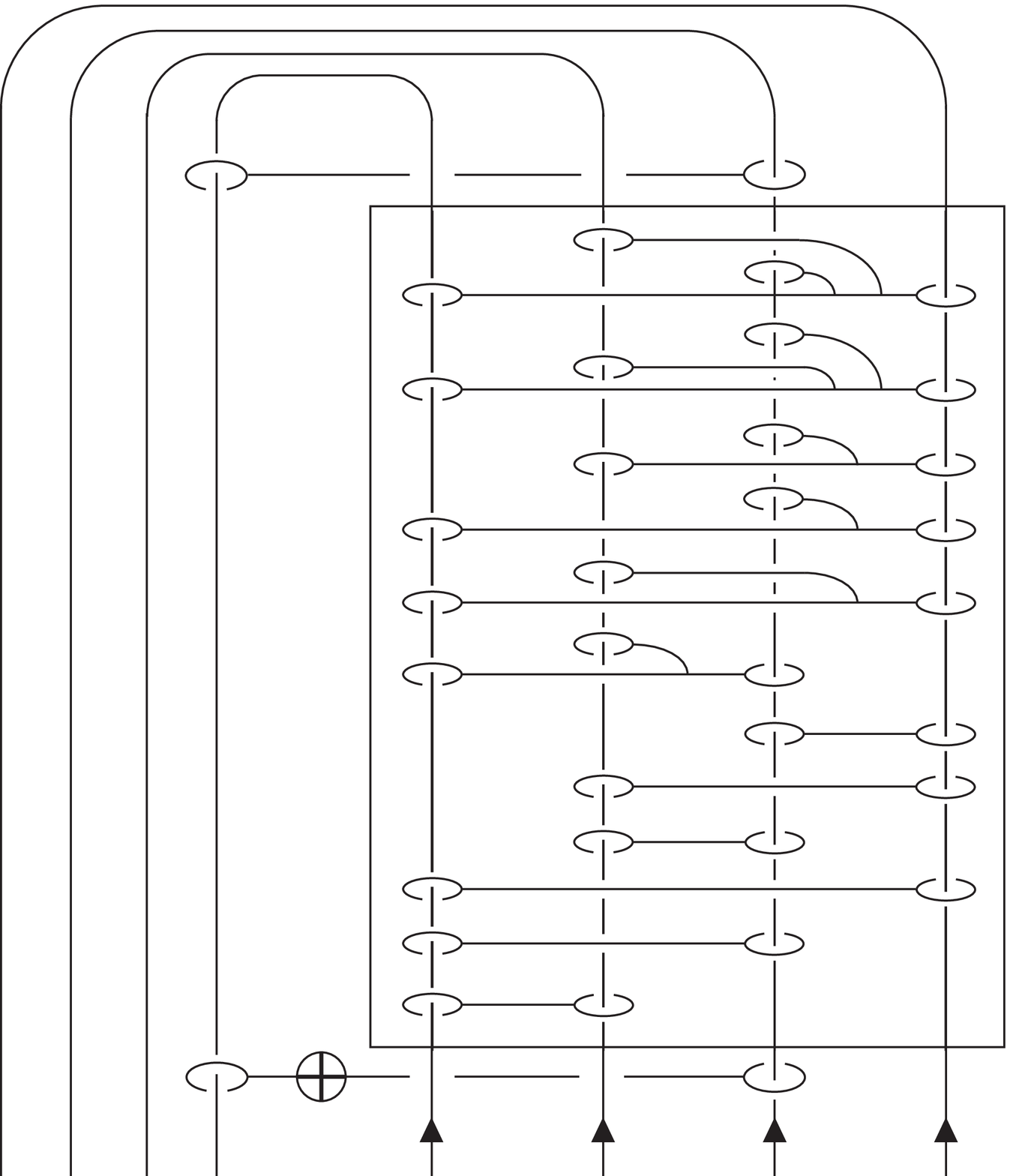}
\put(50,-12){$1$}\put(71,-12){$2$}
\put(92,-12){$3$}\put(113,-12){$4$}
\put(129,63){\large $b$}
\end{overpic}}
\hspace{0.7cm}\mbox{\Large$\rightarrow$}\quad
\raisebox{-70 pt}{\begin{overpic}[width=120pt]{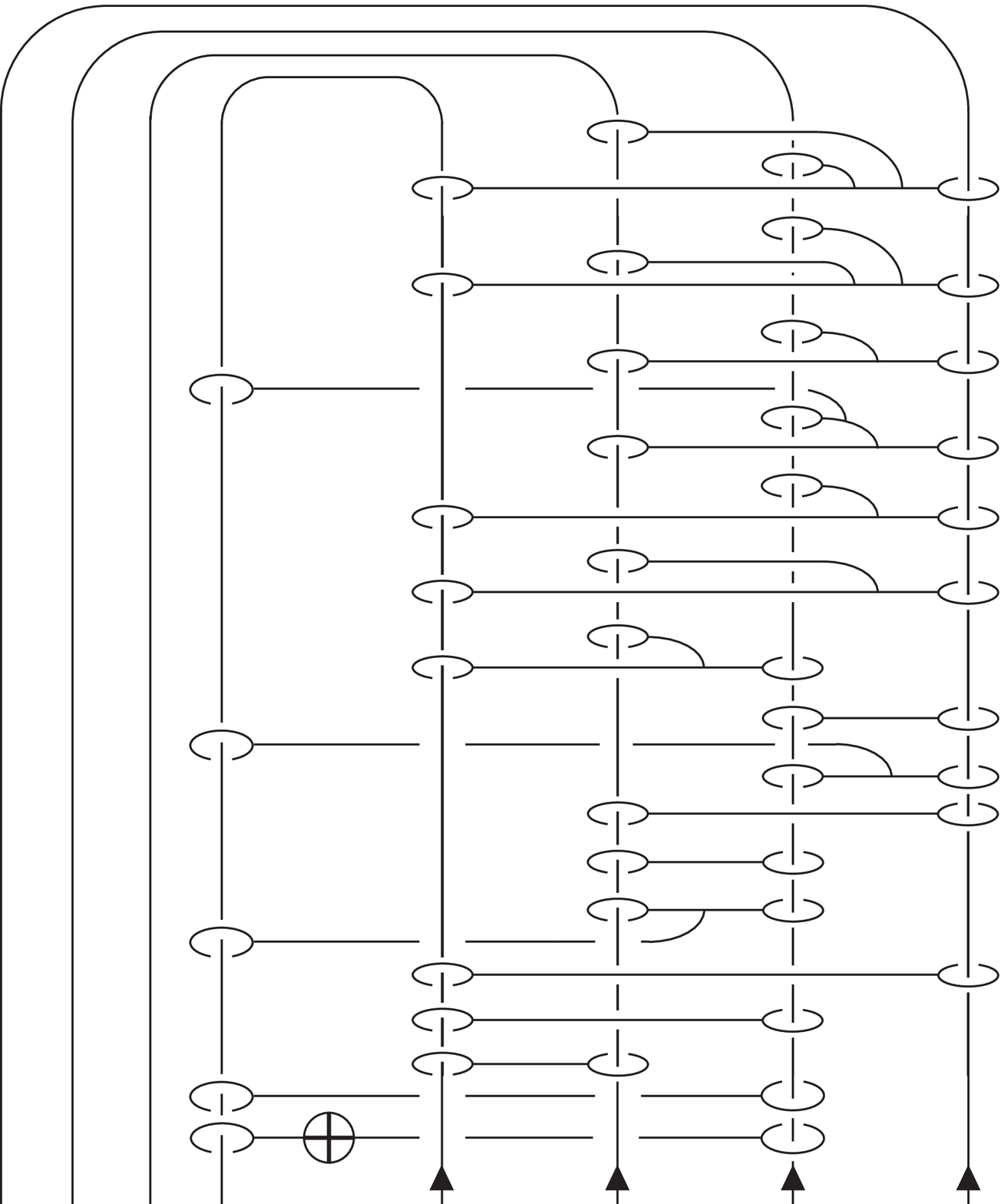}
\put(98,33){\tiny $+y_{23}$}\put(121,50){\tiny $+y_{34}$}
\put(121,88){\tiny $+y_{234}$}
\end{overpic}}
\hspace{0.9cm}\mbox{\Large$\rightarrow$}\quad
\raisebox{-70 pt}{\begin{overpic}[width=120pt]{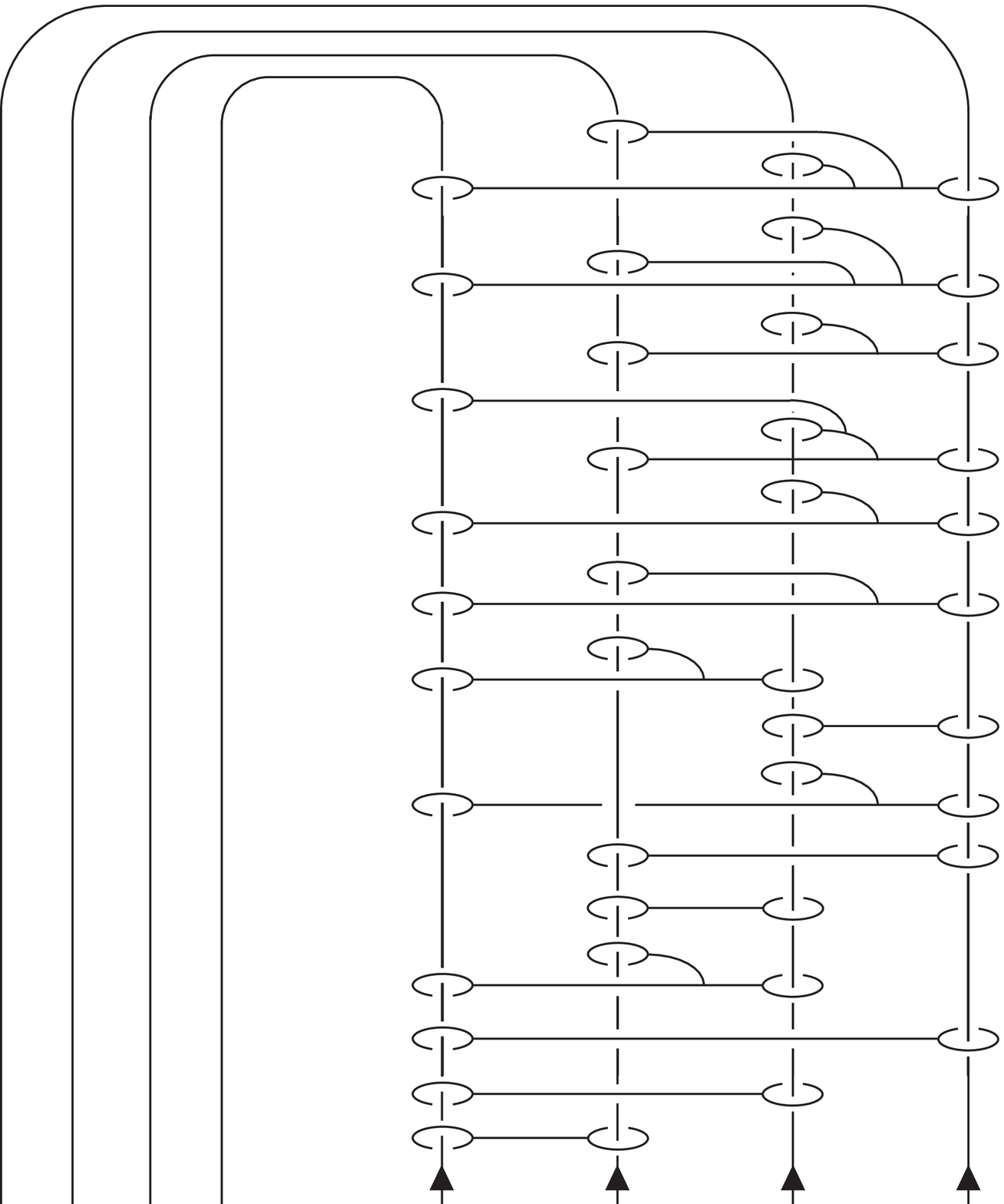}
\put(98,27){\tiny $-y_{23}$}\put(121,46){\tiny $+y_{34}$}
\put(121,87){\tiny $-y_{234}$}
\end{overpic}}
$
\\[30pt]

\if0
$
\mbox{\Large$\rightarrow$}\quad
\raisebox{-70 pt}{\begin{overpic}[width=120pt]{p-conjugate03-2.eps}
\put(98,27){\tiny $-y_{23}$}\put(121,46){\tiny $+y_{34}$}
\put(121,87){\tiny $-y_{234}$}
\end{overpic}}
\qquad\mbox{\Large$\rightarrow$}\quad
\raisebox{-85 pt}{\begin{overpic}[width=90pt]{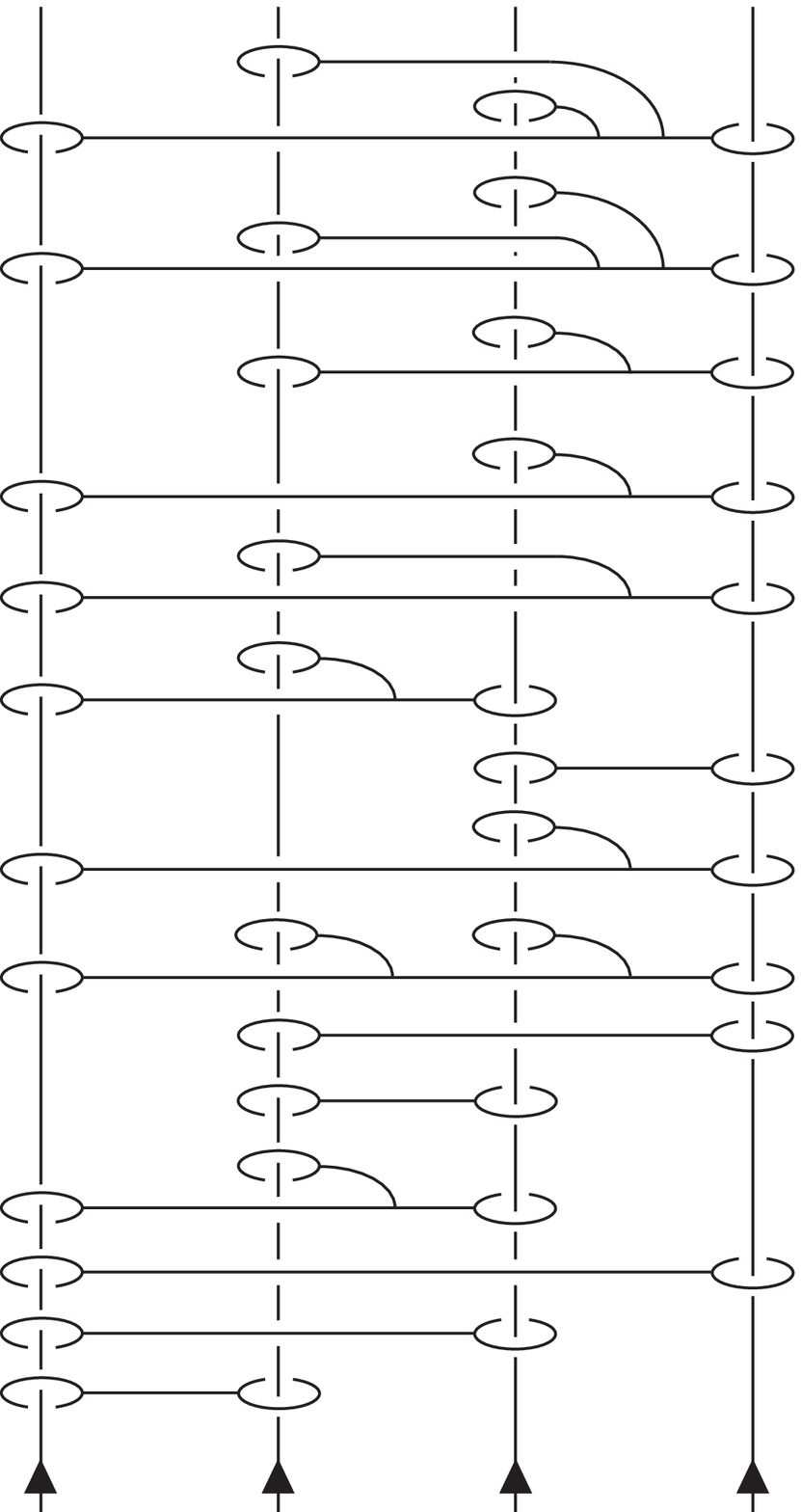}
\put(64,32){\tiny $+y_{23}$}\put(92,59){\tiny $+y_{34}$}
\put(92,71){\tiny $+y_{34}$}\put(92,154){\tiny $+y_{234}$}
\end{overpic}}$\\[30pt]
\fi

$\hspace{-1.0cm}
\mbox{\Large$\rightarrow$}\quad
\raisebox{-80 pt}{\begin{overpic}[width=90pt]{p-conjugate03-3.eps}
\put(64,32){\tiny $-y_{23}$}\put(92,59){\tiny $+y_{34}$}
\put(92,71){\tiny $+y_{34}$}\put(92,154){\tiny $-y_{234}$}
\end{overpic}}
\hspace{0.9cm}\mbox{\Large$\rightarrow$}\quad
\raisebox{-80 pt}{\begin{overpic}[width=90pt]{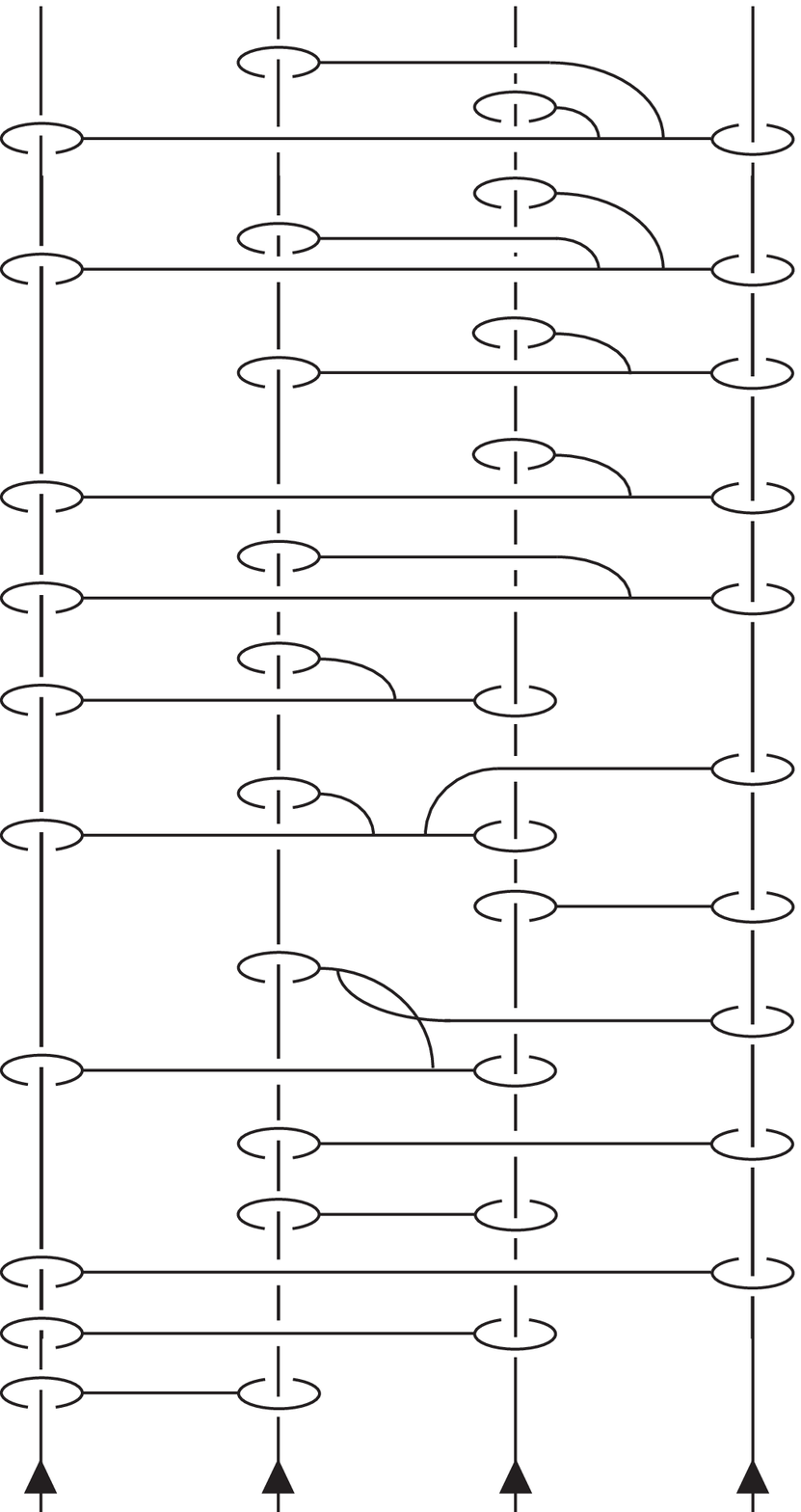}
\put(64,90){\tiny $-y_{23}$}
\put(92,113){\tiny $+y_{34}$}
\put(92,139){\tiny $+y_{34}$}
\put(92,54){\tiny $-y_{23}y_{24}$}
\put(92,82){\tiny $-y_{23}y_{34}$}
\put(92,154){\tiny $-y_{234}$}
\end{overpic}}
\hspace{1.4cm}\mbox{\Large$\rightarrow$}\quad
\raisebox{-70 pt}{\begin{overpic}[width=90pt]{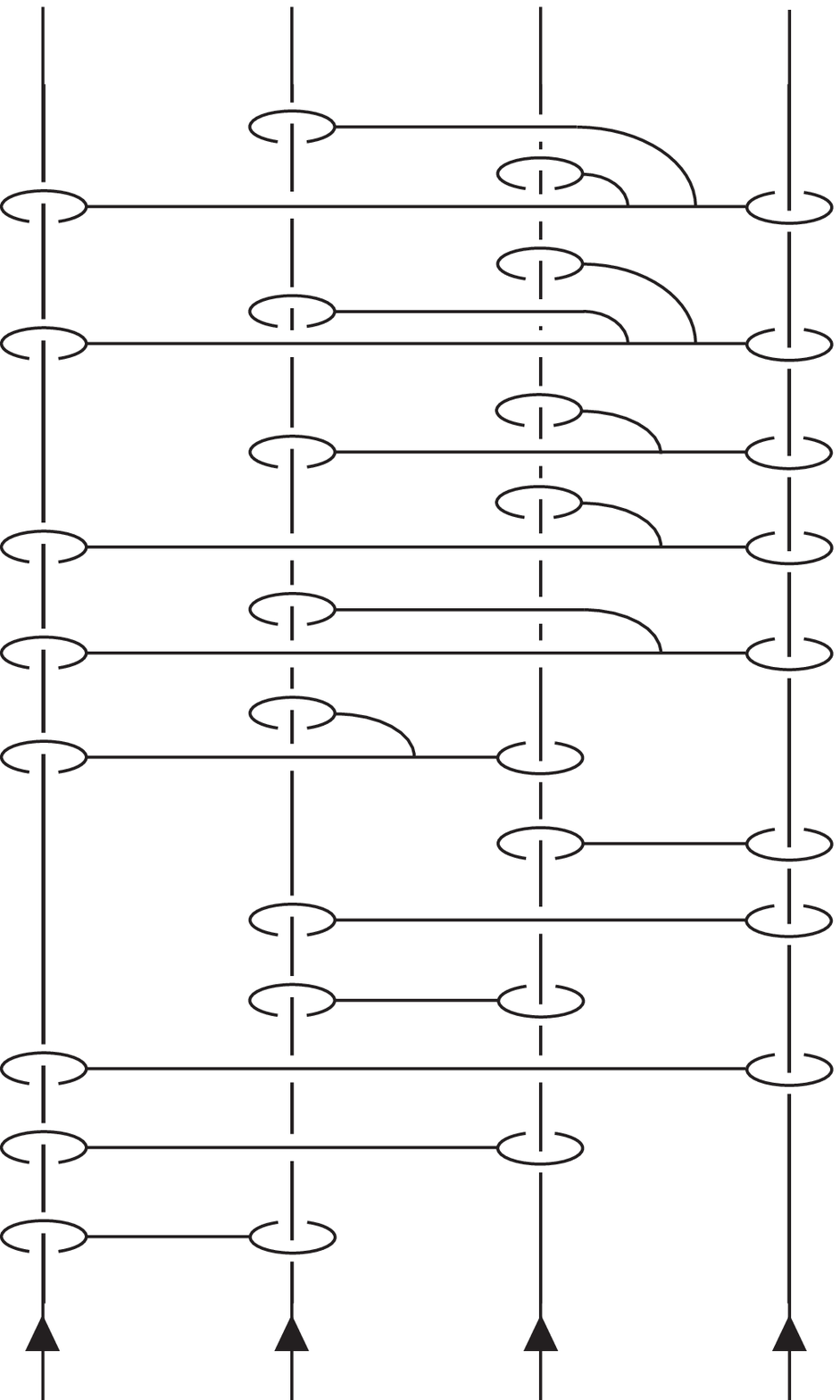}
\put(2,-12){$1$}\put(29,-12){$2$}
\put(55,-12){$3$}\put(82,-12){$4$}
\put(64,68){\tiny $-y_{23}$}
\put(92,101){\tiny $+y_{34}$}
\put(92,113){\tiny $+y_{34}\!-\!y_{23}y_{34}$}
\put(92,127){\tiny $-y_{234}\!+\!y_{23}y_{24}$}
\end{overpic}}
$
\caption{The action of $\overline{x}_{13}$ for 4-component canonical forms} \label{Example01}
\end{figure}

%\begin{figure}[ht] 
%\caption{$\overline{x}_{ij}$ for 5-component standard forms} \label{Example02}
%\end{figure}

%%%%%%%%%%%%%%%%%%%%%%%%%%%%%%%%%%
\section{Results of calculations} \label{result-tables}
%%%%%%%%%%%%%%%%%%%%%%%%%%%%%%%%%%
\par
We show the results of calculations. The results of the action of the generators $\overline{x}_{ij}$ of partial conjugations for 5-component string links are in Table \ref{Act5-compSF}. In Table \ref{ActCom5-compSF}, the actions of the commutators $[\overline{x}_{ij}, \overline{x}_{kl}]$ of $\overline{x}_{ij}$ are shown. In Table \ref{ActConj5-compSF}, the actions of conjugations $cx_{ij}$ are shown.

%%%%%%%%%%%%%%%%%% Actions of the partial conjugations  (start) %%%%%%%%%%%%%%%%%%%%%%%%%%%

\begin{table}[htb] 
\begin{center}
   \caption{Partial conjugations for 5-component string links} \label{Act5-compSF}
\begin{tabular}{|c|l|l|l|} \hline
      & $\overline{x}_{12}$ & $\overline{x}_{13}$ & $\overline{x}_{14}$ \\ \hline 
    $y_{123}$ & $y_{23}$  & $-y_{23}$ & $0$ \\ \hline 
    $y_{124}$ & $y_{24}$ & $0$ & $-y_{24}$\\ \hline
    $y_{134}$ & $0$ & $y_{34}$ & $-y_{34}$\\ \hline
    $y_{125}$ & $y_{25}$ & $0$ & $0$\\ \hline
    $y_{135}$ & $0$ & $y_{35}$ & $0$\\ \hline
    $y_{145}$ & $0$ & $0$ & $y_{45}$\\ \hline
    $y_{234}$ & $0$ & $0$ & $0$\\
\hline
  \end{tabular} 
\end{center}
\end{table}

\newpage
\begin{table}[htb] 
\begin{center}
  \begin{tabular}{|c|l|l|l|} \hline
     & $\overline{x}_{15}$ & $\overline{x}_{21}$ & $\overline{x}_{23}$\\ \hline 
    $y_{123}$ & $0$ & $-y_{13}$ & $y_{13}$\\ \hline 
    $y_{124}$ & $0$ & $-y_{14}$ & $0$\\ \hline
    $y_{134}$ & $0$ & $0$ & $0$\\ \hline
    $y_{125}$ & $-y_{25}$ & $-y_{15}$ & $0$\\ \hline
    $y_{135}$ & $-y_{35}$ & $0$ & $0$\\ \hline
    $y_{145}$ & $-y_{45}$ & $0$ & $0$\\ \hline
    $y_{234}$ & $0$ & $0$ & $y_{34}$\\
\hline
  \end{tabular} 
\end{center}
\end{table}

\newpage
\begin{table}[htb] 
\begin{center}
  \begin{tabular}{|c|l|l|l|} \hline
      & $\overline{x}_{24}$ & $\overline{x}_{25}$ & $\overline{x}_{31}$\\ \hline 
    $y_{123}$ & $0$ & $0$ & $y_{12}$\\ \hline 
    $y_{124}$ & $y_{14}$ & $0$ & $0$\\ \hline
    $y_{134}$ & $0$ & $0$ & $-y_{14}$\\ \hline
    $y_{125}$ & $0$ & $y_{15}$ & $0$\\ \hline
    $y_{135}$ & $0$ & $0$ & $-y_{15}$\\ \hline
    $y_{145}$ & $0$ & $0$ & $0$\\ \hline
    $y_{234}$ & $-y_{34}$ & $0$ & $0$\\
\hline
  \end{tabular} 
\end{center}
\end{table}

\newpage
\begin{table}[htb] 
\begin{center}
  \begin{tabular}{|c|l|l|l|} \hline
      & $\overline{x}_{32}$ & $\overline{x}_{34}$ & $\overline{x}_{35}$\\ \hline 
    $y_{123}$ & $-y_{12}$ & $0$ & $0$\\ \hline 
    $y_{124}$ & $0$ & $0$ & $0$\\ \hline
    $y_{134}$ & $0$ & $y_{14}$ & $0$\\ \hline
    $y_{125}$ & $0$ & $0$ & $0$\\ \hline
    $y_{135}$ & $0$ & $0$ & $y_{15}$\\ \hline
    $y_{145}$ & $0$ & $0$ & $0$\\ \hline
    $y_{234}$ & $-y_{24}$ & $y_{24}$ & $0$\\
\hline
  \end{tabular} \\
\end{center}
\end{table}

\newpage
\begin{table}[htb] 
\begin{center}
  \begin{tabular}{|c|l|l|l|} \hline
      & $\overline{x}_{41}$ & $\overline{x}_{42}$ & $\overline{x}_{43}$\\ \hline 
    $y_{123}$ & $0$ & $0$ & $0$\\ \hline 
    $y_{124}$ & $y_{12}$ & $-y_{12}$ & $0$\\ \hline
    $y_{134}$ & $y_{13}$ & $0$ & $-y_{13}$\\ \hline
    $y_{125}$ & $0$ & $0$ & $0$\\ \hline
    $y_{135}$ & $0$ & $0$ & $0$\\ \hline
    $y_{145}$ & $-y_{15}$ & $0$ & $0$\\ \hline
    $y_{234}$ & $0$ & $y_{23}$ & $-y_{23}$\\
\hline
  \end{tabular} \\
\end{center}
\end{table}

\newpage
\begin{table}[htb] 
\begin{center}
  \begin{tabular}{|c|l|l|} \hline
      & $\overline{x}_{45}$ & $\overline{x}_{51}$ \\ \hline 
    $y_{123}$ & $0$ & $0$ \\ \hline 
    $y_{124}$ & $0$ & $0$ \\ \hline
    $y_{134}$ & $0$ & $0$ \\ \hline
    $y_{125}$ & $0$ & $y_{12}$ \\ \hline
    $y_{135}$ & $0$ & $y_{13}$ \\ \hline
    $y_{145}$ & $y_{15}$ & $y_{14}$ \\ \hline
    $y_{234}$ & $0$ & $0$ \\
\hline
  \end{tabular} \\
\end{center}
\end{table}

\newpage
\begin{table}[htb] 
\begin{center}
  \begin{tabular}{|c|l|l|l|} \hline
      & $\overline{x}_{52}$ & $\overline{x}_{53}$ & $\overline{x}_{54}$ \\ \hline 
    $y_{123}$ & $0$ & $0$ & $0$\\ \hline 
    $y_{124}$ & $0$ & $0$ & $0$\\ \hline
    $y_{134}$ & $0$ & $0$ & $0$\\ \hline
    $y_{125}$ & $-y_{12}$ & $0$ & $0$\\ \hline
    $y_{135}$ & $0$ & $-y_{13}$ & $0$\\ \hline
    $y_{145}$ & $0$ & $0$ & $-y_{14}$\\ \hline
    $y_{234}$ & $0$ & $0$ & $0$\\
\hline
  \end{tabular} \\
\end{center}
\end{table}

%%%%%%%%%%%%%%%%%% Actions of the partial conjugations  (end) %%%%%%%%%%%%%%%%%%%%%%%%%%%

%%%%%%%%%%%%%%%%%% Actions of the commutators (start) %%%%%%%%%%%%%%%%%%%%%%%%%%%

\begin{table}[htb] 
\begin{center}
   \caption{Commutators of $\overline{x}_{ij}$ for 5-component string links} \label{ActCom5-compSF}
  {\small 
  \begin{tabular}{|c|l|l|l|l|} \hline
      & $[\overline{x}_{21}, \overline{x}_{31}]$ & $[\overline{x}_{21}, \overline{x}_{41}]$& $[\overline{x}_{21}, \overline{x}_{51}]$ & $[\overline{x}_{31}, \overline{x}_{41}]$  \\ \hline 
    $y_{1234}$ & $y_{14}$ & $0$ & $0$ & $-y_{12}$ \\ \hline 
    $y_{1324}$ & $-y_{14}$ & $-y_{13}$ & $0$ & $0$\\ \hline
    $y_{1235}$ & $y_{15}$ & $0$ & $0$ & $0$\\ \hline
    $y_{1245}$ & $0$ & $y_{15}$ & $0$ & $0$\\ \hline
    $y_{1325}$ & $-y_{15}$ & $0$ & $-y_{13}$ & $0$\\ \hline
    $y_{1345}$ & $0$ & $0$ & $0$ & $y_{15}$\\ \hline
    $y_{1425}$ & $0$ & $-y_{15}$ & $-y_{14}$ & $0$\\ \hline
    $y_{1435}$ & $0$ & $0$ & $0$ & $-y_{15}$\\ \hline
    $y_{2345}$ & $0$ & $0$ & $0$ & $0$\\ \hline
    $y_{2435}$ & $0$ & $0$ & $0$ & $0$\\ \hline\hline
    $y_{12345}$ & $-y_{14} y_{15} + y_{14} y_{45}$ & $y_{15}$ & $0$ & $y_{15} + y_{12} y_{15}$\\ 
     & $- y_{15} y_{45}+ y_{145}$ & & & \\ \hline 
    $y_{12435}$ & $0$ & $-y_{15} - y_{13} y_{15} + y_{13} y_{35}$ & $0$ & $-y_{15} - y_{12} y_{15}$\\ 
     & & $ - y_{15} y_{35} + y_{135}$ & & \\ 
     %& & $ $ & & \\ 
     \hline
    $y_{13245}$ & $y_{14} y_{15} - y_{14} y_{45} $ & $y_{13} y_{15}$  & $-y_{13}$ & $0$\\ 
     & $+ y_{15} y_{45}-y_{145}$ & & & \\ \hline
    $y_{13425}$ & $0$ & $-y_{15} - y_{13} y_{15}$ & $-y_{13} y_{34} + y_{14} y_{34} $ & $-y_{15} - y_{12} y_{15} + y_{12} y_{25}$\\ 
     &  &  & $ - y_{134}$ & $  - y_{15} y_{25}+ y_{125} $\\ 
     %&  &  &  & $ $\\ 
     \hline
    $y_{14235}$ & $y_{14} y_{15}$ & $y_{13} y_{15} - y_{13} y_{35}$ & $-y_{14}$ & $0$\\ 
     &  & $+ y_{15} y_{35} - y_{135}$ & & \\ \hline
    $y_{14325}$ & $-y_{14} y_{15}$ & $y_{15}$ & $-y_{13} y_{14} + y_{13} y_{34} $ & $y_{15} + y_{12} y_{15}- y_{12} y_{25}$\\ 
     &  &  & $- y_{14} y_{34} + y_{134}$ & $ + y_{15} y_{25}- y_{125}$\\ 
     %&  &  &  & $ $\\ 
     \hline
  \end{tabular} \\[10pt]
  }
  
  {\small
  \begin{tabular}{|c|l|l|l|l|} \hline
      & $[\overline{x}_{31}, \overline{x}_{51}]$ & $[\overline{x}_{41}, \overline{x}_{51}]$ & $[\overline{x}_{12}, \overline{x}_{32}]$ & $[\overline{x}_{12}, \overline{x}_{42}]$\\ \hline 
    $y_{1234}$ & $0$ & $0$  & $0$ & $y_{23}$\\ \hline 
    $y_{1324}$ & $0$ & $0$ & $-y_{24}$ & $-y_{23}$\\ \hline
    $y_{1235}$ & $-y_{12}$ & $0$ & $0$ & $0$\\ \hline
    $y_{1245}$ & $0$ & $-y_{12}$ & $0$ & $0$\\ \hline
    $y_{1325}$ & $0$ & $0$ & $-y_{25}$ & $0$\\ \hline
    $y_{1345}$ & $0$ & $-y_{13}$ & $0$ & $0$\\ \hline
    $y_{1425}$ & $0$ & $0$ & $0$ & $-y_{25}$\\ \hline
    $y_{1435}$ & $-y_{14}$ & $0$ & $0$ & $0$\\ \hline
    $y_{2345}$ & $0$ & $0$ & $0$ & $0$\\ \hline
    $y_{2435}$ & $0$ & $0$ & $0$ & $0$\\ \hline\hline
    $y_{12345}$ & $-y_{12}$ & $-y_{12} - y_{12} y_{23} $ & $0$ & $0$\\ 
     &  & $ + y_{13} y_{23} - y_{123}$ & & \\ \hline 
    $y_{12435}$ & $-y_{12} y_{24} + y_{14} y_{24}$ & $0$ & $0$ & $0$\\ 
     & $ - y_{124}$ & & & \\ \hline
    $y_{13245}$ & $0$ & $-y_{13} - y_{12} y_{13}+ y_{12} y_{23}$ & $-y_{24} y_{45} + y_{25} y_{45} $ & $0$ \\
     &  & $ - y_{13} y_{23} + y_{123}$ & $- y_{245}$ & \\
     %&  & $ $ & & \\ 
     \hline
    $y_{13425}$ & $0$ & $0$ & $y_{24} y_{25}$ & $-y_{25}$\\ \hline
    $y_{14235}$ & $-y_{14} - y_{12} y_{14} + y_{12} y_{24}$ & $0$ & $0$ & $-y_{23} y_{35} + y_{25} y_{35}$\\ 
     & $ - y_{14} y_{24} + y_{124}$ & & &  $- y_{235}$\\ 
     %& $ $ & & & \\ 
     \hline
    $y_{14325}$ & $0$ & $0$ & $0$ & $y_{25} + y_{23} y_{25}$\\ \hline
  \end{tabular} 
  }
\end{center}
\end{table}

\begin{table}[htb] 
\begin{center}
  %\caption{Commutators of $\overline{x}_{ij}$ for 5-component string links (cont.)} 
  %\caption{Commutators of $\overline{x}_{ij}$ for 5-component string links (cont.)} 
  {\small
  \begin{tabular}{|c|l|l|l|l|} \hline
      & $[\overline{x}_{12}, \overline{x}_{52}]$ & $[\overline{x}_{32}, \overline{x}_{42}]$ & $[\overline{x}_{32}, \overline{x}_{52}]$ & $[\overline{x}_{42}, \overline{x}_{52}]$\\ \hline 
    $y_{1234}$ & $0$ & $y_{12}$ & $0$ & $0$\\ \hline 
    $y_{1324}$ & $0$ & $0$ & $0$ & $0$\\ \hline
    $y_{1235}$ & $y_{23}$ & $0$ & $y_{12}$ & $0$\\ \hline
    $y_{1245}$ & $y_{24}$ & $0$ & $0$ & $y_{12}$\\ \hline
    $y_{1325}$ & $-y_{23}$ & $0$ & $0$ & $0$\\ \hline
    $y_{1345}$ & $0$ & $0$ & $0$ & $0$\\ \hline
    $y_{1425}$ & $-y_{24}$ & $0$ & $0$ & $0$\\ \hline
    $y_{1435}$ & $0$ & $0$ & $0$ & $0$\\ \hline
    $y_{2345}$ & $0$ & $y_{25}$ & $0$ & $-y_{23}$\\ \hline
    $y_{2435}$ & $0$ & $-y_{25}$ & $-y_{24}$ & $0$\\ \hline\hline
    $y_{12345}$ & $y_{23} + y_{23} y_{34} $ & $-y_{12} y_{15} + y_{125}$ & $y_{12}$ & $y_{12} + y_{12} y_{23}$\\  
     & $- y_{24} y_{34} + y_{234}$ & & & \\ \hline 
    $y_{12435}$ & $y_{24} + y_{23} y_{24} $ & $y_{12} y_{15} - y_{125}$ & $y_{12} y_{24}$ & $0$\\ 
     & $ - y_{23} y_{34} + y_{24} y_{34} $ & & & \\ 
     & $ - y_{234}$ & & & \\ \hline
    $y_{13245}$ & $-y_{23} - y_{23} y_{24}$ & $0$ & $0$ & $y_{12} y_{13} - y_{12} y_{23}$\\ 
     &  &  &  & $ - y_{123}$\\ \hline
    $y_{13425}$ & $y_{23} y_{24} - y_{23} y_{34}$ & $y_{12} y_{15} - y_{12} y_{25}$ & $0$ & $0$\\ 
     & $ + y_{24} y_{34} - y_{234}$ & $- y_{125}$ & & \\ \hline
    $y_{14235}$ & $-y_{24} - y_{23} y_{24}$ & $0$ & $y_{12} y_{14} - y_{12} y_{24}$ & $0$\\ 
     &  &  & $ - y_{124}$ & \\ \hline
    $y_{14325}$ & $y_{23} y_{34} - y_{24} y_{34}$ & $-y_{12} y_{15} + y_{12} y_{25} $ & $0$ & $0$\\ 
     & $ + y_{234}$ & $ + y_{125}$ & & \\ \hline
  \end{tabular} \\[10pt]
  }
  
  {\small
  \begin{tabular}{|c|l|l|l|l|} \hline
      & $[\overline{x}_{13}, \overline{x}_{23}]$ & $[\overline{x}_{13}, \overline{x}_{43}]$  & $[\overline{x}_{13}, \overline{x}_{53}]$ & $[\overline{x}_{23}, \overline{x}_{43}]$\\ \hline 
    $y_{1234}$ & $-y_{34}$ & $-y_{23}$ & $0$ & $0$\\ \hline 
    $y_{1324}$ & $0$ & $y_{23}$ & $0$ & $y_{13}$ \\ \hline
    $y_{1235}$ & $-y_{35}$ & $0$ & $-y_{23}$ & $0$\\ \hline
    $y_{1245}$ & $0$ & $0$ & $0$ & $0$\\ \hline
    $y_{1325}$ & $0$ & $0$ & $y_{23}$ & $0$\\ \hline
    $y_{1345}$ & $0$ & $0$ & $y_{34}$ & $0$\\ \hline
    $y_{1425}$ & $0$ & $0$ & $0$ & $0$\\ \hline
    $y_{1435}$ & $0$ & $-y_{35}$ & $-y_{34}$ & $0$\\ \hline
    $y_{2345}$ & $0$ & $0$ & $0$ & $0$\\ \hline
    $y_{2435}$ & $0$ & $0$ & $0$ & $-y_{35}$\\ \hline\hline
    $y_{12345}$ & $-y_{34} y_{45} + y_{35} y_{45}$ & $0$ & $-y_{23} + y_{34} $ & $0$\\ 
     & $ - y_{345}$ & & $- y_{23} y_{34}$ &  \\ \hline 
    $y_{12435}$ & $y_{34} y_{35}$ & $-y_{35}$ & $-y_{23} y_{24} - y_{34} $ & $y_{13} y_{15} - y_{13} y_{35} $\\ 
     &  &  & $+ y_{23} y_{34} + y_{234}$ & $- y_{135}$\\ \hline
    $y_{13245}$ & $0$ & $0$ & $y_{23} + y_{23} y_{24} $ & $-y_{13} y_{15} + y_{135}$\\ 
     &  & & $- y_{234}$ & \\ \hline
    $y_{13425}$ & $0$ & $0$ & $-y_{23} y_{24} + y_{23} y_{34}$ & $y_{13} y_{15} - y_{135}$\\ 
     &  &  & $+ y_{234}$ & \\ \hline
    $y_{14235}$ & $0$ & $y_{23} y_{35}$ & $y_{23} y_{24} - y_{234}$ & $-y_{13} y_{15} + y_{13} y_{35} $\\ 
     &  &  &  & $+ y_{135}$\\ \hline
    $y_{14325}$ & $0$ & $-y_{23} y_{25} + y_{235}$ & $-y_{23} y_{34}$ & $0$\\ \hline
  \end{tabular}
  }\\
\end{center}
\end{table}

\begin{table}[htb] 
  %\caption{Commutators of $\overline{x}_{ij}$ for 5-component string links (cont.)} 
  
  {\small
  \begin{tabular}{|c|l|l|l|l|} \hline
      & $[\overline{x}_{23}, \overline{x}_{53}]$ & $[\overline{x}_{14}, \overline{x}_{24}]$ & $[\overline{x}_{14}, \overline{x}_{34}]$ & $[\overline{x}_{24}, \overline{x}_{34}]$\\ \hline 
    $y_{1234}$ & $0$ & $y_{34}$ & $0$ & $-y_{14}$\\ \hline 
    $y_{1324}$ & $0$ & $0$ & $y_{24}$ & $y_{14}$\\ \hline
    $y_{1235}$ & $0$ & $0$ & $0$ & $0$\\ \hline
    $y_{1245}$ & $0$ & $-y_{45}$ & $0$ & $0$\\ \hline
    $y_{1325}$ & $y_{13}$ & $0$ & $0$ & $0$\\ \hline
    $y_{1345}$ & $0$ & $0$ & $-y_{45}$ & $0$\\ \hline
    $y_{1425}$ & $0$ & $0$ & $0$ & $0$\\ \hline
    $y_{1435}$ & $0$ & $0$ & $0$ & $0$\\ \hline
    $y_{2345}$ & $y_{34}$ & $0$ & $0$ & $-y_{45}$\\ \hline
    $y_{2435}$ & $-y_{34}$ & $0$ & $0$ & $0$\\ \hline\hline
    $y_{12345}$ & $0$ & $-y_{45} + y_{34} y_{45}$ & $-y_{45}$ & $y_{14} y_{15} - y_{14} y_{45}$\\ 
     &  &  &  & $ - y_{145}$\\ \hline 
    $y_{12435}$ & $0$ & $-y_{34} y_{35} + y_{345}$ & $0$ & $0$\\ \hline 
    $y_{13245}$ & $y_{13}$ & $0$ & $y_{24} y_{45}$ & $-y_{14} y_{15} + y_{14} y_{45}$\\ 
     &  &  &  & $ + y_{145}$\\ \hline
    $y_{13425}$ & $y_{13} y_{34}$ & $0$  & $-y_{24} y_{25} + y_{245}$ & $0$\\ \hline
    $y_{14235}$ & $0$ & $0$ & $0$ & $-y_{14} y_{15} + y_{145}$\\ \hline
    $y_{14325}$ & $y_{13} y_{14} - y_{13} y_{34} $ & $0$ & $0$ & $y_{14} y_{15} - y_{145}$\\ 
     & $ - y_{134}$ &  &  & \\ \hline
  \end{tabular} \\
  }
\end{table}

%%%%%%%%%%%%%%%%%% The actions of commutators (end) %%%%%%%%%%%%%%%%%%%%%%%%%%%

%%%%%%%%%%%%%%%%%% The results of Conjugation for 4-comp (start) %%%%%%%%%%%%%%%%%%%%%%%%%%%
\begin{table}[htb] 
\begin{center}
   \caption{Conjugations for 4-component string links} \label{ActConj4-compSF}
  \begin{tabular}{|c|c|c|c|} \hline
      & $cx_{12}$ & $cx_{13}$ & $cx_{14}$  \\ \hline
    $y_{123}$ & $y_{23}-y_{13}$ & $-y_{23}+y_{12}$ & 0 \\
    \hline 
    $y_{124}$ & $y_{24}-y_{14}$ & 0 & $-y_{24}+y_{12}$ \\
    \hline
    $y_{134}$ & 0 & $y_{34}-y_{14}$ & $-y_{34}+y_{13}$ \\
    \hline 
    $y_{234}$ & 0 & 0 & 0  \\
    \hline
    $y_{1234}$ & $-y_{134}+y_{234}$ & $y_{34}-y_{23} y_{34}$ & $-y_{234}+y_{123}$ \\
     & $-y_{24} y_{34}+y_{23} y_{34}$ & $-y_{14}+y_{14} y_{23}$  & $-y_{34}+y_{24} y_{34}$ \\
 & $y_{14} y_{34}-y_{14} y_{23}$ & $+y_{12} y_{34}-y_{12} y_{14}$ & $-y_{13} y_{23}-y_{12} y_{34}$ \\
 & $-y_{13} y_{34}+y_{13} y_{14}$ &  & $+y_{12}+y_{12} y_{23}$ \\
    \hline 
    %$[\overline{x}_{32}, \overline{x}_{42}]$ & 0 & 0 & 0 & 0 & $-y_{12}$ & 0 \\
    %\hline
    $y_{1324}$ & $-y_{23} y_{24}+y_{14} y_{23}$ & $-y_{234}-y_{124}$ & $y_{234}-y_{123}$ \\
     & $+y_{13} y_{24}-y_{13} y_{14}$ & $+y_{23} y_{24}+y_{14}$ & $-y_{24}+y_{13}$\\
     &  & $+y_{14} y_{24}-y_{14} y_{23}$ & $-y_{13} y_{24}+y_{13} y_{23}$\\
     &  & $-y_{12} y_{24}+y_{12} y_{14}$ & $-y_{12} y_{23}+y_{12} y_{13}$\\
    \hline
    %$[\overline{x}_{13}, \overline{x}_{43}]$ & 0 & 0 & 0 & 0 & $y_{23}$ & $-y_{23}$ \\
    %\hline 
    %$[\overline{x}_{23}, \overline{x}_{43}]$ & 0 & 0 & 0 & 0 & 0 & $-y_{13}$ \\
    %\hline
    %$[\overline{x}_{14}, \overline{x}_{24}]$ & 0 & 0 & 0 & 0 & $-y_{34}$ & 0 \\
    %\hline
    %$[\overline{x}_{14}, \overline{x}_{34}]$ & 0 & 0 & 0 & 0 & 0 & $-y_{24}$ \\
    %\hline 
    %$[\overline{x}_{24}, \overline{x}_{34}]$ & 0 & 0 & 0 & 0 & $y_{14}$ & $-y_{14}$ \\
    %\hline
  \end{tabular} 
\end{center}

\begin{center}
  \begin{tabular}{|c|c|c|c|} \hline
    & $cx_{23}$ & $cx_{24}$ & $cx_{34}$ \\ \hline
    $y_{123}$ & $y_{13}-y_{12}$ & 0 & 0 \\
    \hline 
    $y_{124}$ & 0 & $y_{14}-y_{12}$ & 0 \\
    \hline
    $y_{134}$ & 0 & $-y_{34}+y_{23}$ & $y_{14}-y_{13}$ \\
    \hline 
    $y_{234}$ & $y_{34}-y_{24}$ & 0 & $y_{24}-y_{23}$ \\
    \hline
    $y_{1234}$ & $y_{134}-y_{124}$ & $y_{14}-y_{14} y_{34}$ & $y_{124}-y_{123}$ \\
     & $+y_{13} y_{34}-y_{13} y_{14}$ & $+y_{14} y_{23}-y_{12}$ & $-y_{14} y_{23}+y_{13} y_{23}$ \\
     & $-y_{12} y_{34}+y_{12} y_{14}$ & $+y_{12} y_{34}-y_{12} y_{23}$ &  \\
    \hline 
    %$[\overline{x}_{32}, \overline{x}_{42}]$ & 0 & 0 & 0 & 0 & $-y_{12}$ & 0 \\
    %\hline
    $y_{1324}$ & $-y_{134}+y_{124}$ & $y_{134}+y_{123}$ & $-y_{14} y_{24}+y_{14} y_{23}$ \\
     & $-y_{13} y_{24}+y_{13} y_{14}$ & $-y_{14}-y_{14} y_{23}$ & $+y_{13} y_{24}-y_{13} y_{23}$ \\
     & $+y_{12} y_{24}-y_{12} y_{14}$ & $+y_{12} y_{23}-y_{12} y_{13}$ & \\
    \hline
    %$[\overline{x}_{13}, \overline{x}_{43}]$ & 0 & 0 & 0 & 0 & $y_{23}$ & $-y_{23}$ \\
    %\hline 
    %$[\overline{x}_{23}, \overline{x}_{43}]$ & 0 & 0 & 0 & 0 & 0 & $-y_{13}$ \\
    %\hline
    %$[\overline{x}_{14}, \overline{x}_{24}]$ & 0 & 0 & 0 & 0 & $-y_{34}$ & 0 \\
    %\hline
    %$[\overline{x}_{14}, \overline{x}_{34}]$ & 0 & 0 & 0 & 0 & 0 & $-y_{24}$ \\
    %\hline 
    %$[\overline{x}_{24}, \overline{x}_{34}]$ & 0 & 0 & 0 & 0 & $y_{14}$ & $-y_{14}$ \\
    %\hline
  \end{tabular} 
\end{center}
\end{table}
%%%%%%%%%%%%%%%%%% The results of Conjugation for 4-comp (end) %%%%%%%%%%%%%%%%%%%%%%%%%%%

%%%%%%%%%%%%%%%%%% The results of Conjugation for 5-comp (start) %%%%%%%%%%%%%%%%%%%%%%%%%%%
\clearpage

\begin{table}[htb] 
  \caption{Conjugations for 5-component string links} \label{ActConj5-compSF}
  \begin{tabular}{|c|l|} \hline
      & $cx_{12}$ \\ \hline 
    $y_{123}$ & $y_{23}-y_{13}$\\ \hline 
    $y_{124}$ & $y_{24}-y_{14}$\\ \hline
    $y_{134}$ & $0$\\ \hline
    $y_{125}$ & $y_{25}-y_{15}$\\ \hline
    $y_{135}$ & $0$\\ \hline
    $y_{145}$ & $0$\\ \hline
    $y_{234}$ & $0$\\
\hline
  \end{tabular} \\
\end{table}

\begin{table}[htb] 
  %\caption{Conjugations for 5-component string links} \label{ActConj5-compSF}
  \begin{tabular}{|c|l|} \hline
      & $cx_{13}$ \\ \hline 
    $y_{123}$ & $-y_{23}+y_{12}$\\ \hline 
    $y_{124}$ & $0$\\ \hline
    $y_{134}$ & $+y_{34}-y_{14}$\\ \hline
    $y_{125}$ & $0$\\ \hline
    $y_{135}$ & $+y_{35}-y_{15}$\\ \hline
    $y_{145}$ & $0$\\ \hline
    $y_{234}$ & $0$\\
\hline
  \end{tabular} \\
\end{table}

\begin{table}[htb] 
  %\caption{Conjugations for 5-component string links} \label{ActConj5-compSF}
  \begin{tabular}{|c|l|} \hline
      & $cx_{14}$ \\ \hline 
    $y_{123}$ & $0$\\ \hline 
    $y_{124}$ & $-y_{24}+y_{12}$\\ \hline
    $y_{134}$ & $-y_{34}+y_{13}$\\ \hline
    $y_{125}$ & $0$\\ \hline
    $y_{135}$ & $0$\\ \hline
    $y_{145}$ & $+y_{45}-y_{15}$\\ \hline
    $y_{234}$ & $0$\\
\hline
  \end{tabular} \\
\end{table}

\begin{table}[htb] 
  %\caption{Conjugations for 5-component string links} \label{ActConj5-compSF}
  \begin{tabular}{|c|l|} \hline
      & $cx_{15}$ \\ \hline 
    $y_{123}$ & $0$\\ \hline 
    $y_{124}$ & $0$\\ \hline
    $y_{134}$ & $0$\\ \hline
    $y_{125}$ & $-y_{25}+y_{12}$\\ \hline
    $y_{135}$ & $-y_{35}+y_{13}$\\ \hline
    $y_{145}$ & $-y_{45}+y_{14}$\\ \hline
    $y_{234}$ & $0$\\
\hline
  \end{tabular} \\
\end{table}

\begin{table}[htb] 
  %\caption{Conjugations for 5-component string links} \label{ActConj5-compSF}
  \begin{tabular}{|c|l|} \hline
      & $cx_{23}$ \\ \hline 
    $y_{123}$ & $+y_{13}-y_{12}$\\ \hline 
    $y_{124}$ & $0$\\ \hline
    $y_{134}$ & $0$\\ \hline
    $y_{125}$ & $0$\\ \hline
    $y_{135}$ & $0$\\ \hline
    $y_{145}$ & $0$\\ \hline
    $y_{234}$ & $+y_{34}-y_{24}$\\
\hline
  \end{tabular} \\
\end{table}

\begin{table}[htb] 
  %\caption{Conjugations for 5-component string links} \label{ActConj5-compSF}
  \begin{tabular}{|c|l|} \hline
      & $cx_{24}$ \\ \hline 
    $y_{123}$ & $0$\\ \hline 
    $y_{124}$ & $+y_{14}-y_{12}$\\ \hline
    $y_{134}$ & $0$\\ \hline
    $y_{125}$ & $0$\\ \hline
    $y_{135}$ & $0$\\ \hline
    $y_{145}$ & $0$\\ \hline
    $y_{234}$ & $-y_{34}+y_{23}$\\
\hline
  \end{tabular} \\
\end{table}

\begin{table}[htb] 
  %\caption{Conjugations for 5-component string links} \label{ActConj5-compSF}
  \begin{tabular}{|c|l|} \hline
      & $cx_{25}$ \\ \hline 
    $y_{123}$ & $0$\\ \hline 
    $y_{124}$ & $0$\\ \hline
    $y_{134}$ & $0$\\ \hline
    $y_{125}$ & $+y_{15}-y_{12}$\\ \hline
    $y_{135}$ & $0$\\ \hline
    $y_{145}$ & $0$\\ \hline
    $y_{234}$ & $0$\\
\hline
  \end{tabular} \\
\end{table}

\begin{table}[htb] 
  %\caption{Conjugations for 5-component string links} \label{ActConj5-compSF}
  \begin{tabular}{|c|l|} \hline
      & $cx_{34}$ \\ \hline 
    $y_{123}$ & $0$\\ \hline 
    $y_{124}$ & $0$\\ \hline
    $y_{134}$ & $+y_{14}-y_{13}$\\ \hline
    $y_{125}$ & $0$\\ \hline
    $y_{135}$ & $0$\\ \hline
    $y_{145}$ & $0$\\ \hline
    $y_{234}$ & $+y_{24}-y_{23}$\\
\hline
  \end{tabular} \\
\end{table}

\begin{table}[htb] 
  %\caption{Conjugations for 5-component string links} \label{ActConj5-compSF}
  \begin{tabular}{|c|l|} \hline
      & $cx_{35}$ \\ \hline 
    $y_{123}$ & $0$\\ \hline 
    $y_{124}$ & $0$\\ \hline
    $y_{134}$ & $0$\\ \hline
    $y_{125}$ & $0$\\ \hline
    $y_{135}$ & $+y_{15}-y_{13}$\\ \hline
    $y_{145}$ & $0$\\ \hline
    $y_{234}$ & $0$\\
\hline
  \end{tabular} \\
\end{table}

\begin{table}[htb] 
  %\caption{Conjugations for 5-component string links} \label{ActConj5-compSF}
  \begin{tabular}{|c|l|} \hline
      & $cx_{45}$ \\ \hline 
    $y_{123}$ & $0$\\ \hline 
    $y_{124}$ & $0$\\ \hline
    $y_{134}$ & $0$\\ \hline
    $y_{125}$ & $0$\\ \hline
    $y_{135}$ & $0$\\ \hline
    $y_{145}$ & $+y_{15}-y_{14}$\\ \hline
    $y_{234}$ & $0$\\ \hline
    $y_{235}$ & $0$\\ \hline
    $y_{245}$ & $+y_{25}-y_{24}$\\ \hline
    $y_{345}$ & $+y_{35}-y_{34}$\\ \hline\hline
    $y_{1234}$ & $0$\\ \hline 
    $y_{1324}$ & $0$\\ \hline
    $y_{1235}$ & $0$\\ \hline
    $y_{1245}$ & $+y_{125}-y_{124}-y_{15} y_{24}+y_{14} y_{24}$\\ \hline
    $y_{1325}$ & $0$\\ \hline
    $y_{1345}$ & $+y_{135}-y_{134}-y_{15} y_{34}+y_{14} y_{34}$\\ \hline
    $y_{1425}$ & $-y_{15} y_{25}+y_{14} y_{25}+y_{15} y_{24}-y_{14} y_{24}$\\ \hline
    $y_{1435}$ & $-y_{15} y_{35}+y_{14} y_{35}+y_{15} y_{34}-y_{14} y_{34}$\\ \hline
    $y_{2345}$ & $+y_{235}-y_{234}-y_{25} y_{34}+y_{24} y_{34}$\\ \hline
    $y_{2435}$ & $-y_{25} y_{35}+y_{24} y_{35}+y_{25} y_{34}-y_{24} y_{34}$\\ \hline\hline
    $y_{12345}$ & $+y_{1235}-y_{1234}+y_{35} y_{125}-y_{34} y_{125}-y_{35} y_{123}+y_{34} y_{123}+y_{15} y_{123}-y_{14} y_{123}$\\
      & $+y_{15} y_{24} y_{34}-y_{14} y_{24} y_{34}$\\ \hline 
    $y_{12435}$ & $-y_{35} y_{125}+y_{34} y_{125}+y_{35} y_{124}-y_{34} y_{124}+y_{15} y_{24} y_{35}-y_{14} y_{24} y_{35}-y_{15} y_{24} y_{34}$\\
      & $+y_14 y_{24} y_{34}$\\ \hline
    $y_{13245}$ & $+y_{1325}-y_{1324}+y_{25} y_{135}-y_{24} y_{135}+y_{25} y_{123}-y_{24} y_{123}-y_{15} y_{123}+y_{14} y_{123}$ \\ \hline
    $y_{13425}$ & $-y_{25} y_{135}+y_{24} y_{135}+y_{25} y_{134}-y_{24} y_{134}+y_{15} y_{25} y_{34}-y_{14} y_{25} y_{34}-y_{15} y_{24} y_{34}$\\
      & $+y_{14} y_{24} y_{34}$\\ \hline
    $y_{14235}$ & $+y_{15} y_{25} y_{35}-y_{14} y_{25} y_{35}-y_{15} y_{24} y_{35}+y_{14} y_{24} y_{35}$\\ \hline
    $y_{14325}$ & $-y_{15} y_{25} y_{34}+y_{14} y_{25} y_{34}+y_{15} y_{24} y_{34}-y_{14} y_{24} y_{34}$\\ \hline
  \end{tabular} \\
\end{table}

%%%%%%%%%%%%%%%%%% The results of Conjugation (end) %%%%%%%%%%%%%%%%%%%%%%%%%%%

\clearpage

%%%%%%%%%%%%%%%%

\end{document}